\documentclass[twoside,11pt]{article}
\usepackage{geometry}
\usepackage{amsmath,amsfonts,amssymb,bm,mathrsfs,amsthm,amsxtra,xcolor}

\usepackage{graphicx}
\usepackage{indentfirst}                                
\usepackage[colorlinks]{hyperref}
\hypersetup{linkcolor=blue,filecolor=black,urlcolor=blue, citecolor=black}   
\usepackage[numbers,sort&compress]{natbib}         

\ifxetex
\usepackage{letltxmacro}
\LetLtxMacro\SavedIncludeGraphics\includegraphics
\def\includegraphics#1#{
	\IncludeGraphicsAux{#1}%
}%
\newcommand*{\IncludeGraphicsAux}[2]{%
	\XeTeXLinkBox{%
		\SavedIncludeGraphics#1{#2}%
}}
\fi

\usepackage[titletoc,title]{appendix}
\usepackage{titlesec}
\titlespacing{\section}{0pt}{2.5ex}{1.5ex}
\titlespacing{\subsection}{0pt}{1.5ex}{1ex}
\titlespacing{\subsubsection}{0pt}{1.5ex}{1ex}
\titleformat{\section}{\large\bfseries\centering}{\thesection}{1em}{}
\titleformat{\subsection}[runin]{\bfseries}{\thesubsection.}{0.5em}{}[.\mbox{\ }]
\titleformat{\subsubsection}[runin]{\bfseries}{\thesubsubsection.}{0.4em}{}[.\mbox{\ }]

\numberwithin{equation}{section}
\newtheorem{lemma}{Lemma}[section]
\newtheorem{proposition}[lemma]{Proposition}
\newtheorem{theorem}{Theorem}[section]

\newtheorem{definition}{Definition}[section]
\newtheorem{remark}{Remark}[section]

\usepackage{ulem}

\newcommand{\VERT}{\vert\kern-0.3ex\vert\kern-0.3ex\vert}
\def\d{\,\mathrm{d}}
\def\p{\partial}

\usepackage{accents}
\makeatletter
\def\widebar{\accentset{{\cc@style\underline{\mskip10mu}}}}
\makeatother

\newcommand\w[1]{\makebox[1em]{$#1$}}

\begin{document}
	\title{\bf Stability and existence of relativistic plasma--vacuum interfaces\let\thefootnote\relax\footnotetext{%
The research of {\sc Paolo Secchi} was supported by the Italian MIUR Project PRIN 20204NT8W4-002.
			The research of {\sc Yuri Trakhinin} was carried out within the framework of the state contract of the Sobolev Institute of Mathematics (Project No.~FWNF-2026-0028).
			The research of {\sc Tao Wang} was supported by the National Natural Science Foundation of China under Grants No.~12422109, 12426203, 12371225, and 12221001.
		}
	}

	\author{
		{\sc Paolo Secchi}
		\thanks{DICATAM, Sezione di Matematica, Universit\`a di Brescia, Via Valotti 9, 25133 Brescia, Italy. e-mail: paolo.secchi@unibs.it}	
		\qquad  \ \
		{\sc Yuri Trakhinin}
		\thanks{Sobolev Institute of Mathematics, Koptyug av.~4, 630090 Novosibirsk, Russia. e-mail: trakhin@math.nsc.ru}
		\qquad  \ \
		{\sc Tao Wang}
		\thanks{School of Mathematics and Statistics, Wuhan University, Wuhan 430072, China. e-mail: tao.wang@whu.edu.cn
		}
	}
	
	\vspace{1mm}
	
	\date{\today}
	
	\maketitle

	\vspace{-4mm}
	
	{\small
		\noindent{\bf Abstract}:\quad
We consider the free boundary problem for relativistic plasma--vacuum interfaces in two and three spatial dimensions. The plasma flow is governed by the equations of ideal relativistic magnetohydrodynamics, while the vacuum magnetic and electric fields satisfy Maxwell's equations.
The plasma and vacuum magnetic fields are tangential to the interface, which moves with the plasma flow.
This yields a nonlinear, multidimensional hyperbolic problem with a free boundary that is characteristic of variable multiplicity. We identify a quantitative stability condition and establish the linear stability of three-dimensional relativistic plasma--vacuum interfaces in the sense that the variable-coefficient linearized problem satisfies energy estimates in anisotropic Sobolev spaces.
In estimating tangential derivatives, we exploit an intrinsic cancellation effect to convert the boundary term into an instant integral.
We then separate the estimate involving spatial derivatives from that involving time derivatives,  so that the instant integral can be mainly absorbed by the instant energy under the stability condition.
Moreover, we prove the local-in-time existence and uniqueness of solutions to the nonlinear problem in two-dimensional space, provided that the plasma and vacuum magnetic fields do not vanish simultaneously at any point of the initial interface.
The proof combines the solvability and tame estimate of the linearized problem with a suitable modified Nash--Moser iteration.
In particular, to establish its solvability, the two-dimensional linearized problem is reduced to a transport equation for the interface function and a hyperbolic boundary problem with maximally nonnegative boundary conditions.

		\vspace{2mm}

		\vspace{1mm}
		
		
		\vspace{2mm}
		{\footnotesize \tableofcontents}
	}

	\vspace{2mm}

	\section{Introduction}\label{sec:intro}

We study the stability and existence of an interface that separates an ideal (i.e., neglecting viscosity and electrical resistivity) relativistically flowing plasma from a vacuum. In the plasma region $\Omega^+(t) \subset\mathbb{R}^d$, we consider the following equations of ideal relativistic magnetohydrodynamics (RMHD) in Minkowski spacetime $\mathbb{R}^{1+d}$ for the spatial dimension $d=2,3$:
	\begin{subequations} \label{RMHD0}
		\begin{alignat}{2}
			\label{RMHD0a}
			&\p_t(\rho \varGamma)+\nabla\cdot (\rho \varGamma v)=0,\\[0.5mm]
			\label{RMHD0b}	
			&\p_t\big(\rho \mathfrak{f} \varGamma^2+\epsilon^2|H|^2-\epsilon^2q\big)
			+\nabla \cdot \big( \rho \mathfrak{f} \varGamma^2 v+\epsilon^2|H|^2 v -\epsilon^2( v\cdot H)H \big)=0,\\[0.5mm]
			\nonumber
			&\p_t\big( \rho \mathfrak{f} \varGamma^2 v+\epsilon^2|H|^2 v -\epsilon^2( v\cdot H)H \big)-\nabla\cdot \big(\varGamma^{-2}H\otimes H\big)+\nabla q\\[0.5mm]
			\label{RMHD0c}
			&\hspace*{1.85em}+\nabla\cdot \big( (\rho \mathfrak{f} \varGamma^2+\epsilon^2 |H|^2)v\otimes v-\epsilon^2 (v\cdot H)(H\otimes v+v\otimes H) \big)=0,\\[0.5mm]
			\label{RMHD0d}
			&\p_t H-\nabla\times (v\times H)=0,
		\end{alignat}
	\end{subequations}
	together with the divergence constraint
	\begin{align} \label{H.inv1}
		\nabla\cdot H=0.
	\end{align}
	Here the particle number density $\rho$,
	the coordinate velocity $v=(v_1,\ldots,v_d)$,
	the magnetic field $H=(H_1,\ldots,H_d)$,
	and the total pressure $q$ are unknown functions of time $t$ and spatial variables $x=(x_1,\ldots,x_d)$.
	The parameter $\epsilon^{-1}$ represents the speed of light
	and the velocity $v$ satisfies the physical condition $|v|<\epsilon^{-1}$.
	The Lorentz factor $\varGamma$, the index of the fluid $\mathfrak{f}$, and the total pressure $q$ are of the form
	\begin{align}  \label{Gamma:def}
		\varGamma=\frac{1}{\sqrt{1-\epsilon^2|v|^2}},\quad
		\mathfrak{f}=1+\epsilon^{2}\left(\mathfrak{e}+\frac{p}{\rho}\right),\quad
		q=p+\frac{|H|^2}{2\varGamma^{2}}+\frac{\epsilon^2}{2} (v\cdot H)^2,
	\end{align}
	where $\mathfrak{e}$ is the internal energy and $p$ is the pressure.
	The equations \eqref{RMHD0} constitute a closed system of conservation laws for the unknowns $U:=(q,v,H,S)\in \mathbb{R}^{2d+2}  $ through smooth constitutive relations $\rho=\rho(p, S)$ and $\mathfrak{e}=\mathfrak{e}(p, S)$, which satisfy the Gibbs relation
	\begin{align} \label{Gibbs}
		\vartheta \d S=\d \mathfrak{e}-\frac{p}{\rho^2} \d \rho,
	\end{align}
	where $S$ is the entropy and $\vartheta>0$ is the absolute temperature.
See Section \ref{sec:nontation} for  the conventional notation in the vector calculus
and Appendix \ref{app:A} for a derivation of the equations \eqref{RMHD0}--\eqref{H.inv1}.

We consider for plasmas very general constitutive relations satisfying the physical assumption that the relativistic sound speed $c_{\rm s}=c_{\rm s}(\rho, S)$ is positive and smaller than the speed of light, i.e.,
\begin{align} \label{cs:def}
	c_{\rm s}(\rho, S):=\sqrt{\frac{1}{\mathfrak{f}}\frac{\p p}{\p \rho} (\rho, S)}\in (0,\epsilon^{-1})
	\qquad \textrm{for all }  \rho \in (\rho_*,\rho^*) ,	
\end{align}
where $\rho_*$, $\rho^*$ are some nonnegative constants.
Then the RMHD equations \eqref{RMHD0} can become symmetric hyperbolic for smooth solutions satisfying $\rho_*<\rho<\rho^*$, as established by {\sc Ruggeri--Strumia} \cite{RS81MR0628566} and {\sc Anile--Pennisi} \cite{A90zbMATH05040442, AP87MR0877994} via the Godunov's procedure.
Furthermore, a concrete form of symmetric matrices was derived by {\sc Freist\"{u}hler--Trakhinin} \cite{FT13MR3044369} for the study of relativistic current-vortex sheets using the Lorentz transform.
For the detailed symmetrization for RMHD, we refer the reader to Appendix \ref{app:A}.
	
In the vacuum region $\Omega^-(t)\subset\mathbb{R}^d$, corresponding to the dynamics of plasmas,
the evolution of the magnetic field $h\in \mathbb{R}^d$ and the electric field $e\in \mathbb{R}^{2d-3}$ is governed by Maxwell's equations in vacuum (see {\sc Jackson} \cite[\S 1.1]{J75MR0436782})
	\begin{align} \label{Maxwell1}
 \epsilon \p_t h+\nabla\times e=0,\qquad
 \epsilon \p_t e-\nabla\times h=0,
	\end{align}
together with  the divergence constraints
\begin{align}
	\label{h.inv1}
	\nabla\cdot h=0,
\end{align}	
and
\begin{alignat}{3}
	\label{e.inv1}
	&\nabla\cdot e=0\qquad &&\textrm{if }d=3.
\end{alignat}	
The two-dimensional Maxwell's equations follow under the assumptions that no unknown depends on $x_3$, and that the third component of the magnetic field and the first two components of the electric field are identically zero. 

The plasma--vacuum interface $\Sigma(t)=\overline{\Omega^+(t)}\cap \overline{\Omega^-(t)}
$ is assumed to move with the plasma flow. Let the plasma and vacuum magnetic fields remain tangential to $\Sigma(t)$. Then the conservation laws \eqref{RMHD0}--\eqref{H.inv1} and \eqref{Maxwell1}--\eqref{e.inv1}  yield the boundary conditions
	\begin{gather}
			\label{BC1}
	H\cdot \bm{n}=0,\ \   h\cdot \bm{n}=0, \ \
				\mathcal{V}=v\cdot \bm{n},\ \
 2q=  |h|^2-|e|^2,\ \
 \epsilon \mathcal{V} h=\bm{n}\times e
 \quad \textrm{on }\Sigma(t),
		\end{gather}	
where we denote by $\bm{n}$ the unit normal pointing into $\Omega^+(t)$ and by $\mathcal{V}$ the normal speed of $\Sigma(t)$. 		
The last two conditions in \eqref{BC1} come from the balance of normal stresses at $\Sigma(t)$. A derivation of \eqref{BC1} is provided in Appendix \ref{app:B}. The system \eqref{RMHD0}--\eqref{BC1} is supplemented with the initial conditions
\begin{align} \label{IC0}
	\Omega^+(0)=\Omega_0^+,\ \
		\Omega^-(0)=\Omega_0^-,\ \
	U|_{t=0}=U_0\   \textrm{in }\Omega_0^+,\ \
	(h,e)|_{t=0}=(h_0,e_0)\   \textrm{in }\Omega_0^-,
\end{align}
where the domains $\Omega_0^{\pm}\subset \mathbb{R}^d$ and the data $(U_0,h_0,e_0)\in \mathbb{R}^{5d-1}$ are prescribed and satisfy the constraints \eqref{H.inv1}, \eqref{h.inv1}--\eqref{e.inv1}, and $|v_0|<\epsilon^{-1}$.

Plasma--vacuum interface problems in non-relativistic and relativistic magnetohydrodynamics (MHD) arise in the mathematical modeling of plasma confinement by magnetic fields \cite{GKP19,G84,BFKK58MR0091737}. The non-relativistic MHD model is adequate for laboratory and astrophysical plasmas with non-relativistic velocities, whereas the RMHD model, which accounts for relativistic effects, is appropriate for relativistically flowing plasmas \cite{A90zbMATH05040442,Lichnerowicz67,GKP19}.

For non-relativistic plasma--vacuum interface problems, the displacement current $\epsilon\partial_te$ is neglected in Maxwell's equations \eqref{Maxwell1}, so that the vacuum magnetic field $h$ satisfies the (elliptic) div-curl system and the vacuum electric field $e$ becomes a secondary variable. In ideal compressible MHD, this leads to a nonlinear hyperbolic--elliptic coupled problem with a free boundary that is {\it characteristic of constant multiplicity} (i.e., the boundary matrix is singular and has constant rank on the boundary; cf.~\cite{R85MR0797053,OSY95MR1346224,S95MMASMR1346663,S96aMR1401431}).

Local well-posedness for non-relativistic plasma--vacuum interface problems has been studied extensively in the recent decades.
For linear stability of 3D plasma--vacuum interfaces in ideal compressible MHD,
the second author \cite{T10MR2718711} proposed two different conditions:
the non-collinearity condition
\begin{align} \label{non-col1}
	|H\times h|\geq \delta>0\qquad \textnormal{on }\Sigma(t)
\end{align}
and  the Taylor-type sign condition
\begin{align} \label{Taylor}
	 \bm{n}\cdot \nabla (q-\tfrac{1}{2}|h|^2)\leq -\delta<0\qquad \textnormal{on }\Sigma(t),
\end{align}
for the basic state around which the problem was linearized,
where $\delta$ is some positive constant.
Basic {\it a priori} estimates were derived in \cite{T10MR2718711} for the {\it variable coefficient} linearized problem under the non-collinearity condition \eqref{non-col1} and for the {\it constant coefficient} linearized problem under the Taylor-type sign condition \eqref{Taylor}.
Subsequently, the first two authors \cite{ST14MR3151094,ST13MR3148595} proved the linear and nonlinear existence of solutions for the 3D plasma--vacuum interface problem under the non-collinearity condition \eqref{non-col1}.  Also see \cite{STW23MR4654316} for another proof of the linear existence.
In the case of zero vacuum magnetic field ($h\equiv 0$),
under the Taylor-type sign condition \eqref{Taylor},
the second and third authors \cite{TW21MR4201624} established
the nonlinear existence in two and three spatial dimensions,
and {\sc Lindblad--Zhang} \cite{LZ23MR4632835} proved
the {\it a priori} estimates without loss of anisotropic regularity.
Recently, {\sc Morando et al.}~\cite{MSTTY24MR4754746} constructed the local solutions to the 2D nonlinear problem, provided that the plasma and vacuum magnetic fields do not vanish simultaneously at any point of the initial interface, that is,
\begin{align} \label{non-vanishing}
	|H|+|h|\geq \delta>0\qquad \textnormal{on }\Sigma(t)
\end{align}
holds at the initial time.
As for the ideal incompressible case, we refer to \cite{MTT14MR3237563,HL21MR4295166} for linear well-posedness, \cite{H17MR3614754,HL14MR3187678,SWZ19MR3981394,GW19MR3980843,LL26MR4976713,LX25MR4861065,Z24MR4668328} for nonlinear well-posedness,
and \cite{HL20MR4093862} for nonlinear ill-posedness.

On the other hand, for well-posedness of the relativistic plasma--vacuum interface problem \eqref{RMHD0}--\eqref{IC0},
the only result that we know is \cite[Theorem 4.2]{T12MR2974767}, where a basic energy estimate was deduced for the 3D linearized problem with variable coefficients satisfying the non-collinearity condition \eqref{non-col1} and the assumption that $|e_1+\epsilon v_2h_3-\epsilon v_3h_2|$ is sufficiently small.
A relevant model is the MHD--Maxwell interface problem, which consists of the non-relativistic MHD equations for plasma flow, Maxwell's equations for vacuum fields, and suitable boundary conditions on the interface.
For this model, basic energy estimates for the 3D and 2D linearized problem were established in \cite{MT14MR3194371,CDS14MR3248397,MSTT20MR4124205} and \cite{CDS16MR3527627,T24MR4793931}, respectively.

{In this paper, we identify a quantitative stability condition and establish the linear stability of 3D relativistic plasma--vacuum interfaces in the sense that the variable-coefficient linearized problem satisfies appropriate energy estimates (cf.~\eqref{stability3D} and Theorem \ref{thm:main1}).}
Our stability condition \eqref{stability3D} specifies an explicit domain for the linear stability of 3D relativistic plasma--vacuum interfaces, which substantially refines the smallness assumption in \cite[Theorem 4.2]{T12MR2974767}.
Moreover, we prove the local-in-time existence and uniqueness of solutions to the nonlinear problem \eqref{RMHD0}--\eqref{IC0} in 2D, provided that the condition \eqref{non-vanishing} is satisfied by the initial data.
To the best of our knowledge, this is the first rigorous result on the existence of relativistic plasma--vacuum interfaces.

Unlike its non-relativistic counterpart, the relativistic plasma--vacuum interface problem \eqref{RMHD0}--\eqref{IC0} is a nonlinear hyperbolic problem, and the moving interface $\Sigma(t)$ is {\it characteristic of variable multiplicity}, that is, the boundary matrix is singular with non-constant rank on $\Sigma(t)$.
According to the well-posedness theory of hyperbolic boundary problems \cite{BGS07MR2284507,CP82MR0678605}, a different number of boundary conditions should be imposed in different parts of the interface, depending on the sign of the normal velocity of the interface (plasma expansion into vacuum or vice versa); see \cite{T12MR2974767, T24MR4793931,CDS14MR3248397} for a thorough discussion of these issues.
To overcome this difficulty, we reformulate \eqref{RMHD0}--\eqref{IC0} to a symmetrizable hyperbolic problem with {\it characteristic boundary of constant multiplicity} by means of the vacuum divergence equations \eqref{h.inv1} and \eqref{e.inv1}. We shall prove that \eqref{h.inv1} and \eqref{e.inv1}  hold automatically for $t>0$ if they are satisfied at the initial time.

For technical simplicity, we suppose that the interface $\Sigma(t)$ has the form of a graph, which allows us to transform the nonlinear problem into a fixed domain by a simple lift of the graph. More general interfaces can be handled by standard but technically involved arguments as in \cite[Remark 2.3]{T09CPAMMR2560044}.
In general, for hyperbolic problems with a characteristic boundary, there is a loss of normal derivatives of the characteristic variables in higher-order energy estimates \cite{S96MR1405665,C07MR2289911}.
As in \cite{
	CW08MR2372810, ST13MR3148595,
	LZ23MR4632835,
	T09ARMAMR2481071,YM91MR1092572,TW21MR4201624,TW22aMR4444136,TW22cMR4506578,MSTTY24MR4754746,ST14MR3151094,STW23MR4654316,MSTY23MR4586863} for various characteristic boundary problems in ideal compressible MHD, where the full regularity cannot be expected,
we shall work for plasmas in anisotropic Sobolev spaces $H_*^m$, first introduced by {\sc Chen} \cite{C07MR2289911}, where two tangential derivatives count as one normal derivative.

To study well-posedness of the variable coefficient linearized problem around the basic state $(\mathring{U},\mathring{h},\mathring{e},\mathring{\varphi})$, we employ Alinhac's good unknowns \cite{A89MR976971} to obtain the effective linear problem \eqref{ELP1}.
For problem \eqref{ELP1}, we construct certain auxiliary functions to perform a partial homogenization, so that the solution satisfies several important homogeneous constraints (cf.~Lemma \ref{lem:homo}).
Then we introduce suitable change of unknowns to separate the non-characteristic variables from the system \eqref{ELP4}.

For the 3D problem, the term $\mathring{e}_1+\epsilon\mathring{v}_2\mathring{h}_3-\epsilon\mathring{v}_3\mathring{h}_2$ does not vanish generally, and hence the boundary conditions involve the traces of the characteristic variables (specifically, the trace of the normal electric field $w_4$), 
which is completely different from the previous works such as \cite{CW08MR2372810,CS08MR2423311,FT13MR3044369,LZ23MR4632835,MSTY23MR4586863,MSTTY24MR4754746,ST13MR3148595,ST14MR3151094,STW23MR4654316,T09ARMAMR2481071,T09CPAMMR2560044,T10MR2718711,TW21MR4201624,TW22aMR4444136,TW22cMR4506578,YM91MR1092572}.
New ideas are required to control the problematic boundary term $\mathcal{Q}_{\alpha}$ arising in the estimate of tangential derivatives due to the complex nature of the boundary conditions.
To overcome this difficulty, we show an intrinsic cancellation effect by means of the boundary conditions and
the restriction of the interior equations on the boundary.
This cancellation reduces the boundary term $\mathcal{Q}_{\alpha}$ to
the instant boundary integral $\mathcal{R}_{\alpha}$ ({\it cf.}~\eqref{Q:es1}--\eqref{R.cal:def}).

To identify a quantitative condition for 3D linear stability, we separate the estimate of $\mathcal{Q}_{\alpha}$ involving spatial derivatives from that involving time derivatives.
More precisely, we first consider the case with spatial derivatives.
Under the non-collinearity condition \eqref{non-coll},
from the homogeneous constraints \eqref{H.cons3} and \eqref{h.cons3}, we can transform the spatial derivatives of the interface function $\psi$ into the traces of the normal magnetic fields $W_5$ and $w_1$ (cf.~\eqref{psi:id1.3D}).
Then we apply a localization method to convert the instant boundary integral $\mathcal{T}$ defined by \eqref{T.cal:def} into volume integrals involving the normal derivatives $( \p_1 W_5, \p_1 w_1, \p_1  w_4)$.
Using the homogeneous constraints \eqref{H.inv3}--\eqref{e.inv3} and an explicit lower bound for the instant energy, we can identify the stability condition \eqref{stability3D} on the basic state, under which the volume integrals can be mainly absorbed by the instant energy.
Then we are led to deal with the estimate of $\mathcal{R}_{\alpha}$ involving time derivatives.
To this end, we observe from boundary conditions \eqref{ELP4}, identities \eqref{psi:id1.3D}, and cancellation property \eqref{cancellation} that the time derivatives $\p_t \psi$ and $\p_t w_4$ can be transformed into the traces of $(W_5, w_1, w_4)$ and $\p_3\mu_2-\p_2\mu_3$, respectively.
As a consequence, we can apply the argument of localization and passing to the volume integral to
control the boundary integral $\mathcal{R}_{\alpha}$ by the instant energy without imposing any additional assumption.
With these estimates and the estimate for the front $\psi$ in hand,
we can deduce the desired higher-order estimate \eqref{tame.es} for the effective linear problem \eqref{ELP1} and therefore establish the quantitative linear stability of 3D relativistic plasma--vacuum interfaces.

{By contrast},
for the 2D linearized problem, 
the boundary conditions \eqref{ELP4c} depend on the traces of $(W,w)$ solely through the non-characteristic variables $\mathcal{W}_{\rm nc}$. 
By eliminating the first derivatives of  the interface function $\psi$ from the third boundary condition in \eqref{ELP4c}, we reformulate the boundary conditions \eqref{ELP4c} to a transport equation for $\psi$ together with two boundary conditions that involve only $(\mathcal{W}_{\rm nc},\psi)$.
Consequently, problem \eqref{ELP4} decouples into a transport equation for $\psi$ and a hyperbolic boundary problem for $(W,w)$ whose boundary source terms depend on $\psi$.
Since the problem for $(W,w)$ is characteristic of constant multiplicity and the corresponding boundary conditions are maximally nonnegative for $\psi=0$, we can establish the existence of solutions to the reformulated problem by employing the fixed point argument developed in \cite{ST13MR3148595}.
However, the presence of characteristic variables in the 3D boundary conditions makes the above approach inapplicable to the 3D case, and the existence of 3D relativistic plasma--vacuum interfaces remains an interesting open problem for future study.

Remark that the estimate \eqref{tame.es} for the effective linear problem \eqref{ELP1} is a so-called tame estimate, since the loss of regularity from the basic state and source terms to the solution is fixed.
Based on the solvability and tame estimate of the 2D linearized problem, we establish the local-in-time existence of solutions to the 2D nonlinear problem via a suitable iteration scheme of Nash--Moser type, as developed by {\sc H\"{o}rmander} \cite{H76MR0602181} and {\sc Coulombel--Secchi} \cite{CS08MR2423311}.
In particular, a smooth intermediate state is constructed and estimated (cf.~Proposition \ref{pro:modified}), so that the state around which we linearize at each iteration step can satisfy the necessary constraints for linear solvability.

The rest of the paper is organized as follows.
In Section \ref{sec:main}, we transform the relativistic plasma--vacuum interface problem \eqref{RMHD0}--\eqref{IC0} into a fixed domain whose boundary is characteristic of constant multiplicity.
We then state our main results on linear stability for $d=2,3$ (Theorems \ref{thm:main1} and \ref{thm:main2}) and on nonlinear existence for the spatial dimension $d=2$ (Theorem \ref{thm:main3}).
We also collect the notation used throughout the paper.
In Section \ref{sec:refor}, we perform a partial homogenization for the effective linear problem \eqref{ELP1} and separate the non-characteristic variables from the system.
Section \ref{sec:lin3D} is devoted to the proof of Theorem \ref{thm:main1}, i.e., the linear stability of 3D relativistic plasma--vacuum interfaces.
To this end, we begin in \S\ref{sec:prelude} with a standard anisotropic energy estimate, introducing the instant energy $\mathcal{E}_{\alpha}$ and the boundary term $\mathcal{Q}_{\alpha}$.
To handle tangential derivatives, we deduce in \S\ref{sec:cancel} an intrinsic cancellation that reduces $\mathcal{Q}_{\alpha}$ to the instant boundary integral $\mathcal{R}_{\alpha}$.
In \S\S \ref{sec:Qes1}--\ref{sec:Qes2}, we identify a quantitative stability condition under which $\mathcal{R}_{\alpha}$ can be mainly absorbed by $\mathcal{E}_{\alpha}$.
Estimate for the front $\psi$ and proof of Theorem \ref{thm:main1} are provided in \S \ref{sec.psi} and \S \ref{sec.Thm2.2}, respectively.
In Section \ref{sec:2D}, we first apply a fixed point argument to prove Theorem \ref{thm:main2}, that is, the solvability and higher-order tame estimates for the effective linear problem \eqref{ELP1} in 2D.
We then construct solutions to the nonlinear problem \eqref{NP1} in 2D and establish Theorem \ref{thm:main3} by employing a modified Nash--Moser iteration to overcome the loss of regularity in the linear well-posedness result.
For completeness, the RMHD system and its symmetrization are presented in Appendix \ref{app:A} and the boundary conditions \eqref{BC1} are deduced in Appendix \ref{app:B}.
To keep the presentation clear, we prove Proposition \ref{Pro:invol} and the solvability of problem \eqref{u.natural:eq} in Appendices \ref{app:C} and \ref{app:D}.

\vspace*{2mm}

	\section{Preliminaries and main results}\label{sec:main}
	
	In this section, we first transform the relativistic plasma--vacuum interface problem {into a fixed domain whose boundary is characteristic of constant multiplicity}.
	Then we state the linear stability theorems for $d=2,3$ and the nonlinear existence theorem for spatial dimension $d=2$. We also collect the notation used throughout the paper.
	
	\subsection{Reduction to characteristic boundary of constant multiplicity}
We assume that the plasma--vacuum interface $\Sigma(t)$ has the form of a graph:
\begin{align*}
	\Sigma(t):=\{x\in\mathbb{R}^d:x_1=\varphi(t,x')\} \qquad \textrm{with \ } x'=(x_2,\ldots,x_d),
\end{align*}
where the interface function $\varphi$ is to be determined.
Denote
\begin{align}
	N:=(1,-\p_2\varphi,\ldots,-\p_d\varphi).
	\label{N:def}
\end{align}
Then the boundary conditions \eqref{BC1} in 3D ($d=3$) become	
\begin{alignat}{3}
\label{BC2a}	
	&	H\cdot N=0,\quad  && h\cdot N=0\qquad &&
	\textrm{on }\Sigma(t),\\
\label{BC2b}		
&\p_t \varphi=v\cdot N,\quad &&2q=  |h|^2-|e|^2\qquad &&
\textrm{on }\Sigma(t),\\
	&
	e_2+\p_2 \varphi e_1=\epsilon \p_t \varphi h_3,\quad &&
e_3+\p_3 \varphi e_1=-\epsilon \p_t \varphi h_2
\qquad &&
\textrm{on }\Sigma(t).
\label{BC2c}	
\end{alignat}
On the other hand, for the 2D 
problem, the vacuum electric field $e$ is a scalar-valued function and the boundary conditions \eqref{BC1} are reduced to \eqref{BC2a}, \eqref{BC2b}, and
\begin{alignat}{3}
 e=-\epsilon \p_t \varphi h_2\qquad \textrm{on }\Sigma(t) \qquad \textrm{if }d=2.
\label{BC2d}		
\end{alignat}

It follows from the constraint \eqref{H.inv1}, the Gibbs relation \eqref{Gibbs}, and the physical condition \eqref{cs:def} that the RMHD equations \eqref{RMHD0} can be reformulated into the symmetrizable hyperbolic system
	\begin{align}
		\label{RMHD:vec}
		\p_t U+
		A_j^+(U)\p_j U=0\qquad \textrm{in } \Omega^+(t) :=\{x\in\mathbb{R}^d: x_1> \varphi(t,x') \}
	\end{align}
	for smooth 
	solutions $U:=(q,v,H,S)\in\mathbb{R}^{2d+2}$ with $\rho=\rho(p,S) \in (\rho_*,\rho^*)$.
	For clear presentation, we put the sophisticated expressions of coefficients $A_j^+(U)$ and symmetrizer $S^+(U)$ in Appendix \ref{app:A} (cf.~\eqref{B.def}--\eqref{S+U}).
Here and below, we adopt the Einstein summation convention, for which Greek and Latin indices range from $0$ to $d$ and from $1$ to $d$, respectively.	
	
	The Maxwell's equations \eqref{Maxwell1} for the vacuum fields $u:=(h,e)\in \mathbb{R}^{3d-3}$ can be written as the symmetric hyperbolic system
	\begin{align}\label{vacuum1}
		\epsilon \p_t u+
		 B_j^-\p_j u=0\qquad\textrm{in }\Omega^-(t):=\{x\in\mathbb{R}^d: x_1< \varphi(t,x') \},
	\end{align}
	where
\begin{gather}
\nonumber 
	B_1^-:=\begin{bmatrix}
		&  & &0 & 0 &0 \\
		&   & & 0 & 0 &-1 \\
		&   & & 0 & 1 &0 \\
		0 & 0 &0 & &  &  \\
		0 & 0 &1 &  &  &  \\
		0 &-1 &0& &  &
	\end{bmatrix},
	\
	B_2^-:=\begin{bmatrix}
		&  & &0 & 0 &1 \\
		&   & & 0 & 0 &0 \\
		&   & & -1 & 0 &0 \\
		0&0& -1& &   & \\
		0 & 0 &0 &  &  & \\
		1 & 0 &0 &  &    &
	\end{bmatrix},
	\
	B_3^-:=\begin{bmatrix}
		&  & & 0& -1 &0\\
		&  && 1& 0 &0\\
		&  && 0 & 0 &0 \\
		0 & 1 &0 & &  & \\
		-1 & 0 &0 & &  & \\
		0 & 0 &0&  &  &
	\end{bmatrix} 
\end{gather}
if $d=3$, while
\begin{align}
\nonumber 
	B_1^-:=\begin{bmatrix}
		0 & 0 &0 \\
		0 & 0 &-1 \\
		\w{0} & -1 &0
	\end{bmatrix},
	\quad
	B_2^-:=\begin{bmatrix}
		0 & 0 &1 \\
		0 & 0 &0 \\
		\w{1} & \w{0} &\w{0}
	\end{bmatrix}\qquad \textrm{if }d=2.
\end{align}	
The plasma--vacuum system \eqref{BC2a}--\eqref{vacuum1} is supplemented with the initial data
\begin{align} \label{IC1}
		\varphi|_{t=0}=\varphi_0\ \ \textrm{on }\mathbb{R}^{d-1},\quad
	U|_{t=0}=U_0\ \ \textrm{in }\Omega^+(0),\quad
	u|_{t=0}=u_0\  \ \textrm{in }\Omega^-(0).
\end{align}

	To study the relativistic plasma--vacuum interface problem \eqref{BC2a}--\eqref{IC1}, we introduce its boundary matrix
	$A_{\rm b}:=\mathrm{diag}\big(A_{\rm b}^+,\, A_{\rm b}^-\big)$ with
	\begin{align}
	\label{Ab:def}	
		A_{\rm b}^+:=\p_t\varphi I_{2d+2}-N_jA_j^+(U),\quad
		A_{\rm b}^-:=-\epsilon\p_t\varphi I_{3d-3}+N_jB_j^-,
	\end{align}
	where $I_m$ stands for  the identity matrix of order $m$
	and $N=(N_1,\ldots,N_d)$ is defined by \eqref{N:def}.
	In view of \eqref{S+U} and \eqref{iden:App1},
	we find that the matrix $A_{\rm b}^+$ has $1$ positive, $1$ negative, and $2d$ zero eigenvalues on the interface $\Sigma(t)$.
	Hence, the boundary matrix $A_{\rm b}$ is singular on $\Sigma(t)$, that is, the boundary $\Sigma(t)$ is characteristic.
	On the other hand, a direct calculation shows that the matrix $A_{\rm b}^-$ has eigenvalues
\begin{alignat}{5} \nonumber 
&\lambda_{1 }=\lambda_{2}=-\epsilon\p_t\varphi-|N| , \ \ &&\lambda_{3 }=\lambda_{4}=-\epsilon\p_t\varphi,\ \ && \lambda_{5}=\lambda_{6}=-\epsilon\p_t\varphi+|N| \qquad   &&\textrm{if }d=3,\\[1mm]
&\lambda_{1}=-\epsilon\p_t\varphi-|N| ,\ \ &&
\lambda_{2}=-\epsilon\p_t\varphi, \ \ &&
  \lambda_{3}=-\epsilon\p_t\varphi+|N|\qquad  &&\textrm{if }d=2.
		\nonumber
\end{alignat}		
{It follows from the first condition in \eqref{BC2b} and the physical assumption $|v|<\epsilon^{-1}$ that $\lambda_{1,2}<0$ and $\lambda_{5,6}>0$ if $d=3$, while $\lambda_{1}<0$ and $\lambda_{3}>0$ if $d=2$.}
Since the interface $\Sigma(t)$ is a free boundary,  the eigenvalue $-\epsilon\p_t\varphi$ has no definite sign.
Consequently,
{the relativistic plasma--vacuum interface problem \eqref{BC2a}--\eqref{IC1} is {\it characteristic of variable multiplicity}, meaning that the boundary matrix is singular and its rank is not constant on the boundary.}
Furthermore, according to the well-posedness theory of hyperbolic boundary problems, a different number of boundary conditions should be imposed in different parts of the interface, depending on the sign of the normal velocity of the interface (plasma expansion into vacuum or vice versa)
.

	To overcome the above difficulty, we observe that the equations \eqref{h.inv1}--\eqref{e.inv1} and the first condition in \eqref{BC2b} can enable us to reformulate \eqref{BC2a}--\eqref{IC1} into {a problem
	 with boundary being {\it characteristic of constant multiplicity} (that is, the boundary matrix has a constant rank on the boundary; cf.~\cite{R85MR0797053,S95MMASMR1346663,S96aMR1401431,OSY95MR1346224}).}
Precisely,
we set $\nu$ to be an extension of the plasma velocity $v$ to the vacuum domain $\Omega^-(t)$, so that
\begin{align} \label{nu:def0}
	\nu(t,\cdot)=v(t,\cdot)\ \ \textrm{on }\Sigma(t),\quad
	|\nu(t,\cdot)|<\epsilon^{-1}\ \ \textrm{in }\Omega^-(t).
\end{align}
By virtue of \eqref{h.inv1} and \eqref{e.inv1},
we can reduce the Maxwell's equations \eqref{Maxwell1} to
\begin{alignat*}{3}
		&	\epsilon \p_t  h+\nabla \times e+\epsilon (\nabla \cdot h) \nu =0, \ \
		\epsilon \p_t  e-\nabla \times h+\epsilon  (\nabla \cdot e )\nu =0\qquad  &&\textrm{if }d=3,\\[0.5mm]
		&	\epsilon \p_t  h+\nabla \times e+\epsilon (\nabla \cdot h )\nu=0, \ \
		\epsilon \p_t e-\nabla \times h =0 &&\textrm{if }d=2,
\end{alignat*}
or equivalently,
	\begin{align}\label{vacuum2}
		\epsilon \p_t u+
		A_j^-(\nu)\p_j u=0\qquad\textrm{in }\Omega^-(t),
	\end{align}
	where  
	\begin{alignat}{5}
		\nonumber
		&A_1^-(\nu):=\begin{bmatrix}
			\epsilon \nu_1 & 0 &0 &0 & 0 &0 \\
			\epsilon \nu_2 & 0 &0& 0 & 0 &-1 \\
			\epsilon \nu_3 & 0 &0& 0 & 1 &0 \\
			0 & 0 &0 &\epsilon \nu_1 & 0 &0 \\
			0 & 0 &1 &\epsilon \nu_2 & 0 &0 \\
			\w{0} & \w{-1} &\w{0}&\w{\epsilon \nu_3} & \w{0} &\w{0}
		\end{bmatrix},
		\quad
&&		A_2^-(\nu):=\begin{bmatrix}
			0 & \epsilon \nu_1 &0 &0 & 0 &1 \\
			0 & \epsilon \nu_2 &0& 0 & 0 &0 \\
			0 & \epsilon \nu_3 &0& -1 & 0 &0 \\
			\w{0} &\w{0} & \w{-1} &\w{0} & \w{\epsilon \nu_1} &\w{0}\\
			0 & 0 &0 &0 & \epsilon \nu_2 &0 \\
			1 & 0 &0 &0 & \epsilon \nu_3 &0
		\end{bmatrix},
		\\[1mm]
		&A_3^-(\nu):=\begin{bmatrix}
			\w{0} &\w{0} &\w{\epsilon \nu_1} & \w{0} & \w{-1} &\w{0}\\
			0 & 0 &\epsilon \nu_2& 1& 0 &0\\
			0 & 0 &\epsilon \nu_3& 0 & 0 &0 \\
			0 & 1 &0 &0 & 0 &\epsilon \nu_1 \\
			-1 & 0 &0 &0 & 0 &\epsilon \nu_2 \\
			0 & 0 &0& 0 & 0 &\epsilon \nu_3
		\end{bmatrix}&&\hspace*{11.8em} \textrm{if }d=3,
		\label{A-:def.3D}	\\[2mm]
	\label{A-:def.2D}
&	A_1^-(\nu):=\begin{bmatrix}
		\epsilon \nu_1 & 0 &0 \\
		\epsilon \nu_2 & 0 &-1 \\
		\w{0} & -1 &0
	\end{bmatrix},
	\quad
&&	A_2^-(\nu):=\begin{bmatrix}
		0 & \epsilon \nu_1  &1 \\
		0 & \epsilon \nu_2 &0 \\
		\w{1} & \w{0} &\w{0}
	\end{bmatrix} \qquad \textrm{if }d=2.
\end{alignat}	
On the interface $\Sigma(t)$,  the boundary matrix
for the reformulated vacuum equations \eqref{vacuum2}
is $-\epsilon\p_t\varphi I_{3d-3}+N_jA_j^-(\nu)$,
which has eigenvalues
\begin{align*}
\left\{
\begin{aligned}
	&\lambda_{1,2}=-\epsilon\p_t\varphi-|N|<0, \\
	& \lambda_{3,4}=-\epsilon(\p_t\varphi-\nu\cdot N)=0,\\
	& \lambda_{5,6}=-\epsilon\p_t\varphi+|N|>0,
\end{aligned}
\right.\ \  \textrm{if }d=3,\quad
\left\{
\begin{aligned}
	& \lambda_{1}=-\epsilon\p_t\varphi-|N|<0,\\
	& \lambda_{2}=-\epsilon(\p_t\varphi-\nu \cdot N)=0, \\
	& \lambda_{3}=-\epsilon\p_t\varphi+|N|>0,
\end{aligned}
\right.\ \ \textrm{if }d=2,
\end{align*}
thanks to	\eqref{nu:def0} and \eqref{BC2b}.
	As a result, the problem \eqref{BC2a}--\eqref{RMHD:vec}, \eqref{IC1}, \eqref{vacuum2} is {\it characteristic of constant multiplicity} on $\Sigma(t)$.

The boundary matrix for the nonlinear
problem \eqref{BC2a}--\eqref{RMHD:vec}, \eqref{IC1}, \eqref{vacuum2}
 has $d$ positive, $d$ negative, and $3d-1$ zero eigenvalues on $\Sigma(t)$.	
Since one boundary condition is needed to determine the front function $\varphi$,
the correct number of boundary conditions should be $d+1$.
	To be consistent with this assertion, we shall keep \eqref{BC2b}--\eqref{BC2d} as genuine boundary conditions  and take \eqref{BC2a} as initial constraints.
	Furthermore, we shall prove that the equations \eqref{H.inv1}, \eqref{h.inv1}, and \eqref{e.inv1} hold automatically for $t>0$ if they are satisfied by the initial data \eqref{IC1}. Therefore, solutions of the reformulated problem
	 \eqref{BC2a}--\eqref{RMHD:vec}, \eqref{IC1}, \eqref{vacuum2} that satisfy \eqref{H.inv1},  \eqref{h.inv1}, and \eqref{e.inv1} at $t=0$ are also solutions of the original relativistic plasma--vacuum interface problem \eqref{BC2a}--\eqref{IC1} (see Proposition \ref{Pro:invol} and Remark \ref{Rem1} in \S \ref{sec2.2}).
	
	It is worth noting that the reduced system \eqref{vacuum2} is symmetrizable hyperbolic with the symmetrizer $S_-(\nu) $ defined by (cf.~\cite[\S 2]{T12MR2974767})
	\begin{alignat}{5}
		&S_-(\nu):=\begin{bmatrix}
			1& 0 &0 &0 & \epsilon \nu_3 & -\epsilon \nu_2\\
			0& 1 &0  & -\epsilon \nu_3  &0 & \epsilon \nu_1\\
			0& 0 &1  & \epsilon \nu_2  &-\epsilon \nu_1 &0\\						
			0 & -\epsilon \nu_3 & \epsilon \nu_2 &1& 0 &0 \\
			\epsilon \nu_3 & 0& -\epsilon \nu_1 &0& 1 &0 \\		
			-\epsilon \nu_2 & \epsilon \nu_1& 0 &0& 0 &1		
		\end{bmatrix}&& \qquad \textrm{if }d=3,
		\label{S-3D}
\\[2mm]
	&S_-(\nu):=\begin{bmatrix}
		1 & 0 & -\epsilon \nu_2\\
		0 & 1& \epsilon \nu_1\\
		-\epsilon \nu_2 & \epsilon \nu_1 & 1
	\end{bmatrix}
&&\qquad \textrm{if }d=2.
	\label{S-2D}
\end{alignat}	
Remark that  the positive definiteness of the matrix $S_-(\nu)$ follows from the second condition in
\eqref{nu:def0}.

\subsection{Reformulation in a fixed domain} \label{sec2.2}
Assume without loss of generality that $\|\varphi_0\|_{L^{\infty}(\mathbb{R}^{d-1})}<1$.
Let us transform the nonlinear free boundary problem \eqref{BC2a}--\eqref{RMHD:vec}, \eqref{IC1}, \eqref{vacuum2} to an equivalent fixed boundary problem.
For this purpose, we replace the unknowns $U$ and $u$ by
\begin{align} \nonumber 
	\widetilde{U}({t},x)=U(t,\Phi^+({t}, {x}), {x}'),\qquad  \tilde{u}({t}, {x})=u(t,\Phi^-({t}, {x}), {x}'),
\end{align}
respectively, where $\Phi^{+}$ and $\Phi^-$ are the lifting functions defined by
\begin{align} \label{Phi:def}
	\Phi^{\pm}(t,x):=\pm x_1+ \chi( x_1)\varphi(t,x'),
\end{align}
with the cut-off function $\chi$ satisfying
\begin{align} \label{chi:def}
	\chi\in C^{\infty}_0(\mathbb{R}; [0,1]),\quad 4\|\chi'\|_{L^{\infty}(\mathbb{R})} < 1,\quad
	\chi\equiv1 \ \ \textrm{on } [-1,1].
\end{align}
Then the change of variable is admissible on the time interval $[0,T]$ for some small $T>0$ so that
 $\|\varphi\|_{L^{\infty}([0,T]\times\mathbb{R}^{d-1})}	\leq 2$.
Dropping  the tildes for convenience,
we can reduce \eqref{BC2a}--\eqref{RMHD:vec}, \eqref{IC1}, \eqref{vacuum2} into the equivalent fixed boundary problem
	\begin{subequations} \label{NP1}
		\begin{alignat}{2}
			\label{NP1a}
			&\mathbb{L}_+(U,\Phi^+) :=L_+(U,\Phi^+)U =0
			& &\textrm{in  }  [0,T]\times
			\Omega,\\
			\label{NP1b}
			&\mathbb{L}_-(u,\nu,\Phi^-) :=L_-(\nu,\Phi^-)u =0
			& &\textrm{in  }  [0,T]\times
			\Omega,\\
			\label{NP1c}
			&  \mathbb{B}(U,u,\varphi)=0
			&\qquad&\textrm{on  }    [0,T]\times \Sigma,\\
			\label{NP1d}
			&(U,u, \varphi)|_{t=0}=(U_0,u_0,\varphi_0),  &  &
		\end{alignat}
	\end{subequations}
with $\Omega:=\mathbb{R}_+\times \mathbb{R}^{d-1}$ and
$\Sigma:=\{0\}\times \mathbb{R}^{d-1}$,
where the function $\nu$ satisfies \eqref{nu:def0},
the operators $L_{\pm}$ are defined by
	\begin{alignat}{3}	\label{L+:def}
		&L_+(U,\Phi):=I_{2d+2}\partial_t+\widetilde{A}_1^+(U,\Phi)\partial_1+ \sum_{j=2}^{d}A_j^+(U)\partial_j,\\
		&L_-(\nu,\Phi):=\epsilon I_{3d-3} \partial_t+\widetilde{A}_1^-(\nu,\Phi)\partial_1+\sum_{j=2}^{d} A_j^-(\nu)\partial_j ,
		\label{L-:def}
	\end{alignat}
	with
	\begin{align}
		\nonumber 
		&\widetilde{A}_1^+(U,\Phi):=
		\frac{1}{\partial_1\Phi}\Big(\!-\partial_t\Phi I_{2d+2}+A_1^+(U)-\sum_{j=2}^{d}\partial_j\Phi A_j^+(U)\Big),\\[1mm]
		&\widetilde{A}_1^-(\nu,\Phi):=
		\frac{1}{\partial_1\Phi}\Big(\!-\epsilon\partial_t\Phi I_{3d-3}+A_1^-(\nu)-\sum_{j=2}^{d}\partial_j\Phi A_j^-(\nu)\Big),
		\nonumber 
	\end{align}
the boundary operator $\mathbb{B}$ is given by
	\begin{alignat}{3} \label{Bbb3D:def}
 &\mathbb{B}(U,u,\varphi):=
 \begin{bmatrix}
	\p_t \varphi -v_1+v_2\p_2\varphi+v_3\p_3\varphi\\[0.5mm]
	q-\tfrac{1}{2}|h|^2+\frac{1}{2}|e|^2\\[0.5mm]
	e_2+\p_2 \varphi e_1-\epsilon \p_t \varphi h_3\\[0.5mm]
	e_3+\p_3 \varphi e_1+\epsilon \p_t \varphi h_2
\end{bmatrix}\qquad &&\textrm{if }d=3,\\[1mm]
 &\mathbb{B}(U,u,\varphi):= \begin{bmatrix}
	\p_t \varphi -v_1+v_2\p_2\varphi \\[0.5mm]	 	
	q-\tfrac{1}{2}|h|^2+\frac{1}{2}e^2\\[0.5mm]
	e+\epsilon \p_t \varphi h_2
\end{bmatrix}\qquad &&\textrm{if }d=2,
\label{Bbb2D:def}	
	\end{alignat}
and the matrices $A_j^+(U)$ and $A_j^-(\nu)$ are defined by \eqref{Ai.def.r} and \eqref{A-:def.3D}--\eqref{A-:def.2D}.
For concreteness, we choose $\nu$ to be extended from the trace of the velocity $v$ by
\begin{align} \label{nu:def1}
	\nu(t,x):=\chi(x_1)v(t,0, x')\qquad \textrm{for } x\in\Omega.
\end{align}

Since the equations for $H$ in \eqref{RMHD:vec} are the same as \eqref{RMHD0d}, we can rewrite the equations for $H$ contained in \eqref{NP1a} as
	\begin{align} \label{Hbb:def}
		\mathbb{H}(H,v,\Phi^+):=	
		\big(\p_t^{\Phi^+}+v\cdot \nabla^{\Phi^+}\big)H-\big(H\cdot \nabla^{\Phi^+}\big)v+H\nabla^{\Phi^+}\cdot v=0
		,
	\end{align}
	where 
	\begin{align}
		\label{Phi.d:def}	
		\p_t^{\Phi}:=\p_t-\frac{\p_t\Phi}{\p_1\Phi}\p_1,\
				\nabla^{\Phi}:=(\p_1^{\Phi},\ldots,\p_d^{\Phi}),\
		\p_1^{\Phi}:=\frac{\p_1}{\p_1\Phi},\
		\p_j^{\Phi}:=\p_j-\frac{\p_j\Phi}{\p_1\Phi}\p_1
	\end{align}
	for $j=2,\ldots,d$.
	In the new variables, the equations \eqref{H.inv1}, \eqref{h.inv1}, \eqref{e.inv1}, and \eqref{BC2a} become
	\begin{alignat}{5}
		\label{H.h.inv2}
		   &\nabla^{\Phi^+}\cdot H  =0,\qquad&& \nabla^{\Phi^-}\cdot h  =0\qquad& &\textrm{in  }    \Omega,
		\\[0.5mm]
		\label{H.h.cons2}
		&   H\cdot N=0,\quad& &h\cdot N=0 \qquad && \textrm{on }    \Sigma,\\[0.5mm]
 \label{e.inv2}
	&\nabla^{\Phi^-}\cdot e  =0&&\quad
	 \textrm{in  }   \Omega\qquad&& \textrm{if }d=3.
\end{alignat}	
The following proposition shows that the equations \eqref{H.h.inv2}--\eqref{e.inv2} can be regarded as constraints on the initial data (see Appendix \ref{app:C} for the proof).
\begin{proposition}
	\label{Pro:invol}
For sufficiently smooth solutions of the problem \eqref{NP1} on the time interval $[0,T]$,
the equations \eqref{H.h.inv2}--\eqref{e.inv2} are satisfied for all $t\in [0,T]$ provided they hold initially.	
\end{proposition}	
\begin{remark} \label{Rem1}
Thanks to Proposition \ref{Pro:invol},
solutions of  
problem \eqref{BC2a}--\eqref{RMHD:vec}, \eqref{IC1}, \eqref{vacuum2} satisfy \eqref{H.inv1}, \eqref{h.inv1}, and \eqref{e.inv1} for all $t\in [0,T]$ provided they hold at the initial time.
Therefore, solutions of problem \eqref{BC2a}--\eqref{RMHD:vec}, \eqref{IC1}, \eqref{vacuum2} with initial constraints \eqref{H.inv1}, \eqref{h.inv1}, and \eqref{e.inv1} are also solutions of the original relativistic plasma--vacuum interface problem \eqref{BC2a}--\eqref{IC1}.	
\end{remark}

As in \cite{
	CW08MR2372810, ST13MR3148595,
	LZ23MR4632835,
	T09ARMAMR2481071,YM91MR1092572,TW21MR4201624,TW22aMR4444136,TW22cMR4506578,MSTTY24MR4754746,ST14MR3151094,STW23MR4654316,MSTY23MR4586863}  for various characteristic boundary problems in ideal compressible MHD, where the full regularity cannot be expected,
we shall work in the anisotropic Sobolev spaces $H_*^m$, first introduced by {\sc Chen} \cite{C07MR2289911}.
Let $\alpha=(\alpha_0,\ldots,\alpha_{d+1})\in\mathbb{N}^{d+2}$ and denote
\begin{align} \label{D*}
	\mathrm{D}_*^{\alpha}:=
	\p_t^{\alpha_0} (\sigma \p_1)^{\alpha_1}\p_2^{\alpha_2} \cdots\p_d^{\alpha_d}  \p_1^{\alpha_{d+1}},
	\quad
	\langle \alpha \rangle :=\alpha_0+\cdots+\alpha_d +2\alpha_{d+1},
\end{align}
where $\sigma=\sigma(x_1)$ is a nondecreasing smooth function defined on $\mathbb{R}$ and satisfies
\begin{align} \label{sigma:def}
	\sigma(x_1)=x_1\ \ \textrm{for } |x_1|\leq 1,\quad
	\sigma(x_1)=\pm 2  \ \ \textrm{for }  x_1\gtrless 4.
\end{align}
For any integer $m\in\mathbb{N}$, interval $I\subset \mathbb{R}$, and domain $\Omega\subset \mathbb{R}^{d}$, 
the anisotropic Sobolev space $H_*^{m}(I\times \Omega)$ is defined by
\begin{align} \label{H*m:def}
	H_*^m(I\times\Omega):=
	\left\{ w:\, \mathrm{D}_*^{\alpha} w\in L^2(I\times\Omega) \textrm{ for all } \alpha\in \mathbb{N}^{d+2}\
	\textrm{with}\  \langle \alpha \rangle\leq m  \right\},
\end{align}
and equipped with the norm ${\|}\cdot{\|}_{H^m_*(I\times \Omega)}$,
where
\begin{align}
	{\|}u{\|}_{H^m_*(I\times \Omega)}^2 :=
	\sum_{\langle \alpha\rangle\leq m} \|\mathrm{D}_*^{\alpha} u\|_{L^2(I \times\Omega)}^2.	
	\label{H*m:norm}
\end{align}
{The anisotropic Sobolev space indicates that one order gain of normal differentiation should be compensated by two order loss of tangential differentiation \cite{C07MR2289911}.}

	\subsection{Linear stability results}
	To present the linear stability results for two and three spatial dimensions ($d=2,3$),
we perform the linearization for the problem \eqref{NP1} around a suitable basic state.

Let $(\mathring{U}(t,x), \mathring{u}(t,x))  \in \mathbb{R}^{5d-1}$ and $\mathring{\varphi}(t,x')\in\mathbb{R}$ with $x'=(x_2,\ldots,x_d)$ be some sufficiently smooth functions defined respectively on $\Omega_T:=(-\infty,T)\times \Omega$ and $\Sigma_T:=(-\infty,T)\times \Sigma$,
where $\mathring{U}=(\mathring{q} ,\mathring{v},\mathring{H},\mathring{S})\in\mathbb{R}^{2d+2}$ and $\mathring{u}=(\mathring{h},\mathring{e}) \in \mathbb{R}^{3d-3}$.
We define 
	\begin{alignat}{3}\nonumber
	&\mathring{\Phi}^{\pm}(t,x):={\pm}x_1+\chi(x_1)\mathring{\varphi}(t,x'),\quad  \mathring{\nu}(t,x):=\chi(x_1)\mathring{v}(t,0,x'),\\[1mm]
	&\mathring{N}(t,x):
	=\left(1,  -\chi(x_1)\p_2\mathring{\varphi}(t,x') ,\ldots, -\chi(x_1)\p_d\mathring{\varphi}(t,x') \right)
	\label{N.ring:def}
	\end{alignat}
	for $x\in \Omega.$
	Assume that the basic state $(\mathring{U},\mathring{u},\mathring{\varphi})$  satisfies
\begin{gather}
	\label{bas1}
	 	 \big\|  \big( \mathring{U}-\widebar{U}, \mathring{u}-\bar{u}\big) \big\|_{H_*^{10}(\Omega_{T})} +\big\|\mathring{\varphi} \big\|_{H^{10}(\Sigma_{T})}
	\leq K,\\[1mm] 	
		\|\mathring{\varphi}\|_{L^{\infty}(\Sigma_T)}\leq
	2,
	\quad
{	\|\mathring{v}\|_{L^{\infty}(\Omega_{T})}<\epsilon^{-1},\quad }
	\label{bas2}
	\rho_*<\inf_{\Omega_{T}} \rho(\mathring{U})\leq \sup_{\Omega_{T}} \rho(\mathring{U})<\rho^*,
\end{gather}
for some constant $K>0$ {and constant state $(\widebar{U},\bar{u})\in\mathbb{R}^{5d-1}$},
the constraints
	\begin{alignat}{3}
		\label{bas3}
	&	 \mathring{h}\cdot \mathring{N}=0,
		\quad  \mathring{H}\cdot \mathring{N}=0,
\quad \p_t \mathring{\varphi}=\mathring{v}\cdot \mathring{N}
&&		\qquad  \textrm{on }\Sigma_T,\\[1mm]
		\label{bas4}
&		 \mathbb{H}(\mathring{H},\mathring{v},\mathring{\Phi}^+) =0,\quad
 {\mathbb{L}_-(\mathring{u},\mathring{\nu},\mathring{\Phi}^-) =0}
&&		\qquad   \textrm{on }\overline{\Omega_T},\\[1mm]
		\label{bas5}
&			\nabla^{\mathring{\Phi}^+}\cdot \mathring{H}=0	,\quad
		\nabla^{\mathring{\Phi}^-}\cdot \mathring{h}=0
&&		\qquad   \textrm{on }{\Omega_T},
	\end{alignat}
the boundary conditions
\begin{alignat}{4}
		\label{bas6}	
	&\mathring{e}=-\epsilon\p_t\mathring{\varphi}\mathring{h}_2
\qquad &&\textrm{on }\Sigma_T\quad \textrm{if }d=2,\\[1mm]
		\label{bas7}	
\mathring{e}_2+\p_2 \mathring\varphi \mathring{e}_1=\epsilon \p_t \mathring\varphi \mathring{h}_3,\quad
&\mathring{e}_3+\p_3 \mathring\varphi \mathring{e}_1=-\epsilon \p_t \mathring\varphi \mathring{h}_2		
\qquad &&\textrm{on }\Sigma_T\quad \textrm{if }d=3,
\end{alignat}
and the elliptic equation
\begin{align} \label{bas8}
			\nabla^{\mathring{\Phi}^-}\cdot \mathring{e}=0	
					\qquad   \textrm{on } {\Omega_T}\quad \textrm{if }d=3.
\end{align}	
As shown in Appendix \ref{app:C}, the equations \eqref{bas4} and the third constraint in \eqref{bas3} imply that the first two identities in \eqref{bas3}, and the equations \eqref{bas5}, \eqref{bas8} hold for $t>0$ if they hold at $t=0$.

The linearized operators for the interior equations
\eqref{NP1a}--\eqref{NP1b} around the basic state $(\mathring{U},\mathring{u},\mathring{\varphi})$ are defined by
	\begin{align}
&	\mathbb{L}'_+\big(\mathring{U},\mathring{\Phi}^+\big)(U,\Psi)
	:=   {{L}_{+}\big(\mathring{U},\mathring{\Phi}^+\big) {U}+\mathcal{C}_{+}(\mathring{U},\mathring{\Phi}^+) {U}  -\frac{1}{\partial_1\mathring{\Phi}^+} {L}_+(\mathring{U},\mathring{\Phi}^+)\Psi \p_1 \mathring{U}},	\nonumber 	\\[0.5mm]	
&	\mathbb{L}_{-}'\big(\mathring{u},\mathring{{\nu}}, \mathring{\Phi}^-\big)(u,{\nu},\Psi)
	:=  {	L_-\big(\mathring{{\nu}} , \mathring{\Phi}^-\big) {u}
		+\mathcal{C}_{-}(\mathring{u},\mathring{\nu},\mathring{\Phi}^-)  {\nu}
		-\frac{1}{\partial_1\mathring{\Phi}^-} {L}_-(\mathring{\nu},\mathring{\Phi}^- )\Psi \p_1 \mathring{u} },\nonumber
\end{align}
with $\Psi(t,x)=\chi(x_1)\psi(t,x')$, where
\begin{align}	\nonumber 
	&\mathcal{C}_+(\mathring{U},\mathring{\Phi}^+)U:=
	\sum_{k=1}^{2d+2}U_k\bigg(
	\frac{\p \widetilde{A}^{+}_1}{\p {U_k}}(\mathring{U},\mathring{\Phi}^+) \partial_1 \mathring{U}
	+\sum_{j=2}^{d}\frac{\p A^{+}_j}{\p {U_k}}(\mathring{U}) \partial_j \mathring{U}
	\bigg),\\[0.5mm]
	&\mathcal{C}_{-}(\mathring{u},\mathring{\nu},\mathring{\Phi}^-)  {\nu}
	:=
	\sum_{k=1}^{d}\nu_k\bigg(
	\frac{\p \widetilde{A}^{-}_1}{\p {\nu_k}}(\mathring{\nu},\mathring{\Phi}^-) \partial_1 \mathring{u}
	+\sum_{j=2}^{d}\frac{\p A^{-}_j}{\p {\nu_k}}(\mathring{\nu}) \partial_j \mathring{u}	
	\bigg)=0.
	\label{identity1}
\end{align}
We emphasize that the constraints \eqref{bas5} and \eqref{bas8} are used to obtain the last identity in \eqref{identity1}.
Introducing the good unknowns of {\sc Alinhac} \cite{A89MR976971}:
\begin{align} \label{good}
	\dot{U}:=U-\frac{\Psi}{\partial_1 \mathring{\Phi}^+}\partial_1\mathring{U},
	\quad \dot{u}:=u-\frac{\Psi}{\p_1 \mathring{\Phi}^-}\p_1 \mathring{u},
\end{align}
we infer (cf.~\cite[Proposition\,1.3.1]{M01MR1842775})
\begin{alignat}{3}
	&	\mathbb{L}'_+\big(\mathring{U},\mathring{\Phi}^+\big)(U,\Psi)
	=   {L}_{+}\big(\mathring{U},\mathring{\Phi}^+\big)\dot{U}+\mathcal{C}_{+}(\mathring{U},\mathring{\Phi}^+)\dot{U}  +\frac{\Psi}{\partial_1\mathring{\Phi}^+}	\partial_1\mathbb{L}_+(\mathring{U},\mathring{\Phi}^+ ),
	\label{Alinhac1}
	\\[1mm]
	&		\mathbb{L}_{-}'\big(\mathring{u},\mathring{{\nu}}, \mathring{\Phi}^-\big)(u,{\nu},\Psi)
	=
	L_-\big(\mathring{{\nu}} , \mathring{\Phi}^-\big)\dot{u}
	+\frac{\Psi}{\partial_1\mathring{\Phi}^-}\partial_1\mathbb{L}_-(\mathring{u},\mathring{{\nu}} , \mathring{\Phi}^- ).
	\label{Alinhac2}
\end{alignat}

According to \eqref{Bbb3D:def} and \eqref{Bbb2D:def}, the linearized operator for the boundary conditions \eqref{NP1c} is defined by
	\begin{alignat}{3}
&		\mathbb{B}'(\mathring{U} ,\mathring{u}, \mathring{\varphi})(U,u,\psi)
		:=
		\begin{bmatrix}
			(\p_t + \mathring{v}_2 \p_2 + \mathring{v}_3 \p_3) \psi-v\cdot\mathring{N}\\[0.5mm]
			q-\mathring{h}\cdot h+ \mathring{e}\cdot e\\[0.5mm]				
			e_2+\p_2\mathring{\varphi} e_1-\epsilon \p_t\mathring{\varphi} h_3
			+\mathring{e}_1\p_2 \psi
			-\epsilon \mathring{h}_3 \p_t \psi\\[0.5mm]
			e_3+\p_3\mathring{\varphi} e_1+\epsilon \p_t\mathring{\varphi} h_2
			+\mathring{e}_1\p_3 \psi
			+\epsilon \mathring{h}_2 \p_t \psi
		\end{bmatrix}\quad &&\textrm{if }d=3,
		\label{B'bb3D:def}\\[1mm]
&		\mathbb{B}'(\mathring{U} ,\mathring{u}, \mathring{\varphi})(U,u,\psi)
:=
\begin{bmatrix}
	(\p_t + \mathring{v}_2 \p_2  ) \psi-v\cdot\mathring{N}\\[0.5mm]
	q-\mathring{h}\cdot h+ \mathring{e} e\\[0.5mm]				
	e+\epsilon \p_t\mathring{\varphi} h_2
	+\epsilon \mathring{h}_2 \p_t \psi
\end{bmatrix}\quad &&\textrm{if }d=2.
\label{B'bb2D:def}		
	\end{alignat}

	Neglecting the last terms in \eqref{Alinhac1} and \eqref{Alinhac2} gives the effective linear problem
	\begin{subequations} \label{ELP1}
		\begin{alignat}{3}
			&\mathbb{L}'_{e+}(\mathring{U}, \mathring{\Phi}^+) \dot{U}
			:=L_+(\mathring{U}, \mathring{\Phi}^+) \dot{U}
			+\mathcal{C}_+( \mathring{U},\mathring{\Phi}^+) \dot{U}=f^+
			&\qquad  &\textrm{in } \Omega_T,
			\label{ELP1a}\\[0.5mm]
			&L_-(\mathring{{\nu}},\mathring{\Phi}^-) \dot{u}=f^-
			&\quad  &\textrm{in } \Omega_T,
			\label{ELP1b}\\[0.5mm]
			\label{ELP1c}   &
			\mathbb{B}'_{e}(\mathring{U}, \mathring{u}, \mathring{\varphi}) (\dot{U},\dot{u},\psi)
			=g
			&&\textrm{on } \Sigma_T,\\[0.5mm]
			&(\dot{U},\dot{u},\psi)=0
			&& \textrm{if } t<0.
			\label{ELP1d}
		\end{alignat}
	\end{subequations}
The operator $\mathbb{B}'_{e}$ 
is defined via $\mathbb{B}'_{e}(\mathring{U}, \mathring{u}, \mathring{\varphi}) (\dot{U},\dot{u},\psi)=\mathbb{B}'(\mathring{U} ,\mathring{u}, \mathring{\varphi})(U,u,\psi)$, so that
	\begin{align}
		\label{B'e:def}
		&\mathbb{B}'_{e}(\mathring{U}, \mathring{u}, \mathring{\varphi}) (\dot{U},\dot{u},\psi)
		:=	\mathbb{B}'(\mathring{U} ,\mathring{u}, \mathring{\varphi})(\dot{U},\dot{u},\psi)
		+\psi \mathring{b},
	\end{align}
	where $\mathbb{B}'$  is given in \eqref{B'bb3D:def}--\eqref{B'bb2D:def} and
	\begin{alignat}{3} \label{b.ring3D:def}
&		\mathring{b}=
		\begin{bmatrix}
			\mathring{b}_1\\[0.5mm]
			\mathring{b}_2\\[0.5mm]
			\mathring{b}_3\\[0.5mm]
			\mathring{b}_4	
		\end{bmatrix}
		:=	\begin{bmatrix}
			-\p_1 \mathring{v}\cdot \mathring{N}\\[0.5mm]
			\p_1\mathring{q}-\mathring{{h}}\cdot \p_1\mathring{{h}}+
			\mathring{e}\cdot \p_1\mathring{e}\\[0.5mm]
			\p_1 \mathring{e}_2+\p_2\mathring{\varphi}\p_1 \mathring{e}_1-\epsilon \p_t\mathring{\varphi} \p_1 \mathring{h}_3	 \\[0.5mm]
			\p_1 \mathring{e}_3+\p_3\mathring{\varphi}\p_1 \mathring{e}_1+\epsilon \p_t\mathring{\varphi} \p_1 \mathring{h}_2
		\end{bmatrix}\quad &&\textrm{if }d=3,\\[1mm]
&		\mathring{b}=
\begin{bmatrix}
	\mathring{b}_1\\[0.5mm]
	\mathring{b}_2\\[0.5mm]
	\mathring{b}_3
\end{bmatrix}
:=	\begin{bmatrix}
	-\p_1 \mathring{v}\cdot \mathring{N}\\[0.5mm]
	\p_1\mathring{q}-\mathring{{h}}\cdot \p_1\mathring{{h}}+
	\mathring{e}  \p_1\mathring{e}\\[0.5mm]
	\p_1 \mathring{e}+\epsilon \p_t\mathring{\varphi} \p_1 \mathring{h}_2
\end{bmatrix}\quad &&\textrm{if }d=2.		
\label{b.ring2D:def}
	\end{alignat}
	To be consistent with \eqref{ELP1d}, we assume that the source terms $f^{\pm}$ and $g$ vanish in the past.
The dropped terms in \eqref{Alinhac1} and \eqref{Alinhac2} will be taken as error terms and the case of general initial data will be addressed in the nonlinear analysis.

	Now we state the quantitative linear stability result for 3D relativistic plasma--vacuum interfaces with variable coefficients.
	\begin{theorem}
		\label{thm:main1}
Let $d=3$ and $\delta>0$.		
Suppose that the basic state  $(\mathring{U},\mathring{u},\mathring{\varphi}) $ satisfies
the constraints \eqref{bas1}--\eqref{bas8} for some constant $K>0$ and constant state $(\widebar{U},\bar{u})\in\mathbb{R}^{14}$.
Then there exists $T_0>0$, such that if
$(\mathring{U},\mathring{u},\mathring{\varphi}) $ satisfies
$(\mathring{U}-\widebar{U},\mathring{u}-\bar{u},\mathring{\varphi}) \in H_*^{{m+4}}(\Omega_{T})\times H^{{m+4}}(\Omega_{T}) \times  H^{{m+4}}(\Sigma_T)$
and the stability conditions
\begin{gather}
		\label{non-coll}
{ \big|\mathring{H}\times \mathring{h}\big|\geq \delta >0 \qquad \textnormal{on }\Sigma_T,}\\[1mm]
\sup_{\Sigma_{T}}\frac{\big|\mathring{e}_1+\epsilon \mathring{v}_2\mathring{h}_3-\epsilon \mathring{v}_3\mathring{h}_2\big| \big| {  (  \mathring{H}_2, \mathring{H}_3, \mathring{h}_2,\mathring{h}_3)}\big|\big|\mathring{N}\big|}{\big|\mathring{H}_2\mathring{h}_3 -\mathring{H}_3\mathring{h}_2\big|\big(1-\epsilon\big|\mathring{v}\big|\big)} <\frac{1}{2},
		\label{stability3D}
\end{gather}
and $(f^+,f^-,g )\in H_*^{m+1}(\Omega_{T})\times H^{m+2}(\Omega_{T})\times H^{m+{5/2}}(\Sigma_{T})$ vanishes in the past
for some $T\in (0, T_0]$ and integer $m\geq 6$, then solutions $(\dot{U},\dot{u},\psi)$ of problem \eqref{ELP1} satisfy
\begin{align}
	\nonumber
	&\|\dot{U}\|_{H_*^{m}(\Omega_{T})}^2+
	\|\dot{u}\|_{H^{m}(\Omega_{T})}^2+	\|\psi\|_{H^{m+1/2}(\Sigma_T)}^2&&
	\\[1mm]
	&\quad \leq C(K)\Big(
	\|f^+  \|_{H_*^{m+1}(\Omega_{T}) }^2+\|f^-  \|_{H^{m+2}(\Omega_{T}) }^2
	+\|g\|_{H^{m+{5/2}}(\Sigma_T)}^2&&
	\nonumber\\[1mm]
	& \hspace*{7em}
	+	\mathring{\rm C}_{m+4}
	\big\|\big(f^+,f^-, g \big) \big\|_{H_*^{6}(\Omega_{T}) \times H^{7}(\Omega_{T}) \times H^{8}(\Sigma_{T})  }^2\Big), &&
	\label{tame.es}
\end{align}	
where $\mathring{\rm C}_{m}$ is defined by
\begin{align}
	\label{norm:ring} 	
	\mathring{\rm C}_{m}:= 		 1+ \big\|\mathring{U}-\widebar{U}\big\|_{H_*^{m}(\Omega_{T})}^2+\big\|\mathring{u}-\bar{u}\big\|_{H^{m}(\Omega_{T})}^2 +\big\|\mathring{\varphi} \big\|_{H^{m}(\Sigma_{T})}^2.
\end{align}
	\end{theorem}

\begin{remark}
Condition \eqref{non-coll} is the non-collinearity condition introduced in {\rm\cite{T10MR2718711}} for the stability of non-relativistic plasma--vacuum interfaces.
		Condition \eqref{stability3D} provides an explicit domain for the linear stability of 3D relativistic plasma--vacuum interfaces, which substantially refines the assumption in {\rm\cite[Theorem 4.2]{T12MR2974767}} that $ |\mathring{e}_1+\epsilon \mathring{v}_2\mathring{h}_3-\epsilon \mathring{v}_3\mathring{h}_2 |$ is sufficiently small on ${\Sigma_T}$.
\end{remark}

\begin{remark}
	A relevant model is the MHD--Maxwell interface problem {\rm\cite{MT14MR3194371,CDS14MR3248397}}, consisting of the non-relativistic MHD equations for plasma flow, Maxwell's equations for vacuum fields, and suitable boundary conditions on the interface.
	For this model, we can obtain a stability result similar to Theorem \ref{thm:main1}, which essentially improves the results in {\rm\cite{MT14MR3194371,CDS14MR3248397}} that require the smallness of $ |\mathring{e}_1+\epsilon \mathring{v}_2\mathring{h}_3-\epsilon \mathring{v}_3\mathring{h}_2 ||_{\Sigma_T}$.
\end{remark}

The following main result concerns the existence of solutions to the 2D linearized relativistic plasma--vacuum interface problem with variable coefficients.
		\begin{theorem}
		\label{thm:main2}
Let $d=2$ and $\delta>0$. Suppose that the basic state  $(\mathring{U},\mathring{u},\mathring{\varphi}) $ satisfies \eqref{bas1}--\eqref{bas6} for some constant $K>0$ and constant state $(\widebar{U},\bar{u})\in\mathbb{R}^9$.
Then there exists $T_0>0$, such that if
 $(\mathring{U},\mathring{u},\mathring{\varphi}) $ satisfies
$(\mathring{U}-\widebar{U},\mathring{u}-\bar{u},\mathring{\varphi}) \in H_*^{{m+4}}(\Omega_{T})\times H^{{m+4}}(\Omega_{T}) \times  H^{{m+4}}(\Sigma_T)$
and the stability condition
\begin{align}
	\label{bas10}
{	\big| \mathring{H}\big|+\big|\mathring{h} \big|\geq {\delta} >0	\qquad \textrm{on }\Sigma_T, }
\end{align}
and $(f^+,f^-,g )\in H_*^{m+1}(\Omega_{T})\times H^{m+2}(\Omega_{T})\times H^{m+{5/2}}(\Sigma_{T})$ vanishes in the past
for some time $T\in (0, T_0]$ and integer $m\geq 6$,
then problem \eqref{ELP1} has a unique solution $(\dot{U},\dot{u},\psi)\in H_*^{m}(\Omega_{T})\times H^{m}(\Omega_{T})\times H^{m+1/2}(\Sigma_T)$ satisfying tame estimate \eqref{tame.es}.
	\end{theorem}
\begin{remark}
	Theorem \ref{thm:main2} establishes the existence of solutions to the 2D linearized relativistic plasma--vacuum interface problem with variable coefficients, while the solvability for the 3D problem is left as an interesting open problem for future study.
\end{remark}

\subsection{Nonlinear existence result}	

To state the existence theorem for the nonlinear problem \eqref{NP1} in two-dimensional space ($d=2$), we introduce the compatibility conditions on the initial data \eqref{NP1d}.
Assume without loss of generality that $\|\varphi_0\|_{L^{\infty}(\mathbb{R})}<1$.
Let the initial data satisfy
\begin{align}  \label{initial:H1}
	(U_0-\widebar{U},\, u_0-\bar{u},\varphi_0)\in H^{m+3/2}(\Omega)\times  H^{m+3/2}(\Omega)\times  H^{m+2}(\mathbb{R}),
\end{align}
{the physical assumption $\|v_0\|_{L^{\infty}(\Omega)}<\epsilon^{-1}$},
and the hyperbolicity condition
\begin{align}
	\label{hyperbolicity}
	\rho_*< \rho(U_0)<\rho^* \qquad \textrm{on }\overline{\Omega}
\end{align}
for some constant state $(\widebar{U},\bar{u})\in\mathbb{R}^9$ and integer $m\geq 3$.
{Here $\rho_*$ and $\rho^*$ are the nonnegative constants specified in \eqref{cs:def}, so that the
	matrix $S^+(U_0)$ defined by \eqref{S+U} is positive definite under condition \eqref{hyperbolicity}.}
We define $\Phi_0^{\pm}$ and $\nu_0$ 
by
\begin{alignat}{3}
	\Phi_0^{\pm}(x):=\pm x_1+\chi(x_1)\varphi_0(x_2),\quad
	\nu_0(x):=\chi(x_1)v_0(0, x_2) \qquad   \textrm{for  }   x\in \Omega.
	\nonumber 
\end{alignat}
Let us denote
\begin{align*}
	U_{(j)}  := \p_t^{j }U\big|_{t=0},\quad
	u_{(j)}  := \p_t^{j }u\big|_{t=0},\quad
	\varphi_{(j)}:= \p_t^{j}\varphi\big|_{t=0}\qquad
	\textrm{for }j\in\mathbb{N}.
\end{align*}
Then
${U}_{(0)}={U}_0,$ $u_{(0)}=u_0,$ and $\varphi_{(0)}=\varphi_0.$
For integer $j\geq 1$, we take $j-1$ time derivatives of the interior equations \eqref{NP1a}--\eqref{NP1b} and the first condition in \eqref{NP1c}.
Evaluating the resulting identities at the initial time determines the traces $U_{(j)}  $, $u_{(j)}  $, and $\varphi_{(j)}$ inductively as functions of $U_0,$  $u_0,$ $\varphi_0$, and their spatial derivatives.
In particular, we have
\begin{align}
	\varphi_{(j+1)} = v_{1(j)}\big|_{x_1=0}-\sum_{i=0}^{j}
	\begin{pmatrix}
		j \\ i
	\end{pmatrix}  \p_2\varphi_{(j-i)}v_{2(i)}\big|_{x_1=0},
	\label{trace.id1}
\end{align}
where $\left(\begin{smallmatrix}
	j \\[0.3mm] i
\end{smallmatrix}\right)$ is the binomial coefficient.
Moreover, using the extension \eqref{nu:def1} and the multiplicative properties of Sobolev spaces, we obtain the following lemma; see \cite[Lemma 4.2.1]{M01MR1842775} for the detailed proof.
\begin{lemma} \label{lem:CA1}
	Suppose that 
	$(U_0,u_0,\varphi_0)$ satisfy $\|\varphi_0\|_{L^{\infty}(\mathbb{R})}<1$,
	{$\|v_0\|_{L^{\infty}(\Omega)}<\epsilon^{-1}$,} and \eqref{initial:H1}--\eqref{hyperbolicity} for some constant state $(\widebar{U},\bar{u})\in\mathbb{R}^9$ and integer $m\geq 3$. Then the procedure described above determines the traces
	$(U_{(j)},u_{(j)})\in H^{m+3/2-j}(\Omega)$ and
	$\varphi_{(j)}\in H^{m+2-j}(\mathbb{R})$, for $j=1,2,\ldots,m$,
	satisfying
	\begin{align}
		\sum_{j=1}^{m}
		\left( \|(U_{(j)},\,u_{(j)})\|_{H^{m+3/2-j}(\Omega) }
		+\|\varphi_{(j)}\|_{H^{m+2-j}(\mathbb{R}) }
		\right)\leq C(M_0)
		\nonumber
	\end{align}
	for some positive constant $C(M_0)$ depending on
	\begin{align} \label{M0:def}
		M_0:=\big\|(U_{0}-\widebar{U},\, u_{0}-\bar{u})\big\|_{H^{m+3/2}(\Omega) } +\big\|\varphi_{0}\big\|_{H^{m+2}(\mathbb{R}) } .
	\end{align}
\end{lemma}

Taking time derivatives of the second and third conditions in \eqref{NP1c} (cf.~\eqref{Bbb2D:def}) leads to the terminology for the compatibility conditions on the initial data.
\begin{definition}
	\label{def:1}
	Assume that all the conditions in Lemma \ref{lem:CA1} hold.
	Then the initial data $(U_0, u_0, \varphi_0)$ are said to be compatible up to order $m$,
	if the traces $U_{(j)}$, $u_{(j)}$, and $\varphi_{(j)}$, for $j=1,2,\ldots,m$,  satisfy
	\begin{alignat}{5}
		\label{comp1}
		&  2q_{(0)}= |h_{(0)}|^2-  e_{(0)} ^2,\quad &&
		q_{(j)} = \sum_{i=0}^{j -1}  \begin{pmatrix}
			j-1 \\ i
		\end{pmatrix}
		\left(h_{(i)}\cdot h_{(j-i)}-e_{(i)}e_{(j-i)}\right) \quad&&\textrm{on }\Sigma,\\
		\label{comp2}
		& e_{(0)}=-\epsilon \varphi_{(1)} h_{2(0)},\quad &&
		e_{(j)} = -\epsilon\sum_{i=0}^{j }  \begin{pmatrix}
			j \\ i
		\end{pmatrix} \varphi_{(i+1)}h_{2(j-i)}
		\quad&&\textrm{on }\Sigma.
	\end{alignat}
\end{definition}


We are now ready to present the local existence theorem for the nonlinear problem \eqref{NP1} in 2D.

\begin{theorem}
	\label{thm:main3}
	Let $d=2$ and $\delta_0>0$.
	Assume that the initial data
	$(U_0, u_0, \varphi_0)$ satisfy  	$\|\varphi_0\|_{L^{\infty}(\mathbb{R})}<1$,
	{the physical assumption $\|v_0\|_{L^{\infty}(\Omega)}<\epsilon^{-1}$,}
	the hyperbolicity condition  \eqref{hyperbolicity},
	the constraints \eqref{H.h.inv2} and \eqref{H.h.cons2},
	\begin{align}  \label{initial:H2}
		(U_0-\widebar{U},\, u_0-\bar{u},\varphi_0)\in H^{m+11.5}(\Omega)\times  H^{m+11.5}(\Omega)\times  H^{m+12}(\mathbb{R}),
	\end{align}
	the compatibility conditions up to order $m+10$,
	and the stability condition
	\begin{align} \label{stability2D}
		|H_0|+| h_0 |    \geq \delta_0>0 \qquad \textrm{on } \Sigma,
	\end{align}
	for some constant state $(\widebar{U},\bar{u})\in\mathbb{R}^9$ and integer $m\geq 12$.
	Then there exists $T>0$,
	such that problem \eqref{NP1} has a unique solution $(U,u, \varphi)$ on the time interval $[0, T]$ with
	$U-\widebar{U}\in H_{*}^{m}([0,T]\times\Omega)$,
	$u-\bar{u}\in H^{m}([0,T]\times\Omega)$, and
	$ \varphi\in  H^{m+0.5}([0,T]\times\mathbb{R}).$
\end{theorem}


\begin{remark}
	In view of Remark \ref{Rem1}, 
	Theorem \ref{thm:main3} implies corresponding existence results for both the free boundary problem \eqref{BC2a}--\eqref{RMHD:vec}, \eqref{IC1}, \eqref{vacuum2} and the original relativistic plasma--vacuum interface problem \eqref{BC2a}--\eqref{IC1},
	since the relations $|\p_1\Phi^{\pm}|>0$  and the constraints \eqref{H.h.inv2}--\eqref{H.h.cons2} hold on $[0, T]$.
\end{remark}

\begin{remark}
	{Condition \eqref{stability2D} was exploited in {\rm\cite{MSTTY24MR4754746}} for the stability of 2D non-relativistic plasma--vacuum interfaces to gain an extra half-order regularity for the interface function.}
\end{remark}

\begin{remark}
	From \eqref{NP1c} and \eqref{Gamma:def}, we have
	$$	|h|^2=e^2+2p+ \frac{|H|^2}{ \varGamma^{2}}+ {\epsilon^2}  (v\cdot H)^2
	\qquad\textnormal{on }\Sigma.
	$$
Therefore, for general plasmas satisfying the conventional assumption that $p>0$ whenever $\rho>0$, the stability condition \eqref{stability2D} is implied by the hyperbolicity condition \eqref{hyperbolicity}.
	This argument
	excludes the case of Chaplygin gases, where the pressure $p$ can be negative for small densities and condition \eqref{stability2D} does not appear to follow from hyperbolicity.
\end{remark}

	
	\subsection{Notation} \label{sec:nontation}
	We employ the following notation throughout the paper:
\begin{itemize}
		\item
Let $d=2,3$ denote the spatial dimension.
We adopt the Einstein summation convention, for which Greek and Latin indices range from $0$ to $d$ and from $1$ to $d$, respectively.	

\item
We denote by $\p_0=\p_t:=\frac{\p}{\p t}$ the time derivative and by $\nabla:=(\p_1,\ldots,\p_d)$ the gradient with $\p_i:=\frac{\p}{\p x_i}$.
For any scalar $f$,
vectors $u=(u_1,\ldots,u_d)$ and
$w=(w_1,\ldots,w_d)$, and matrix $F=(F_{ij})_{d\times d}$,
we define 
\begin{align*}
	&\nabla\cdot F:=  (\p_j F_{1j},\ldots,\p_j F_{dj}),
	\quad \nabla\cdot u:= \p_j u_{j},
	\quad u\otimes w=(u_i w_j)_{d\times d},
	\\[0.5mm]
	&u\times w:=
	\left\{
	\begin{aligned}
		&u_1 w_2 -u_2 w_1\ \quad & \textrm{if }d=2,\\
		&(u_2 w_3 -u_3 w_2,u_3 w_1 -u_1 w_3,u_1 w_2 -u_2 w_1)\ \quad & \textrm{if }d=3,
	\end{aligned}
	\right.\\[0.5mm]
	&\nabla\times w:=
	\left\{
	\begin{aligned}
		&\p_1 w_2 -\p_2 w_1\ \quad & \textrm{if }d=2,\\
		&(\p_2 w_3 -\p_3 w_2,\p_3 w_1 -\p_1 w_3,\p_1 w_2 -\p_2 w_1) \ \quad & \textrm{if }d=3,
	\end{aligned}
	\right.\\[0.5mm]
	&u\times f = -f\times u:=(u_2 f, -u_1 f),
	\quad
	\nabla\times f :=(\p_2f , -\p_1 f)
	\quad \textrm{if }d=2.
\end{align*}

		\item
For a vector $w\in\mathbb{R}^d$, we write $w'$ for its last $d-1$ components, so that $w=(w_1,w')$  and $\nabla:=(\p_1,\nabla')$.	
Set $\mathrm{D}:=(\p_0,\nabla)$ and $\mathrm{D}_{\rm tan}:=(\p_0, \nabla')$.
For multi-indices $\alpha:=(\alpha_0,\alpha_1,\ldots,\alpha_{d})$ and $\beta:=(\beta_0, \beta_{2}, \ldots,\beta_{d})$, we denote
\begin{align*}
	\mathrm{D}^{\alpha}:=\p_0^{\alpha_0}\p_1^{\alpha_1} \cdots\p_d^{\alpha_d} ,\quad
	\mathrm{D}_{\rm tan}^{\beta}:=\p_0^{\beta_0}\p_2^{\beta_2}\cdots \p_d^{\beta_d}.
\end{align*}
	For $m\in \mathbb{N}$, let $\mathring{\rm c}_m$ denote
a generic and smooth matrix-valued function of
$\{(\mathrm{D}^{\alpha} \mathring{U},\mathrm{D}^{\alpha} \mathring{{u}},\mathrm{D}^{\alpha}\mathring{\varphi}): |\alpha|\leq m\}$.
The exact form of $\mathring{\rm c}_m$ may vary at different places.

		\item
		Denote by   $\Sigma:=\{0\}\times\mathbb{R}^{d-1}$
		the boundary of
		$\Omega:=\mathbb{R}_{+}\times \mathbb{R}^{d-1}$.
Set
		\begin{align*}
			\Sigma_T:=(-\infty,T)\times \Sigma,\quad
			\Omega_T :=(-\infty,T)\times \Omega \qquad
			\textrm{for }T>0.
		\end{align*}	
		Define the operator $\mathrm{D}_*^{\alpha}$ and the weight function $\sigma$ by \eqref{D*} and \eqref{sigma:def}, respectively.
		Let the anisotropic Sobolev space $H_*^m(I\times \Omega)$ and its norm
		$\|\cdot \|_{H_*^m(I\times\Omega)}$ be defined by \eqref{H*m:def} and \eqref{H*m:norm}.
		We abbreviate $ \mathrm{D}_*:=(\p_t,\sigma \p_1, \p_2,\ldots,\p_d)$.		
		
\item
We denote by $O_m$ and $I_m$ the zero and identity matrices of order $m$, respectively.
We write $C$ for some universal positive constant
and $C(\cdot)$ for some positive constant depending on the quantities listed in the parenthesis.
Notation $A\lesssim B$ means that $A \leq C B$.
Let $A\lesssim_{a_1,\ldots,a_m} B$ denote $A \leq C(a_1,\ldots,a_k)B$ for given parameters $a_1,\ldots,a_k$.
		
\end{itemize}

\vspace*{2mm}
	
	\section{Reformulation of the linearized problem}	 \label{sec:refor}
To prove Theorems \ref{thm:main1}--\ref{thm:main3}, we construct certain auxiliary functions to partially homogenize the effective linear problem \eqref{ELP1} and introduce a suitable change of unknowns to separate the noncharacteristic variables from the system.	

	\subsection{Partial homogenization}
In the following lemma, we perform a partial homogenization for the problem \eqref{ELP1}.
	
	\begin{lemma}
	\label{lem:homo}		
		Suppose that the basic state $(\mathring{U},\mathring{u},\mathring{\varphi}) $ satisfies \eqref{bas1}--\eqref{bas8}.
		Then there exist auxiliary functions  $u^{\natural}=(h^{\natural}, e^{\natural}) \in\mathbb{R}^{3d-3}$
		and $U^{\natural}=(q^{\natural}, v^{\natural}, H^{\natural},0)\in\mathbb{R}^{2d+2}$, such that
		\begin{align} \nonumber
			&	 \|(u^{\natural},\, q^{\natural},\, v^{\natural})\|_{H^{m+1}(\Omega_{T})}^2
			+\|H^{\natural}\|_{H_*^{m+1}(\Omega_{T})}^2\\[0.5mm]
			&\qquad \lesssim_{K}
			\|f^+  \|_{H_*^{m+1}(\Omega_{T}) }^2+\|f^-  \|_{H^{m+2}(\Omega_{T}) }^2
			+\|g\|_{H^{m+{5/2}}(\Sigma_T)}^2
			\nonumber\\[0.5mm]
			& \qquad \qquad\
			+\mathring{\rm C}_{m+4}
			\big\|\big(f^+,f^-, g \big) \big\|_{H_*^{6}(\Omega_{T}) \times H^{7}(\Omega_{T}) \times H^{8}(\Sigma_{T})  }^2,
			\label{Unatural:es}
		\end{align}
		and the {differences}
		\begin{align} \label{Uu:def}
			U:=\dot{U}-U^{\natural},\quad {u}:=\dot{u}-u^{\natural}
		\end{align}	
		satisfy the equations
		\begin{subequations} \label{ELP2}
			\begin{alignat}{3}
				\label{ELP2a}
				&\mathbb{L}'_{e+}(\mathring{U}, \mathring{\Phi}^+) {U}=\tilde{f}^+
				:= f^+-\mathbb{L}_{e+}'(\mathring{U},\mathring{\Phi}^+)U^{\natural}
				&\qquad &  \textrm{in } \Omega_T,\\[0.5mm]
				\label{ELP2b}
				& L_-(\mathring{\nu},\mathring{\Phi}^-) {u}=0
				&& \textrm{in }\Omega_T,\\[0.5mm]
				\label{ELP2c}
				&  \mathbb{B}'_{e}(\mathring{U}, \mathring{u}, \mathring{\varphi}) (U,{u},\psi)=0
				&& \textrm{on } \Sigma_T,\\[0.8mm]
				\label{ELP2d}
				&
				({U},u,\psi) =0   &&\textrm{if } t<0,
			\end{alignat}
		\end{subequations}
		and the homogeneous constraints
		\begin{alignat}{3}
			\label{H.cons3}	
&H\cdot\mathring{N}-\mathring{H}'\cdot \nabla'\psi+\p_1\mathring{H}\cdot\mathring{N}\psi =0 && \textrm{on } \Sigma_{T},\\[0.5mm]			
			\label{h.cons3}		
			&h\cdot\mathring{N}-\mathring{h}'\cdot \nabla'\psi+\p_1\mathring{h}\cdot\mathring{N}\psi =0&&  \textrm{on } \Sigma_{T},	 \\[0.5mm]		
\label{H.inv3}	
&\p_1\big(H\cdot\mathring{N}\big)+\nabla'\cdot\big(\p_1\mathring{\Phi}^+ H'\big) =0
&\quad& \textrm{in }\Omega_T,\\[0.5mm]	
			\label{h.inv3}	
&\p_1\big(h\cdot\mathring{N}\big)+\nabla'\cdot \big(\p_1\mathring{\Phi}^- h'\big) =0	&&	
\textrm{in }\Omega_T,\\[0.5mm]				
			\label{e.inv3}	
			&\p_1\big(e\cdot\mathring{N}\big)+\nabla'\cdot\big(\p_1\mathring{\Phi}^- e'\big)=0 &&	
			\textrm{in }\Omega_T\quad \textrm{if }d=3 ,
		\end{alignat}	
where 		$\mathring{\rm C}_{m}$
and $\mathring{N}$ are defined by \eqref{norm:ring} and \eqref{N.ring:def}, respectively,
$\nabla':=(\p_2,\ldots,\p_d)$, and $w':=(w_2,\ldots,w_d)$.
	\end{lemma}
	\begin{proof}
We prove the lemma for $d=3$ in the following five steps.		
The 2D case can be handled similarly, and hence we omit the details here. 

		\vspace*{2mm}
		\noindent {\it Step 1: Construction of  $u^{\natural}$}.\quad
		We first reformulate the vacuum equations in \eqref{ELP1} to those with homogeneous source terms.
		For this purpose, we set $u^{\natural}$ as the unique solution of 
		the problem	
		\begin{align}
			\label{u.natural:eq}
			\left\{
			\begin{aligned}
				&L_-(\mathring{\nu},\mathring{\Phi}^-) {u}^{\natural}=f^-
				&\qquad  &\textrm{in } \Omega_T,	\\
				& e_2^{\natural}+\p_2\mathring{\varphi}e_1^{\natural}-\epsilon \p_t\mathring{\varphi} h_3^{\natural}=g_3
				&\quad  &\textrm{on } \Sigma_T,\\[0.3mm]
				& e_3^{\natural}+\p_3\mathring{\varphi}e_1^{\natural}+\epsilon \p_t\mathring{\varphi} h_2^{\natural}=g_4
				&\quad  &\textrm{on } \Sigma_T,\\[0.3mm]		
				&u^{\natural}=0
				&& \textrm{if } t<0.
			\end{aligned}
			\right.
		\end{align}
Then ${u}:=\dot{u}-u^{\natural}$ satisfies \eqref{ELP2b} and the last two conditions in \eqref{ELP2c}.		
However, the problem \eqref{u.natural:eq} is not maximally nonnegative, so the standard arguments in \cite{S95MMASMR1346663,S96aMR1401431,OSY95MR1346224,R85MR0797053} cannot be applied directly for the solvability.
Certain reductions should be made for constructing solutions $u^{\natural}$ in $H^{m+1}(\Omega_{T})$.
The proof of the existence and regularity of solutions to problem \eqref{u.natural:eq} is postponed to Appendix \ref{app:D}.
In particular, we have
\begin{align}   \nonumber
	\| u^{\natural} \|_{H^{m+1}(\Omega_{T})}^2
	\lesssim_K
	\|(f^-, g)\|_{H^{m+2}(\Omega_{T})\times H^{m+{5/2}}(\Sigma_T)}^2
	+ \mathring{\rm C}_{m+4}
	\|(f^-, g)\|_{H^{7}(\Omega_{T})\times H^{8}(\Sigma_T)}^2.
\end{align}

		\vspace*{2mm}
\noindent {\it Step 2: Construction of $(q^{\natural}, v^{\natural})$}.\quad
The first two equations in \eqref{ELP2c} are implied by
setting
\begin{align} \label{U.sharp:def}
	q^{\natural}:=\mathfrak{R}_{T}\big(g_2+\mathring{h}\cdot h^{\natural}-\mathring{e}\cdot e^{\natural}\big),\quad
	v^{\natural}:=\left(\mathfrak{R}_{T}(-g_1 ),\,  0,\,0\right),
\end{align}
where $\mathfrak{R}_T$ is a continuous extension operator from $H^{m+1/2}(\Sigma_T)$ to $H^{m+1}(\Omega_{T})$.

		\vspace*{2mm}
\noindent {\it Step 3: Construction of $H^{\natural}$}.\quad
Let us find a suitable function $H^{\natural}$ such that \eqref{H.cons3} and \eqref{H.inv3} are satisfied.
Note that the fifth, sixth, and seventh components of \eqref{ELP1a} can be written as
\begin{multline}
  \mathbb{H}_e'(\mathring{H},\mathring{v},\mathring{\Phi}^+)(\dot{H},\dot{v})
 	:=\big(\p_t^{\mathring{\Phi}^+}+\mathring{v}\cdot \nabla^{\mathring{\Phi}^+}\big)\dot{H}-\big(\mathring{H}\cdot\nabla^{\mathring{\Phi}^+}\big)\dot{v}+\mathring{H}\nabla^{\mathring{\Phi}^+}\cdot\dot{v}	 \\
	+\big(\dot{v}\cdot\nabla^{\mathring{\Phi}^+}\big)\mathring{H}-\big(\dot{H}\cdot\nabla^{\mathring{\Phi}^+}\big)\mathring{v}+\dot{H}\nabla^{\mathring{\Phi}^+}\cdot\mathring{v}
	=\big(f_5^+,\, f_6^+,\, f_7^+\big).
	\label{H.e:def}
\end{multline}
We define $H^{\natural}$ as the unique solution of
\begin{align}
	\mathbb{H}_e'(\mathring{H},\mathring{v},\mathring{\Phi}^+)(H^{\natural},v^{\natural})=
	\big(f_5^+,\, f_6^+,\, f_7^+\big)\qquad \textrm{in }\Omega_{T}
	\label{H.nat}
\end{align}
vanishing in the past,
where the operator $\mathbb{H}_e'(\mathring{H},\mathring{v},\mathring{\Phi}^+)$ is defined by \eqref{H.e:def}.
The solvability of the linear equation \eqref{H.nat} does not need any boundary condition due to the third condition in \eqref{bas3}.

		\vspace*{2mm}
\noindent {\it Step 4:  Proof of \eqref{Unatural:es},  \eqref{ELP2}, and \eqref{h.cons3}}.\quad
In view of \eqref{u.natural:eq}, \eqref{U.sharp:def}, and \eqref{H.nat}, we can deduce \eqref{ELP2} and  \eqref{Unatural:es} by using the embedding and Moser-type calculus inequalities for standard and anisotropic Sobolev spaces (see, for instance, \cite[\S 3.2]{TW21MR4201624}).
Take the scalar product of \eqref{h:equ} with $\mathring{N}=(1,-\p_2\mathring{\Phi}^-,-\p_3\mathring{\Phi}^-)$ to get
\begin{align}
	\nonumber 
	\epsilon \p_t\big(h\cdot \mathring{N}\big)
	+\p_2\big(e_3 +\p_3\mathring{\varphi}e_1+\epsilon \p_t\mathring{\varphi} h_2\big)
	-\p_3\big(e_2+\p_2\mathring{\varphi}e_1-\epsilon \p_t\mathring{\varphi} h_3\big)
	=0
	\ \  \textrm{on } \Sigma_{T}.
\end{align}
Plug the last two conditions in \eqref{ELP2c} into the last identity to have
\begin{align*}
	&	\epsilon\p_t\big(h\cdot \mathring{N}-\mathring{h}'\cdot\nabla'\psi -\nabla'\cdot\mathring{h}'\psi\big)\\
	&\qquad	+\p_2\big(\psi(\epsilon\p_t \mathring{h}_2+\p_3\mathring{e}_1 -\mathring{b}_4)\big)
	+\p_3\big(\psi(\epsilon\p_t \mathring{h}_3-\p_2\mathring{e}_1 +\mathring{b}_3)\big)	
	=0
	\quad \textrm{on }\Sigma_{T},
\end{align*}
where $\mathring{b}_3$ and $\mathring{b}_4$ are defined by \eqref{b.ring3D:def}.
The second identity in \eqref{bas4} implies
\begin{align*}
	\epsilon\p_t \mathring{h}_3-\p_2\mathring{e}_1 +\mathring{b}_3=
	\epsilon\p_t \mathring{h}_2+\p_3\mathring{e}_1 -\mathring{b}_4=0
	\quad \textrm{on }\Sigma_{T}.
\end{align*}
Then we use \eqref{ELP2d} to obtain the constraint \eqref{h.cons3}.		
		
		\vspace*{2mm}
		\noindent {\it Step 5: Proof of \eqref{H.cons3} and \eqref{H.inv3}--\eqref{e.inv3}}.\quad
	For $U=(q, v,H,S)$ defined by \eqref{Uu:def},
it  follows from \eqref{H.e:def} and \eqref{H.nat} that  $\mathbb{H}_e'(\mathring{H},\mathring{v},\mathring{\Phi}^+)(H ,v )=0$.
Then we can deduce from \eqref{ELP2c}
that (cf.~\cite[Proposition 2]{T09ARMAMR2481071})
\begin{alignat}{3}
	\nonumber 
	&\big(\p_t^{\mathring{\Phi}^+}+\mathring{v}\cdot \nabla^{\mathring{\Phi}^+}\big) \nabla^{\mathring{\Phi}^+}\cdot {H}+\big(\nabla^{\mathring{\Phi}^+}\cdot \mathring{v}\big) \nabla^{\mathring{\Phi}^+}\cdot {H}=0 \quad
	&&\textrm{in } \Omega_T,\\
	& \big(\p_t+\mathring{v}'\cdot\nabla'+\nabla'\cdot\mathring{v}'\big)\big(H\cdot\mathring{N}-\mathring{H}'\cdot\nabla'\psi+\p_1\mathring{H}\cdot\mathring{N}\psi \big)=0
	&\qquad &\textrm{on } \Sigma_T,
	\nonumber
\end{alignat}
which combined with \eqref{ELP2d} yield the constraints \eqref{H.cons3} and \eqref{H.inv3}.
		
The first three equations  in \eqref{ELP2b} can be rewritten as
\begin{align}
	\epsilon \p_t^{\mathring{\Phi}^-}h+\nabla^{\mathring{\Phi}^-}\times e+\epsilon  ( \nabla^{\mathring{\Phi}^-}\cdot h )\mathring{\nu}=0, \label{h:equ}
\end{align}
where
$\p_t^{\mathring{\Phi}}$ and $\nabla^{\mathring{\Phi}}$ are defined by \eqref{Phi.d:def}.
Applying the divergence operator $\nabla^{\mathring{\Phi}^-}\cdot$ to \eqref{h:equ},
we utilize $\p_t\mathring{\Phi}^{-}=\mathring{\nu}\cdot \mathring{N}|_{\Sigma_T}$ and \eqref{ELP2d} to obtain $\nabla^{\mathring{\Phi}^-}\cdot h=0$, which leads to the identity \eqref{h.inv3}.		

Similarly, we can deduce from the last three equations in \eqref{ELP2b}, i.e.,
\begin{align}
	\epsilon \p_t^{\mathring{\Phi}^-}e-\nabla^{\mathring{\Phi}^-}\times h+\epsilon   (\nabla^{\mathring{\Phi}^-}\cdot e) \mathring{\nu}=0, \label{e:equ}
\end{align}
that $\nabla^{\mathring{\Phi}^-}\cdot e=0$, which implies \eqref{e.inv3}.		
This completes the proof of this lemma.
	\end{proof}
	
	\subsection{Separation of the non-characteristic variables}
	Next we reformulate \eqref{ELP2} into an equivalent problem with a concise boundary matrix in order to separate the non-characteristic variables from the system.
	To this end, we introduce
	\begin{align}
		\label{W:def}
		W:= (q,\, \mathring{N}\cdot v,\, v',\, \mathring{N}\cdot H, \, H',\, S) =\mathring{J}_1^{-1}U
 ,
	\end{align}
	where
	\begin{align}
		\label{J1.ring:def}
		&\mathring{J}_1:=
		\begin{bmatrix}
			\;1 \;&   & & & &  \\
			&\; 1\; & \;\nabla' {\mathring{\Phi}^+} \;       &  & & \\
			& 0 & I_{d-1}        & & & \\
			&  &      & \;1 \;&\;\nabla' {\mathring{\Phi}^+} \;    & \\
			&  &      & 0 & I_{d-1}  & \\
			&  &  &  &    & \w{\;1\;}
		\end{bmatrix}
	\end{align}
	with $\nabla':=(\p_2,\ldots,\p_d)$.
	In view of the conditions \eqref{ELP2c} and constraint \eqref{h.cons3}, we set 
	\begin{align}	
&\mu:=\begin{bmatrix}
			h\cdot\mathring{N}\\[0.5mm]
			h_2+\p_2\mathring{\Phi}^-h_1+\epsilon  \p_t\mathring{\Phi}^- e_3 \\[0.5mm]
			h_3+\p_3\mathring{\Phi}^-h_1-\epsilon  \p_t\mathring{\Phi}^- e_2 \\[0.5mm]
			e\cdot\mathring{N}\\[0.5mm]
			e_2+\p_2\mathring{\Phi}^-e_1-\epsilon  \p_t\mathring{\Phi}^- h_3 \\[0.5mm]
			e_3+\p_3\mathring{\Phi}^-e_1+\epsilon  \p_t\mathring{\Phi}^- h_2
		\end{bmatrix}
 \textrm{ if }d=3,	\label{mu:def1.3D}
	\\[1mm]
&\mu:=\begin{bmatrix}
	h\cdot\mathring{N}\\[0.5mm]
	h_2+\p_2\mathring{\Phi}^-h_1+\epsilon  \p_t\mathring{\Phi}^- e \\[0.5mm]
	e+\epsilon \p_t\mathring{\Phi}^- h_2
\end{bmatrix}	
 \textrm{ if }d=2,
	\label{mu:def1}
	\end{align}
which can be written as
\begin{align}
	\label{mu:def2}
	\mu:=\mathring{J}_2^{-1}u.
\end{align}
The change of variables \eqref{mu:def1.3D}--\eqref{mu:def2} is invertible due to \eqref{bas3} and the second condition in \eqref{bas2}.
	Then we use the constraints \eqref{bas6} and \eqref{bas7} to transform the problem \eqref{ELP2}  to
	\begin{subequations}
		\label{ELP3}
		\begin{alignat}{3}
			\label{ELP3a}
			&{\mathsf{A}}_0^+\p_t W+ {\mathsf{A}}_j^+\p_j W +{\mathsf{A}}_{d+1}^+ W =\mathring{J}_1^{\mathsf{T}}S_+(\mathring{U})\tilde{f}^+
			&\qquad &\textrm{in }\Omega_T,\\[1mm]
			\label{ELP3b}
			&{\mathsf{A}}_0^-\p_t   \mu+  {\mathsf{A}}_j^-\p_j   \mu  +{\mathsf{A}}_{d+1}^- \mu =0
			&&\textrm{in }\Omega_T,\\[1mm]
			\label{ELP3c}
&
\mathsf{B}(W,\mu, \psi) =0
&& \textrm{on }\Sigma_T,\\[1mm]
			\label{ELP3d}
			&
			(W,\mu, \psi) =0
			&& \textrm{if } t<0,
		\end{alignat}
	\end{subequations}
where
	\begin{align}
\left\{\begin{aligned}		
			&{\mathsf{A}}_{0}^{+}:= \mathring{J}_1^{\mathsf{T}}S_{+}(\mathring{U})\mathring{J}_1,\quad &&
			{\mathsf{A}}_{0}^{-}:=\epsilon\mathring{J}_2^{\mathsf{T}}S_{-}(\mathring{\nu})\mathring{J}_2, 	
			\\[0.5mm]
			& 		{\mathsf{A}}_{1}^{+}:= \mathring{J}_1^{\mathsf{T}}S_{+}(\mathring{U})\widetilde{A}_{1}^{+}(\mathring{U},\mathring{\Phi}^+)\mathring{J}_1,\quad &&
			{\mathsf{A}}_{1}^{-}:=\mathring{J}_2^{\mathsf{T}}S_{-}(\mathring{\nu})\widetilde{A}_{1}^{-}(\mathring{\nu},\mathring{\Phi}^-)\mathring{J}_2,\\[0.5mm]
			&{\mathsf{A}}_{j}^{+}:= \mathring{J}_1^{\mathsf{T}}S_{+}(\mathring{U}){A}_{j}^{+}(\mathring{U})\mathring{J}_1,\quad&&
			{\mathsf{A}}_{j}^{-}:=\mathring{J}_2^{\mathsf{T}}S_{-}(\mathring{\nu}){A}_{j}^{-}(\mathring{\nu})\mathring{J}_2, \\[0.5mm]
			&	{\mathsf{A}}_{d+1}^{+}:= \mathring{J}_1^{\mathsf{T}}S_{+}(\mathring{U})\mathbb{L}'_{e+}(\mathring{U}, \mathring{\Phi}^+) \mathring{J}_1,
			\quad &&
			{\mathsf{A}}_{d+1}^{-}:=\mathring{J}_2^{\mathsf{T}}S_{-}(\mathring{\nu})L_-(\mathring{\nu}, \mathring{\Phi}^-) \mathring{J}_2,
\end{aligned}\right.		
		\label{A.sf:def}
	\end{align}
for $ j=2,\ldots,d,$ with $S_{+} $ and $S_- $ being defined by \eqref{S+U} and \eqref{S-3D}--\eqref{S-2D},
and
\begin{alignat}{3}
	\nonumber		
	&\mathsf{B}(W,\mu, \psi)	
	:=\begin{bmatrix}
		W_2-(\p_t+\mathring{v}_2\p_2+\mathring{v}_3\p_3+\mathring{b}_1)\psi\\[1mm]
		W_1-\mathring{{h}}_2 \mu_2-\mathring{{h}}_3 \mu_3+\mathring{e}_1 \mu_4	+\mathring{b}_2 \psi\\[1mm]				
		\mu_5-\epsilon \mathring{h}_3 \p_t \psi+\mathring{e}_1\p_2\psi
		+\mathring{b}_3 \psi \\[1mm]
		\mu_6+\epsilon \mathring{h}_2 \p_t \psi+\mathring{e}_1\p_3\psi
		+\mathring{b}_4\psi
	\end{bmatrix}\quad
	&&	\textrm{if }d=3,\\[1mm]
	&\mathsf{B}(W,\mu, \psi)	
	:=\begin{bmatrix}
		W_2-(\p_t+\mathring{v}_2\p_2 +\mathring{b}_1)\psi\\[1mm]		
		W_1-\mathring{{h}}_2 \mu_2 		+\mathring{b}_2 \psi\\[1mm]		
		\mu_3+\epsilon \mathring{h}_2 \p_t \psi+\mathring{b}_3\psi
	\end{bmatrix}\quad
	&&	\textrm{if }d=2.
	\nonumber	
\end{alignat}
	Using \eqref{iden:App1} and \eqref{bas3}, after a long but direct calculation, we discover
	\begin{alignat}{5}
		\label{decom1.3D}
	&	{\mathsf{A}}_1^+\big|_{\Sigma}=
		\begin{bmatrix}
			0 & \varGamma(\mathring{v}) &0 \\
			\varGamma(\mathring{v}) & 0 & 0\\
			0  &  0 &  O_{6}
		\end{bmatrix} ,\
		{\mathsf{A}}_1^-\big|_{\Sigma}=-
		\begin{bmatrix}
			 0  &  \epsilon \mathring{v}_2 &  \epsilon \mathring{v}_3 & 0 &0&0 \\[0.3mm]
			 \epsilon \mathring{v}_2\;&0&0&0&0&-1\\[0.3mm]
			\epsilon \mathring{v}_3&0&0&0&1&0\\[0.3mm]
			0 &0&0  &  0  &  \epsilon \mathring{v}_2 &  \epsilon \mathring{v}_3 \\[0.3mm]
			0&0&1 &\epsilon \mathring{v}_2 &0&0\\[0.3mm]
			0&-1&0 &\epsilon \mathring{v}_3&0&0
		\end{bmatrix}
		&&\quad
	\textrm{if }d=3,\\[2mm]
	\label{decom1.2D}
&	{\mathsf{A}}_1^+\big|_{\Sigma}=
	\begin{bmatrix}
		0 & \varGamma(\mathring{v}) &0 \\
		\varGamma(\mathring{v}) & 0 & 0\\
		0  & 0  & O_{4}
	\end{bmatrix} ,\
	{\mathsf{A}}_1^-\big|_{\Sigma}=-
\begin{bmatrix}
	0 & \epsilon \mathring{v}_2&0 \\
	\epsilon \mathring{v}_2\; \;& 0 \;& \;-1\\
	0 & -1 & 0
\end{bmatrix} &&\quad \textrm{if }d=2.
\end{alignat}	

To distinguish the non-characteristic variables from the others,
we shall make further reformulation according to \eqref{decom1.3D} and \eqref{decom1.2D}.
Precisely, we set
	\begin{align}
		w:=\begin{bmatrix}
			\mu_1  \\[0.5mm]
			\mu_2-\epsilon \mathring{\nu}_3{\mu_4} \\[0.5mm]
			\mu_3+\epsilon \mathring{\nu}_2{\mu_4} \\[0.5mm] 		
			{\mu_4}  \\[0.5mm]
			\mu_5+\epsilon \mathring{\nu}_3\mu_1 \\[0.5mm]
			\mu_6-\epsilon \mathring{\nu}_2\mu_1
		\end{bmatrix}
		\ \ \textrm{if }d=3,
		\quad
	w:=\begin{bmatrix}
	\mu_1  \\[0.5mm]  \mu_2 \\[0.5mm] \mu_3-\epsilon \mathring{\nu}_2\mu_1
\end{bmatrix}		
		\ \ \textrm{if }d=2,	
		\label{w:def1}	
	\end{align}
which can be written as
\begin{align}
	w=\mathring{J}_3^{-1}\mathring{J}_2^{-1}u,
		\label{w:def2}		
\end{align}	
where $\mathring{J}_2$ is given by \eqref{mu:def1.3D}--\eqref{mu:def2}
and
	\begin{align}
		\label{J3.ring:def}
		\mathring{J}_3:=
		\begin{bmatrix}
			1 & \w{0}& \w{0}& 0& 0&0\\
			0& 1& 0  & \epsilon \mathring{\nu}_3& 0& 0\\
			0&0&1 & -\epsilon \mathring{\nu}_2& 0 & 0\\
			0& 0&0 &1 & 0& 0\\
			-\epsilon \mathring{\nu}_3& \w{0}& \w{0}&	\w{0}& \w{1}&\w{0} \\
			\epsilon \mathring{\nu}_2 & 0 & 0 &0&0&1 		
		\end{bmatrix}\ \ \textrm{if }d=3,
		\quad
\mathring{J}_3:=
\begin{bmatrix}
	1 & 0& 0\\
	0 & 1 & 0\\
	\epsilon \mathring{\nu}_2 & \w{0} &\w{1}
\end{bmatrix}\  \ \textrm{if }d=2.
	\end{align}
	Using the constraint \eqref{h.cons3}, we obtain the following equivalent symmetric hyperbolic problem in terms of
	$W$ and $w$ (cf.~\eqref{W:def} and \eqref{w:def1}--\eqref{w:def2}):
	\begin{subequations}
		\label{ELP4}
		\begin{alignat}{3}
			\label{ELP4a}
			& {\bm{A}}_0^+\p_tW  +{ \bm{A}}_j^+\p_jW +{ \bm{A}}_{d+1}^+ W = \bm{f}^+
			&\qquad  &\textrm{in }\Omega_T,\\
			\label{ELP4b}
			& { \bm{A}}_0^-\p_tw  +{ \bm{A}}_j^-\p_j w    +{ \bm{A}}_{d+1}^-  w =0
			&&\textrm{in }\Omega_T,\\
			\label{ELP4c}
&\bm{B}(W,w,\psi)=0
&&\textrm{on }\Sigma_T,\\			
			\label{ELP4d}
			&
			(W,w, \psi) =0
			&& \textrm{if } t<0,
		\end{alignat}
	\end{subequations}
where
	\begin{alignat}{3} \label{A.bm:def}
\left\{
\begin{aligned}
	&	\bm{f}^+:=\epsilon\varGamma(\mathring{v})^{-1}\mathring{J}_1^{\mathsf{T}}S_+(\mathring{U})\tilde{f}^+, \quad
\bm{A}_i^+:=\epsilon\varGamma(\mathring{v})^{-1}\mathsf{A}_i^+,\quad
\bm{A}_i^-:=\mathring{J}_3^{\mathsf{T}}\mathsf{A}_i^-\mathring{J}_3,
\\
&		\bm{A}_{d+1}^+:=\epsilon\varGamma(\mathring{v})^{-1}\mathsf{A}_{d+1}^+,\quad	
\bm{A}_{d+1}^-:=\mathring{J}_3^{\mathsf{T}}\big({\mathsf{A}}_0^-\p_t  +{\mathsf{A}}_j^-\p_j   +{\mathsf{A}}_{d+1}^-\big)\mathring{J}_3, 	
\end{aligned}		
\right.
\end{alignat}
for $i=0,\ldots,d,$
with $\tilde{f}^+$  and $\mathsf{A}_i^{\pm}$ given in  \eqref{ELP2a} and \eqref{A.sf:def}, and
\begin{alignat}{3}
\label{B.bm3D:def}	
	&\bm{B}(W,\mu, \psi)	
	:=\begin{bmatrix}
		W_2-(\p_t+\mathring{v}'\cdot\nabla'+\mathring{b}_1)\psi\\[1mm]				
		W_1-\mathring{{h}}_2 w_2-\mathring{{h}}_3 w_3+\mathring{\zeta} w_4
		+\mathring{b}_2 \psi\\[1mm]
		w_5-\epsilon \mathring{h}_3 (\p_t+\mathring{v}'\cdot\nabla')\psi
		+\mathring{\zeta} \p_2\psi
		+\mathring{b}_5 \psi \\[1mm]
		w_6+\epsilon \mathring{h}_2 (\p_t+\mathring{v}'\cdot\nabla') \psi+\mathring{\zeta} \p_3\psi
		+\mathring{b}_6\psi
	\end{bmatrix}\quad
	&&	\textrm{if }d=3,\\[1mm]
	&\bm{B}(W,\mu, \psi)	
	:=\begin{bmatrix}
		W_2-(\p_t+\mathring{v}_2\p_2 +\mathring{b}_1)\psi\\[1mm]				
		W_1-\mathring{{h}}_2 w_2 +\mathring{b}_2 \psi\\[1mm]
		w_3+\epsilon \mathring{h}_2 (\p_t+\mathring{v}_2\p_2)\psi+\mathring{b}_4\psi
	\end{bmatrix}\quad
	&&	\textrm{if }d=2,
\label{B.bm2D:def}		
\end{alignat}
with (cf.~\eqref{b.ring3D:def}--\eqref{b.ring2D:def})
\begin{alignat}{3}
	\label{zeta:def}
&\mathring{\zeta}:=
\mathring{e}_1+\epsilon \mathring{v}_2\mathring{h}_3-\epsilon \mathring{v}_3\mathring{h}_2, && \\[1mm]	
	\nonumber
&
	(\mathring{b}_5,\,
	\mathring{b}_6)
:=
	(\mathring{b}_3+\epsilon\mathring{v}_3\p_1\mathring{h}\cdot\mathring{N},\,
	\mathring{b}_4-\epsilon\mathring{v}_2\p_1\mathring{h}\cdot\mathring{N})
 \quad  && \textrm{ if }d=3,\\[1mm]	
&\mathring{b}_4 :=\mathring{b}_3-\epsilon\mathring{v}_2\p_1\mathring{h}\cdot\mathring{N}
  && \textrm{ if }d=2.
\nonumber
\end{alignat}
The equivalence of the problems \eqref{ELP4} and \eqref{ELP2} follows by utilizing the identity \eqref{h.cons3} and the vacuum equations \eqref{ELP2b}.

By virtue of \eqref{decom1.3D}--\eqref{decom1.2D}, we obtain
	\begin{alignat}{3}
		\label{decom2.3D}
&		\bm{A}_1^+\big|_{\Sigma}=
		\begin{bmatrix}
			0 & \epsilon&0 \\
			\epsilon& 0 & 0\\
			0 \;&\; 0 \;&\; O_6
		\end{bmatrix}	=:{\bm{B}}_1^{+},\   \
		\bm{A}_1^-\big|_{\Sigma}=
		\begin{bmatrix}
			&  &  & 0 &0&0 \\
			& & &0&0&1\\
			& &  &0&\w{-1}&0\\
			0 &0&0  &  & &\\
			0&0&\w{-1} & & & \\
			0&1&0&&&
		\end{bmatrix}	=:{\bm{B}}_1^{-}
\ \ \	&&	\textrm{if }d=3,\\[1mm]
		\label{decom2.2D}
&		\bm{A}_1^+\big|_{\Sigma}=
\begin{bmatrix}
	0 & \epsilon&0 \\
	\epsilon& 0 & 0\\
	0 \;&\; 0 \;&\; O_4
\end{bmatrix}	=:{\bm{B}}_1^{+},\ \
\bm{A}_1^-\big|_{\Sigma}=
	\begin{bmatrix}
	0 & 0&0 \\
	0& 0 & 1\\
	\w{0} & \w{1} & \w{0}
\end{bmatrix}	=:{\bm{B}}_1^{-}
\ &&	\textrm{if }d=2.
\end{alignat}
Then we decompose
\begin{align}
	{\bm{A}}_1^{\pm}={\bm{B}}_{0}^{\pm}+{\bm{B}}_1^{\pm}
	\quad \textrm{with } \ {\bm{B}}_{0}^{\pm}|_{\Sigma}=0.
	\label{decom3}
\end{align}
According to the kernels of the boundary matrices $\bm{A}_1^{\pm}$ on ${\Sigma}$, we denote by
	\begin{alignat}{5}
		\label{W.nc:def}		
	\mathcal{W}_{\rm nc}:=
\left\{
\begin{aligned}
	&(W_1,W_2,w_2,w_3,w_5,w_6)\quad &&\textrm{if }d=3,\\[1mm]
	&(W_1,W_2,w_2,w_3)\quad &&\textrm{if }d=2,
\end{aligned}
\right.
	\end{alignat}
the noncharacteristic variables,
and by
	\begin{alignat}{5}
	\label{W.c:def}		
	\mathcal{W}_{\rm c}:=
	\left\{
	\begin{aligned}
		&(W_3,W_4,W_5,W_6,W_7,W_8,w_1,w_4)\quad &&\textrm{if }d=3,\\[1mm]
		&(W_3,W_4,W_5,W_6,w_1)\quad &&\textrm{if }d=2,
	\end{aligned}
	\right.
\end{alignat}	
the characteristic variables.

It is important to point out that the homogeneous constraint \eqref{h.cons3} has been used to make the boundary conditions independent of the characteristic variable $w_1$.
Consequently, the boundary conditions \eqref{ELP4c} for the 2D problem depend upon the traces of $(W,w)$ only through the noncharacteristic variables $\mathcal{W}_{\rm nc}$.
In contrast, for the 3D problem, $\mathring{\zeta}=\mathring{e}_1+\epsilon \mathring{v}_2\mathring{h}_3-\epsilon \mathring{v}_3\mathring{h}_2$ does not vanish generally,
so the boundary conditions \eqref{ELP4c} involve the traces of  the characteristic variables $\mathcal{W}_{\rm c}$.
This situation, which appears also in the study of thermoelastic contact discontinuities \cite{CSW20MR4110436}, is different from the standard one (see, for instance, \cite[\S 4.1]{BGS07MR2284507}).


\vspace*{2mm}
	\section{Linear stability in 3D} \label{sec:lin3D}
This section is devoted to proving Theorem \ref{thm:main1}, that is, the linear stability of relativistic plasma--vacuum interfaces in three-dimensional space.

	\subsection{Prelude} \label{sec:prelude}
Let $m\in\mathbb{N}_+$ and $\alpha=(\alpha_0,\ldots,\alpha_{d+1})\in\mathbb{N}^{d+2}$ with
$\langle \alpha \rangle \leq m$.	
Applying the operator $\mathrm{D}_*^{\alpha}:=\p_t^{\alpha_0} (\sigma \p_1)^{\alpha_1}\p_2^{\alpha_2} \cdots\p_d^{\alpha_d}  \p_1^{\alpha_{d+1}}
$  
to \eqref{ELP4a} and \eqref{ELP4b},
we take the scalar product of the resulting equations
with $\mathrm{D}_*^{\alpha} W$ and $\mathrm{D}_*^{\alpha}w$, respectively,
to find
\begin{align}
 \mathcal{E}_{\alpha}(t)
:=\int_{\Omega} \left(\bm{A}_0^+\mathrm{D}_*^{\alpha}W\cdot \mathrm{D}_*^{\alpha}W
	+ \bm{A}_0^-\mathrm{D}_*^{\alpha}w\cdot \mathrm{D}_*^{\alpha}w\right)(t,x)\mathrm{d}x
 	\leq
	 \mathcal{Q}_{\alpha}+\int_{\Omega_t} {R}_{1},
	 \label{es1}
\end{align}
where
\begin{align}
	\nonumber
	{R}_{1}:=\; &
	\mathrm{D}_*^{\alpha}W\cdot \bigg(
	2\mathrm{D}_*^{\alpha}(\bm{f}^+- {\bm{A}}_{d+1}^+ W)
	-\sum_{i=0}^{d}\big(2[\mathrm{D}_*^{\alpha},{\bm{A}}_i^+\p_i]W
	-\p_i {\bm{A}}_i^+\mathrm{D}_*^{\alpha}W\big) \bigg)\\
	&	
	-	\mathrm{D}_*^{\alpha}w\cdot \bigg(
	2\mathrm{D}_*^{\alpha}({\bm{A}}_{d+1}^- w)
	+\sum_{i=0}^{d}\big(2[\mathrm{D}_*^{\alpha},{\bm{A}}_i^-\p_i]w
	-\p_i {\bm{A}}_i^-\mathrm{D}_*^{\alpha}w\big) \bigg),
 \nonumber 
	\\[2mm]
		\label{Q:def}
\mathcal{Q}_{\alpha}:=\; &\int_{\Sigma_t}{Q}_{\alpha},\qquad 		
{Q}_{\alpha}:=\big(\bm{A}_1^+ \mathrm{D}_*^{\alpha}W\cdot \mathrm{D}_*^{\alpha}W
+
\bm{A}_1^- \mathrm{D}_*^{\alpha}w\cdot \mathrm{D}_*^{\alpha}w\big)\big|_{\Sigma}.	
\end{align}

In this section, we consider the spatial dimension $d=3$.
To estimate the last term in \eqref{es1},
we plug \eqref{decom2.3D} and \eqref{decom3} into \eqref{ELP4a}--\eqref{ELP4b} and get
\begin{alignat}{3}
	\label{d1W.3D}
	&	\begin{bmatrix}
		\epsilon \p_1 W_2\\   \epsilon \p_1 W_1  \\0
	\end{bmatrix}
	=\bm{f}^+-{\bm{A}}_{4}^+ W-\sum_{i=0,2,3}{\bm{A}}_i^+\p_i W-{\bm{B}}_0^+ \p_1W
	\qquad &&\textrm{in }\Omega_{T}	,\\
	\label{d1w.3D}	
	&	\begin{bmatrix}
		0\\ \p_1 w_6\\ -\p_1 w_5\\ 0\\ -\p_1 w_3\\  \p_1w_2
	\end{bmatrix}
	=-{\bm{A}}_4^- w-\sum_{i=0,2,3}{\bm{A}}_i^-\p_i w-{\bm{B}}_0^- \p_1w
	\quad &&\textrm{in }\Omega_{T},
\end{alignat}
where $\p_0:=\p_t$.
Similar to \cite[Lemma 3.5]{TW21MR4201624}, we can use the identities \eqref{d1W.3D}--\eqref{d1w.3D} and the Moser-type calculus inequalities for anisotropic Sobolev spaces 
to deduce
\begin{align}
	\label{R1:es}
	\int_{\Omega_t} {R}_{1}\lesssim_K
	\mathcal{M}_1(t)
	\qquad \textrm{for all } \langle\alpha\rangle \leq m,
\end{align}
where
{\begin{align}
 	\mathcal{M}_1(t):=
{\|}(\bm{f}^+,W,w){\|}_{H_*^m(\Omega_{t})}^2
	+\mathring{\rm C}_{m+4}  \|(\bm{f}^+,W,w)\|_{W_*^{2,\infty}(\Omega_{t})}^2
\label{M1cal:def}	
\end{align}}
with
$
\|u\|_{W_*^{2,\infty}(\Omega_{t}) }:=\|(u, \mathrm{D}_* u)\|_{W^{1,\infty}(\Omega_{t}) }.	
$

Next we estimate the boundary integral in \eqref{es1}.
Plug \eqref{decom2.3D} into \eqref{Q:def} to get
\begin{align} \label{Q3D:id1}
	{Q}_{\alpha} =
	2\epsilon\mathrm{D}_*^{\alpha}W_1\mathrm{D}_*^{\alpha}W_2
	+2\mathrm{D}_*^{\alpha}w_2\mathrm{D}_*^{\alpha}w_6
	-2\mathrm{D}_*^{\alpha}w_3\mathrm{D}_*^{\alpha}w_5.
\end{align}
If $\alpha_1>0$, then ${Q}_{\alpha}=0$, since $\sigma=0$ on the boundary $\Sigma$.
If $\alpha_1=0$ and $\alpha_4>0$, then we utilize the identities \eqref{d1W.3D}, \eqref{d1w.3D}, and \eqref{Q3D:id1} to obtain
	\begin{align}
		\mathcal{Q}_{\alpha}\lesssim_K
		\sum_{i=0,2,3}\|\mathrm{D}_*^{\alpha-{\mathbf{e}}} (\bm{f}^+,  {\bm{A}}_4^+ W,{\bm{A}}_4^- w,{\bm{A}}_i^+\p_i W,{\bm{A}}_i^-\p_i w,{\bm{B}}_0^{+}\p_1W,{\bm{B}}_0^{-}\p_1w)\|_{L^2(\Sigma_t)}^2	 \nonumber
	\end{align}
	for ${\mathbf{e}}:=(0, 0, 0, 0, 1).  $
	As in \cite[\S\,3.5.3]{TW22cMR4506578}, we can employ the trace theorem and Moser-type calculus inequalities for anisotropic Sobolev spaces to derive
\begin{align}
\mathcal{Q}_{\alpha}\lesssim_K \mathcal{M}_1(t)
\qquad \textrm{for all }  \langle \alpha \rangle \leq m \ \textrm{ with }
\alpha_1+\alpha_4>0.
\nonumber
\end{align}
Plugging the last inequality and \eqref{R1:es} into \eqref{es1} and using the positivity of the matrices $\bm{A}_0^{\pm}$ lead to
\begin{align}
	\sum_{\langle \alpha\rangle \leq m,\
	\alpha_1+\alpha_4>0}	\|(\mathrm{D}_*^{\alpha}W,\, \mathrm{D}_*^{\alpha}w)(t)\|_{L^2(\Omega)}^2
	\lesssim_K  \mathcal{M}_1(t),
\label{es2}
\end{align}
where $\mathcal{M}_1(t)$ is defined by \eqref{M1cal:def}.	
The case with pure tangential derivatives $\alpha_1=\alpha_4=0$ will be handled in the remaining part of this section.

	\subsection{Cancellation} \label{sec:cancel}
We shall establish a cancellation to estimate the boundary integral in \eqref{es1} with pure tangential derivatives.
Let $\alpha=(\alpha_0,0,\alpha_2,\alpha_3,0)\in\mathbb{N}^5$
satisfy $\langle \alpha \rangle  \leq m$.
Then $\mathrm{D}_*^{\alpha}=\p_t^{\alpha_0}\p_2^{\alpha_2}\p_3^{\alpha_3}$ and $\alpha_0+\alpha_2+\alpha_3\leq m$. It follows from \eqref{Q3D:id1} and \eqref{ELP4c} that
\begin{align}\nonumber
	{Q}_{\alpha}
	=\; &
	2\epsilon \mathrm{D}_*^{\alpha}W_2\mathrm{D}_*^{\alpha}\Big(
	\mathring{{h}}_2 w_2+\mathring{{h}}_3 w_3-\mathring{\zeta} w_4
	-\mathring{b}_2 \psi
	\Big)\nonumber\\[0.5mm]
 &
-2 \mathrm{D}_*^{\alpha}w_2\mathrm{D}_*^{\alpha}\Big(
\epsilon \mathring{h}_2 (W_2-\mathring{b}_1\psi) +\mathring{\zeta} \p_3\psi
+\mathring{b}_6\psi
\Big)\nonumber
\\[0.5mm]
	&-2\mathrm{D}_*^{\alpha}w_3\mathrm{D}_*^{\alpha}\Big(
	\epsilon \mathring{h}_3 (W_2-\mathring{b}_1\psi)
	-\mathring{\zeta} \p_2\psi
	-\mathring{b}_5 \psi
	\Big).\nonumber
\end{align}
Then we get
	\begin{align}  \label{Q3D:id2}
		{Q}_{\alpha}
		=
		-2 \mathring{\zeta}  \Big(
\epsilon \mathrm{D}_*^{\alpha}W_2\mathrm{D}_*^{\alpha} w_4
+\mathrm{D}_*^{\alpha}w_2 \mathrm{D}_*^{\alpha}  \p_3\psi
-\mathrm{D}_*^{\alpha}w_3\mathrm{D}_*^{\alpha} \p_2\psi
\Big)+R_2
	\end{align}
with
\begin{align}
R_2:=
	\mathrm{D}_*^{\alpha} \mathcal{W}_{\rm nc}
	\Big(\big[\mathrm{D}_*^{\alpha},\,\mathring{\rm c}_0\big]
	\big(\mathcal{W}_{\rm nc},w_4, \nabla'\psi\big)
	+\mathrm{D}_*^{\alpha}\big(\psi \mathring{\rm c}_1\big)
	\Big), \nonumber
\end{align}
where
$\mathcal{W}_{\rm nc}$ is defined by \eqref{W.nc:def},
$\nabla':=(\p_2,\p_3)$,
and $\mathring{\rm c}_m$ denotes a generic and smooth matrix-valued function of
$\{(\mathrm{D}^{\alpha} \mathring{U},\mathrm{D}^{\alpha} \mathring{{u}},\mathrm{D}^{\alpha}\mathring{\varphi}): |\alpha|\leq m\}$.
In view of \eqref{w:def1} and the first condition in \eqref{ELP4c}, we have
\begin{align}
	\nonumber
	W_2=\p_t\psi+(\mathring{v}'\cdot\nabla' )\psi+\mathring{\rm c}_1 \psi ,\ \
	w_2=\mu_2-\epsilon \mathring{v}_3 {w_4},\ \
		w_3=\mu_3+\epsilon \mathring{v}_2{w_4} \quad
		\textrm{on }\Sigma_T.
\end{align}
Insert the last identities into \eqref{Q3D:id2} to discover
\begin{align}\nonumber
	{Q}_{\alpha}=\;&
-2 \mathring{\zeta}  \Big(
\epsilon \mathrm{D}_*^{\alpha}\p_t\psi\mathrm{D}_*^{\alpha} {w_4}
+\mathrm{D}_*^{\alpha}	\mu_2 \mathrm{D}_*^{\alpha}  \p_3\psi
-\mathrm{D}_*^{\alpha}\mu_3\mathrm{D}_*^{\alpha} \p_2\psi
\Big)
 \\[0.5mm]
&
-2 \epsilon\mathring{\zeta}  \Big(
\mathrm{D}_*^{\alpha} (\mathring{v}'\cdot\nabla' )\psi\mathrm{D}_*^{\alpha} {w_4}
- \mathrm{D}_*^{\alpha}({w_4}	\mathring{v}'   ) \cdot \mathrm{D}_*^{\alpha}  \nabla' \psi
\Big)
+ \mathring{\rm c}_0  \mathrm{D}_*^{\alpha}(\mathring{\rm c}_1 \psi  )\mathrm{D}_*^{\alpha} w_4
+R_2
\nonumber\\[0.5mm]
\nonumber
=\;&\p_t \Big(\!-2 \mathring{\zeta}
\epsilon \mathrm{D}_*^{\alpha} \psi\mathrm{D}_*^{\alpha} {w_4} \Big)
-\p_3\Big(2\mathring{\zeta}   \mathrm{D}_*^{\alpha}	\mu_2 \mathrm{D}_*^{\alpha}   \psi   \Big)
+ \p_2\Big(2\mathring{\zeta} \mathrm{D}_*^{\alpha}\mu_3\mathrm{D}_*^{\alpha}\psi  \Big)\\[0.5mm]
&+R_3
+2\mathring{\zeta}  \mathrm{D}_*^{\alpha}\psi\mathrm{D}_*^{\alpha}\Big(\epsilon \p_t {w_4}+\p_3\mu_2-\p_2\mu_3\Big),
\label{Q3D:id3}
\end{align}	
with
	\begin{align}
		R_3:=\;&	\mathrm{D}_*^{\alpha}\big(\mathcal{W}_{\rm nc},w_4,\nabla'\psi \big)
		\Big(\mathring{\rm c}_0\big[\mathrm{D}_*^{\alpha},\,\mathring{\rm c}_0\big]
		\big(\mathcal{W}_{\rm nc},w_4,\nabla'\psi\big)
		+\mathring{\rm c}_0\mathrm{D}_*^{\alpha}\big(\psi \mathring{\rm c}_1\big)
		\Big) \nonumber\\[0.5mm]
\label{R3:def}
&
+\mathring{\rm c}_0 \mathrm{D}_*^{\alpha} \nabla'\psi \big[\mathrm{D}_*^{\alpha},\,\mathring{\rm c}_0\big] w
+\mathring{\rm c}_1 \mathrm{D}_*^{\alpha}\psi \mathrm{D}_*^{\alpha}w
+\mathring{\rm c}_1 \mathrm{D}_*^{\alpha}\psi \big[\mathrm{D}_*^{\alpha},\,\mathring{\rm c}_0\big]w .
\end{align}

The following lemma provides the desired cancellation for estimating the last term in \eqref{Q3D:id3}.
\begin{lemma}
It holds that
\begin{align}
	\label{cancellation}
		\epsilon \p_t {w_4}+\p_3\mu_2-\p_2\mu_3=0\qquad \textrm{on } \Sigma_{T}.
\end{align}
\end{lemma}
\begin{proof}
It follows from the definition \eqref{mu:def1.3D} that
\begin{align} \nonumber
 &	\epsilon \p_t {w_4}+\p_3\mu_2-\p_2\mu_3\\
\nonumber & \quad	=
	\epsilon \p_t e_1-\epsilon \nabla'\mathring{\Phi}^-\cdot \p_t e'
	+\epsilon \p_t\mathring{\Phi}^-\nabla'\cdot  e'
		-\p_2h_3-\p_3\mathring{\Phi}^-\p_2 h_1
	+\p_3h_2+ \p_2\mathring{\Phi}^-\p_3 h_1\\
\nonumber & \quad	=
\big(	\epsilon \p_t^{\mathring{\Phi}^-} e_1-\p_2^{\mathring{\Phi}^-} h_3+\p_3^{\mathring{\Phi}^-} h_2\big)
	-\p_2  \mathring{\Phi}^- \big( 	\epsilon \p_t^{\mathring{\Phi}^-} e_2-\p_3^{\mathring{\Phi}^-} h_1+\p_1^{\mathring{\Phi}^-} h_3 \big)\\
&\qquad\; 	  -\p_3  \mathring{\Phi}^- \big( 	\epsilon \p_t^{\mathring{\Phi}^-} e_3-\p_1^{\mathring{\Phi}^-} h_2+\p_2^{\mathring{\Phi}^-} h_1 \big)
+\epsilon\p_t  \mathring{\Phi}^-
\nabla^{\mathring{\Phi}^-} \cdot e.
\label{can:pro1}
\end{align}	
In view of the equations \eqref{ELP2b}, we have
\begin{align} \nonumber
	\epsilon \p_t^{\mathring{\Phi}^-}e-\nabla^{\mathring{\Phi}^-}\times h+\epsilon \nabla^{\mathring{\Phi}^-}\cdot e \mathring{\nu}=0\qquad  \textrm{in }\Omega_T.
\end{align}
Plugging the last identity into \eqref{can:pro1} and using the third constraint in \eqref{bas3},
we consider the resulting equation on the boundary $\Sigma$ to obtain \eqref{cancellation}.
\end{proof}
Plugging the identity \eqref{cancellation} into \eqref{Q3D:id3} and using \eqref{ELP4d}, we discover
\begin{align} \label{Q:es1}
	\mathcal{Q}_{\alpha}=\int_{\Sigma_t}{Q}_{\alpha}=
	\mathcal{R}_{\alpha}(t)+
	\int_{\Sigma_{t}}R_3,
\end{align}
where $R_3$ is defined by \eqref{R3:def} and
\begin{align}\label{R.cal:def}
	\mathcal{R}_{\alpha}(t):=-2\epsilon \int_{\Sigma} \mathring{\zeta}
	\mathrm{D}_*^{\alpha} \psi\mathrm{D}_*^{\alpha} {w_4}.
\end{align}

Let us estimate the last term in \eqref{Q:es1} first
and put the estimate of the instant boundary integral $\mathcal{R}_{\alpha}(t)$ in the next two subsections.
Introducing
\begin{align} \label{W.cal:nc}
	\widetilde{\mathcal{W}}_{\rm nc}:=
	(W_1,W_2, \mathring{N}\cdot H
	,  w),
\end{align}	
{for which the normal derivatives can be compensated from the identities \eqref{H.inv3}--\eqref{e.inv3} and \eqref{d1W.3D}--\eqref{d1w.3D} as }
\begin{align}
\label{W.cal:id}
	\p_1 \widetilde{\mathcal{W}}_{\rm nc}=
	\mathring{\rm c}_1\mathrm{D}_*\big( W,  w \big)
	+	\mathring{\rm c}_2\big(\bm{f}^+, W, w \big)
\end{align}
for $\mathrm{D}_*:=(\p_t,\sigma\p_1,\p_2,\p_3)$.
Hence
\begin{align}
	\big\| \widetilde{\mathcal{W}}_{\rm nc} \big\|_{H^{m-1/2}(\Sigma_t)}^2
	\lesssim \sum_{\langle\gamma\rangle\leq m-1}  \big\|\mathrm{D}_*^{\gamma}\widetilde{\mathcal{W}}_{\rm nc} \big\|_{H^{1}(\Omega_{t}) }^2
\lesssim_{K}  \mathcal{M}_1(t)
	\label{Wnc:es}
\end{align}
for $m\geq 1$.
We suppose $m\geq1$, so that
\begin{align}
	\nonumber
	\int_{\Sigma_{t}}R_3
	\;  & \lesssim_K  \big\|\mathrm{D}_*^{\alpha}\big(\widetilde{\mathcal{W}}_{\rm nc},\nabla'\psi\big) \big\|_{H^{-1/2}(\Sigma_t)}^2
	+ \big\|\big([\mathrm{D}_*^{\alpha},\, \mathring{\rm c}_0]\big(\widetilde{\mathcal{W}}_{\rm nc},\nabla'\psi\big),\, \mathrm{D}_*^{\alpha} (\psi \mathring{\rm c}_1 )\big) \big\|_{H^{1/2}(\Sigma_t)}^2  \\[0.5mm]
	\;  & \lesssim_K \big\|  \widetilde{\mathcal{W}}_{\rm nc}  \big\|_{H^{m-1/2}(\Sigma_t)}^2
	+\big\|\psi\big\|_{H^{m+1/2}(\Sigma_t)}^2
	+\big\|[\mathrm{D}_*^{\alpha},\, \mathring{\rm c}_1]\big(\widetilde{\mathcal{W}}_{\rm nc},\nabla'\psi, \psi\big) \big\|_{H^{1/2}(\Sigma_t)}^2. \nonumber
\end{align}
Using the Moser-type calculus inequalities yields
\begin{align}
		\int_{\Sigma_{t}}R_3 \lesssim_K   \mathcal{M}(t)
		\qquad \textrm{for \ }
	\mathcal{M}(t)	:=
		\mathcal{M}_1(t)+\mathcal{M}_2(t),	
	\label{R3:es}
\end{align}
where $\mathcal{M}_1(t)$ is given in \eqref{M1cal:def} and
\begin{align}
	\mathcal{M}_2(t):=
		&	{\|}\psi{\|}_{H^{m+1/2}(\Sigma_{t})}^2	+\mathring{\rm C}_{m+4}  \|\psi\|_{L^{\infty}(\Sigma_{t})}^2
.
 \label{M2cal:def}		
\end{align}
Plug \eqref{Q:es1}, \eqref{R3:es}, and \eqref{R1:es} into \eqref{es1} to get
\begin{align}
	\mathcal{E}_{\alpha}(t)
	\leq \mathcal{R}_{\alpha}(t)+C(K)\mathcal{M}(t)\qquad
	\textrm{for } \alpha=(\alpha_0,0,\alpha_2,\alpha_3,0) \textrm{ with }|\alpha|\leq m,
	\label{es3}
\end{align}
where $\mathcal{R}_{\alpha}(t)$ and $\mathcal{M}(t)$ are given in \eqref{R.cal:def} and \eqref{R3:es}, respectively.

\subsection{Estimate of $\mathcal{R}_{\alpha}$ with spatial derivatives} \label{sec:Qes1}
In this subsection, we make the estimate of $\mathcal{R}_{\alpha}(t)$ defined by \eqref{R.cal:def} for $\alpha=(\alpha_0,0,\alpha_2, \alpha_3,0)$ with $\alpha_2+\alpha_3\geq 1$ and $|\alpha|\leq m$.
To this end, we
denote $\mathbf{e}_3:=(0,0,1,0,0)$, $\mathbf{e}_4:=(0,0,0,1,0)$
and let $\beta=(\beta_0,0,\beta_2, \beta_3,0)\in\mathbb{N}^5$ with $|\beta|\leq m-1$.
Then it follows from \eqref{es3} with $\alpha=\beta+\mathbf{e}_3$ and $\alpha=\beta+\mathbf{e}_4$  that
\begin{align}
	\mathcal{E}_{\beta+\mathbf{e}_3}(t)+
	\mathcal{E}_{\beta+\mathbf{e}_4}(t)
	\leq
	 \mathcal{T}+
	C(K)\mathcal{M}(t)
,
	\label{es4}
\end{align}
where
\begin{align}
\label{T.cal:def}
 \mathcal{T} := \mathcal{R}_{\beta+\mathbf{e}_3}(t)
 +\mathcal{R}_{\beta+\mathbf{e}_4}(t)=
	-2\epsilon \int_{\Sigma} \mathring{\zeta}
	\left(\mathrm{D}_*^{\beta} \p_2\psi\mathrm{D}_*^{\beta} \p_2{w_4}
	+	\mathrm{D}_*^{\beta} \p_3\psi\mathrm{D}_*^{\beta} \p_3{w_4}\right).
\end{align}

To estimate the instant boundary integral $\mathcal{T}$, we shall use the identities \eqref{H.cons3} and \eqref{h.cons3}, which can be rewritten as
\begin{align}
	\nonumber
	\begin{bmatrix}
		\mathring{h}_2 & \mathring{h}_3\\[0.5mm]
		\mathring{H}_2 & \mathring{H}_3
	\end{bmatrix}
	\begin{bmatrix}
		\p_2\psi \\[0.8mm] \p_3\psi
	\end{bmatrix}
	=\begin{bmatrix}
		w_1 \\[0.8mm] W_5
	\end{bmatrix}
	+\psi \mathring{\rm c}_1
	\qquad\textrm{on }\Sigma_{T}.
\end{align}
{It follows from the non-collinearity condition \eqref{non-coll} and the first two constraints in \eqref{bas3} that}
\begin{align} \label{stab1.3D}
	\big|\mathring{h}_2\mathring{H}_3  - \mathring{h}_3\mathring{H}_2\big|\geq \kappa_0>0  \qquad\textrm{ on }\Sigma_{T}
\end{align}
for some constant $\kappa_0>0$, which implies
\begin{align}
	\p_2\psi =\frac{\mathring{H}_3 w_1 -\mathring{h}_3W_5}{\mathring{h}_2\mathring{H}_3 -\mathring{h}_3\mathring{H}_2}+\psi \mathring{\rm c}_1,
	\quad
	\p_3\psi =\frac{\mathring{h}_2W_5-\mathring{H}_2 w_1}{\mathring{h}_2\mathring{H}_3 -\mathring{h}_3\mathring{H}_2}+\psi \mathring{\rm c}_1
	\qquad\textrm{on }\Sigma_{T}.
	\label{psi:id1.3D}		
\end{align}
Plug \eqref{psi:id1.3D} into \eqref{T.cal:def} to get
\begin{align}
	\mathcal{T}	= 		-2\epsilon \sum_{j=1}^{4} \int_{\Sigma} T_j		+\int_{\Sigma} R_4,
	\label{Q:es2}
\end{align}
where
\begin{alignat*}{5}
	&T_1:=\mathring{\xi} \mathring{H}_3 \mathrm{D}_*^{\beta}w_1  	\mathrm{D}_*^{\beta} \p_2{w_4},
	\qquad &&
	T_2:=-\mathring{\xi}
	\mathring{h}_3 \mathrm{D}_*^{\beta}W_5
	\mathrm{D}_*^{\beta} \p_2{w_4},\\[0.5mm]
	&T_3:=\mathring{\xi} \mathring{h}_2\mathrm{D}_*^{\beta}W_5
	\mathrm{D}_*^{\beta} \p_3{w_4},
	\qquad &&
	T_4:=-\mathring{\xi} 	 \mathring{H}_2\mathrm{D}_*^{\beta}w_1
	\mathrm{D}_*^{\beta} \p_3{w_4},
\end{alignat*}
with
\begin{align} \label{xi.ring:def}
	\mathring{\xi}:=\frac{\mathring{\zeta}}{\mathring{h}'\times \mathring{H}'}
	=\frac{\mathring{e}_1+\epsilon \mathring{v}_2\mathring{h}_3-\epsilon \mathring{v}_3\mathring{h}_2}{\mathring{h}_2\mathring{H}_3 -\mathring{h}_3\mathring{H}_2},
\end{align}
and
$
	R_4:= \big(\mathring{\rm c}_0
	\big[\mathrm{D}_*^{\beta},\,\mathring{\rm c}_0\big]
	\big(w_1, W_5\big)   +
 \mathring{\rm c}_0\mathrm{D}_*^{\beta}\big(\psi \mathring{\rm c}_1\big)\big)\mathrm{D}_*^{\beta}\nabla'{w_4}.
$
Applying integration by parts  and using \eqref{Wnc:es}, we infer
\begin{align}
 \int_{\Sigma} R_4 &=\int_{\Sigma_t} \p_t R_4 \nonumber\\[1mm]
& 	\lesssim_K
\big\| \big([\mathrm{D}_*^{\beta},\mathring{\rm c}_0]\widetilde{\mathcal{W}}_{\rm nc},\,
\mathrm{D}_*^{\beta}\big(\psi \mathring{\rm c}_1\big)
\big) \big\|_{H^{3/2}(\Sigma_{t})}^2
+\big\| \mathrm{D}_*^{\beta}\mu \big\|_{H^{1/2}(\Sigma_{t})}^2	
\lesssim_K \mathcal{M}(t), \label{R4:es}
\end{align}
where $\mathcal{M}(t)$ is given in \eqref{R3:es}.

The other terms in \eqref{Q:es2} shall be reduced by applying the argument of passing to the volume integral (cf.~\cite[pp.~273--274]{T09ARMAMR2481071}).
Precisely, we have
\begin{align}
	-2\epsilon   \int_{\Sigma} T_j
	=2\epsilon \int_{\Omega} \chi_{\varepsilon}(x_1)T_{j+4}+\int_{\Omega}R_5,
	\label{Q:es3}
\end{align}
for $j=1,\ldots,4$,
where
$\chi_{\varepsilon}(x_1):=\chi(x_1/\varepsilon)$
with $\chi$ satisfying \eqref{chi:def},
$\varepsilon>0$ is some small constant to be determined later,
\begin{align*}
	&T_5:= {\mathring{\xi}\mathring{H}_3} \left(\mathrm{D}_*^{\beta}\p_1 w_1  	\mathrm{D}_*^{\beta} \p_2{w_4}	 -  \mathrm{D}_*^{\beta} \p_2w_1  	\mathrm{D}_*^{\beta} \p_1{w_4}\right),\\[1mm]
	&T_6:= -{\mathring{\xi}\mathring{h}_3} \left(\mathrm{D}_*^{\beta}\p_1 W_5	\mathrm{D}_*^{\beta} \p_2{w_4}	 -  \mathrm{D}_*^{\beta} \p_2W_5	\mathrm{D}_*^{\beta} \p_1{w_4}\right),\\[1mm]
	&T_7:= {\mathring{\xi}\mathring{h}_2} \left(\mathrm{D}_*^{\beta}\p_1 W_5  	\mathrm{D}_*^{\beta} \p_3{w_4}	 -  \mathrm{D}_*^{\beta} \p_3W_5  	\mathrm{D}_*^{\beta} \p_1{w_4}\right),\\[1mm]
	&T_8:=- {\mathring{\xi}\mathring{H}_2} \left(\mathrm{D}_*^{\beta}\p_1 w_1  	\mathrm{D}_*^{\beta} \p_3{w_4}	 -  \mathrm{D}_*^{\beta} \p_3w_1  	\mathrm{D}_*^{\beta} \p_1{w_4}\right).
\end{align*}
and $R_5:= \mathring{\rm c}_1\mathrm{D}_*^{\beta}\big(w_1, W_5\big)\mathrm{D}_*^{\beta}\nabla{w_4}$.
Thanks to \eqref{e.inv3}, we can utilize integration by parts and Moser-type calculus inequalities to derive 
\begin{align}
	\int_{\Omega}R_5
=\int_{\Omega_t}\p_t \Big(\mathring{\rm c}_1\mathrm{D}_*^{\beta}\big(w_1, W_5\big)\mathrm{D}_*^{\beta}\nabla'\big({w_4}, \p_1\mathring{\Phi}^{-}e'\big)\Big)
\lesssim_K \mathcal{M}(t). \label{R5:es}
\end{align}
Substituting the estimates \eqref{R4:es}--\eqref{R5:es} into \eqref{Q:es2} implies
\begin{align} \label{Q:es4}
	\mathcal{T}	\leq C(K)\mathcal{M}(t)+	2\epsilon \sum_{j=9}^{12} \int_{\Omega} \chi_{\varepsilon}(x_1)T_j,
\end{align}
where
\begin{align}
	\nonumber 	T_9:=\; &{\mathring{\xi} } \mathrm{D}_*^{\beta}\p_1 w_1
	\left({ \mathring{H}_3}  \mathrm{D}_*^{\beta}\p_2 {w_4}
	-{ \mathring{H}_2}  \mathrm{D}_*^{\beta}\p_3 {w_4} \right),
	\\[1mm]
	\nonumber 	T_{10}:=\; &{\mathring{\xi} } \mathrm{D}_*^{\beta}\p_1 {w_4}
	\left(
	{ \mathring{H}_2}  \mathrm{D}_*^{\beta}\p_3 w_1
	-{ \mathring{H}_3}  \mathrm{D}_*^{\beta}\p_2 w_1 \right),
	\\[1mm]
	\nonumber 	T_{11}:=\; &{\mathring{\xi} } \mathrm{D}_*^{\beta}\p_1 {w_4}
	\left({ \mathring{h}_3}  \mathrm{D}_*^{\beta}\p_2 W_5
	-{ \mathring{h}_2}  \mathrm{D}_*^{\beta}\p_3 W_5 \right),
	\\[1mm]
	\nonumber 	T_{12}:=\; &{\mathring{\xi} } \mathrm{D}_*^{\beta}\p_1 W_5
	\left(
	{ \mathring{h}_2}  \mathrm{D}_*^{\beta}\p_3 {w_4}
	-{ \mathring{h}_3}  \mathrm{D}_*^{\beta}\p_2 {w_4}
	\right).
\end{align}

We shall identify a quantitative stability condition
, under which
the last term in \eqref{Q:es4} can be mainly absorbed by the instant energy functionals in \eqref{es4}.
To this end, we use the constraints \eqref{H.inv3}--\eqref{e.inv3} and integration by parts to reformulate the last term in \eqref{Q:es4}.
For instance,
we utilize \eqref{h.inv3} and ${w_4}=\mathring{N}\cdot e$ to obtain
\begin{align} \nonumber
& \int_{\Omega} \chi_{\varepsilon}(x_1){\mathring{\xi} } \mathrm{D}_*^{\beta}\p_1 w_1{ \mathring{H}_3}  \mathrm{D}_*^{\beta}\p_2 {w_4}
	 =-	  \int_{\Omega} \chi_{\varepsilon}(x_1){\mathring{\xi} }{ \mathring{H}_3}   \mathrm{D}_*^{\beta}\nabla'\cdot(\p_1\mathring{\Phi}^-h')\mathrm{D}_*^{\beta}\p_2 {w_4}\\[1mm]
&\qquad =	-	   \int_{\Omega} \chi_{\varepsilon}(x_1){\mathring{\xi} }{ \mathring{H}_3}   \mathrm{D}_*^{\beta}\p_2(\p_1\mathring{\Phi}^-h') \cdot \mathrm{D}_*^{\beta}\nabla' (\mathring{N}_i e_i)
+\int_{\Omega} \mathring{\rm c}_1\mathrm{D}_*^{\beta}(\p_1\mathring{\Phi}^-h_3)\mathrm{D}_*^{\beta}\nabla' {w_4}
\nonumber\\[1mm]
&\qquad =
-	   \int_{\Omega} \chi_{\varepsilon}(x_1){\mathring{\xi} }{ \mathring{H}_3} \p_1\mathring{\Phi}^-  \mathrm{D}_*^{\beta}\p_2 h'  \cdot \big(\mathring{N}_i \mathrm{D}_*^{\beta}\nabla'  e_i\big)
+\int_{\Omega} \mathring{\rm c}_1\mathrm{D}_*^{\beta}(\p_1\mathring{\Phi}^-h_3)\mathrm{D}_*^{\beta}\nabla' {w_4}
\nonumber\\[1mm]
&\hspace*{3em}+  \int_{\Omega}\mathring{\rm c}_1[\mathrm{D}_*^{\beta}\nabla', \mathring{\rm c}_1]\big(h,e\big)\mathrm{D}_*^{\beta}\nabla' \big(\mathring{\rm c}_1(h,e)\big).
\nonumber
\end{align}
Similarly, from \eqref{H.inv3}--\eqref{e.inv3} and $({w_4}, w_1,  W_5)=(\mathring{N}\cdot e,\, \mathring{N}\cdot h,\, \mathring{N}\cdot H)$, we can deduce
\begin{align}
	 \int_{\Omega} \chi_{\varepsilon}(x_1)T_j=\;&
	  \int_{\Omega} \chi_{\varepsilon}(x_1)T_{j+4}
	 +\int_{\Omega}\mathring{\rm c}_1\mathrm{D}_*^{\beta}\big(\mathring{\rm c}_1(h,e,H)\big)\mathrm{D}_*^{\beta}\nabla' \big({w_4},w_1,W_5\big)\nonumber
	 \\[1mm]
	&+  \int_{\Omega}\mathring{\rm c}_1[\mathrm{D}_*^{\beta}\nabla', \mathring{\rm c}_1]\big(h,e,H\big)\mathrm{D}_*^{\beta}\nabla' \big(\mathring{\rm c}_1(h,e,H)\big),
	 \label{Q:es5}
\end{align}
for $j=9,\ldots,12$, where
\begin{align}
	\nonumber 	T_{13}:=\; &-{\mathring{\xi} }\p_1\mathring{\Phi}^- \left(\mathring{N}_i  \mathrm{D}_*^{\beta}\nabla' e_i\right)\cdot
	\left({ \mathring{H}_3}  \mathrm{D}_*^{\beta}\p_2  h'
	-{ \mathring{H}_2}  \mathrm{D}_*^{\beta}\p_3  h' \right),
	\\[1mm]
	\nonumber 	T_{14}:=\; &-{\mathring{\xi} } \p_1\mathring{\Phi}^- \left(\mathring{N}_i  \mathrm{D}_*^{\beta}\nabla' h_i\right)\cdot
	\left(
	{ \mathring{H}_2}  \mathrm{D}_*^{\beta}\p_3  e'
	-{ \mathring{H}_3}  \mathrm{D}_*^{\beta}\p_2 e' \right),
	\\[1mm]
	\nonumber 	T_{15}:=\; &-{\mathring{\xi} } \p_1\mathring{\Phi}^- \left(\mathring{N}_i  \mathrm{D}_*^{\beta}\nabla' H_i\right)\cdot
	\left({ \mathring{h}_3}  \mathrm{D}_*^{\beta}\p_2   e'
	-{ \mathring{h}_2}  \mathrm{D}_*^{\beta}\p_3   e'  \right),
	\\[1mm]
	\nonumber 	T_{16}:=\; &-{\mathring{\xi} } \p_1\mathring{\Phi}^+\left(\mathring{N}_i \mathrm{D}_*^{\beta}\nabla' e_i\right)\cdot
	\left(
	{ \mathring{h}_2}  \mathrm{D}_*^{\beta}\p_3    H'
	-{ \mathring{h}_3}  \mathrm{D}_*^{\beta}\p_2    H'
	\right).
\end{align}
Similar to \eqref{R5:es},
the last two terms in \eqref{Q:es5} can be bounded by $C(K)\mathcal{M}(t)$.
A direct computation gives
\begin{align}
	\sum_{j=9}^{12}\big|T_{j+4}\big|\leq
	 \max\big|\p_1\mathring{\Phi}^{\pm}\big|\big|{\mathring{\xi} (  \mathring{H}' ,\,  \mathring{h}')}\big|\big|\mathring{N} \big|\big| \mathrm{D}_*^{\beta}\nabla' (H, h, e)\big|^2
	 ,
	\nonumber
\end{align}
which along with \eqref{Q:es4}, \eqref{Q:es5}, and \eqref{es4} yields
\begin{align}
&	\mathcal{E}_{\beta+\mathbf{e}_3}(t)+
\mathcal{E}_{\beta+\mathbf{e}_4}(t)\nonumber \\[1mm]
&\quad	\leq
	C(K)\mathcal{M}(t) +
2\epsilon   \int_{\Omega} \chi_{\varepsilon}(x_1)
	 \max\big|\p_1\mathring{\Phi}^{\pm}\big|\big|{\mathring{\xi} (  \mathring{H}' ,\,  \mathring{h}')}\big|\big|\mathring{N} \big|\big| \mathrm{D}_*^{\beta}\nabla' (H, h, e)\big|^2
	.
	\label{es4b}
\end{align}

To control the last term in \eqref{es4b}, we make an explicit lower bound for the instant energy functional $\mathcal{E}_{\alpha}(t)$ defined in \eqref{es1}.
It follows from \eqref{W:def}, \eqref{w:def2}, \eqref{A.bm:def}, \eqref{A.sf:def}, and \eqref{B.def}--\eqref{S+U} that
\begin{align}
	\mathcal{E}_{\alpha}(t)
	=\;&
	\int_{\Omega} \left(\bm{A}_0^+\mathrm{D}_*^{\alpha}(\mathring{J}_1^{-1}U)\cdot \mathrm{D}_*^{\alpha}(\mathring{J}_1^{-1}U)
	+ \bm{A}_0^-\mathrm{D}_*^{\alpha}(\mathring{J}_3^{-1}\mathring{J}_2^{-1}u)\cdot \mathrm{D}_*^{\alpha}(\mathring{J}_3^{-1}\mathring{J}_2^{-1}u)\right) \nonumber\\[1mm]
	=\;&
	\epsilon \int_{\Omega} \left(\varGamma(\mathring{v})^{-1}S_{+}(\mathring{U})\mathrm{D}_*^{\alpha}U\cdot \mathrm{D}_*^{\alpha}U
	+  S_{-}(\mathring{\nu})\mathrm{D}_*^{\alpha}u\cdot \mathrm{D}_*^{\alpha}u\right)
+\int_{\Omega}R_6
\nonumber	\\[1mm]
=\; &
\epsilon \int_{\Omega} \left(\varGamma(\mathring{v})^{-1} B_{0}(\mathring{V})\mathrm{D}_*^{\alpha}V\cdot \mathrm{D}_*^{\alpha}V
+  S_{-}(\mathring{\nu})\mathrm{D}_*^{\alpha}u\cdot \mathrm{D}_*^{\alpha}u\right)
+\int_{\Omega}R_6,
\nonumber 
\end{align}
with $R_6:=\mathring{\rm c}_1\mathrm{D}_*^{\alpha}\big(\mathring{\rm c}_1(W,w) )\big[\mathrm{D}_*^{\alpha},\mathring{\rm c}_1\big] (W,w)$,
where
$U=(q,v,H,S)$, $u=(h,e)$,
$V=(p,\varGamma v, H,S)$, and $\mathring{V}=V(\mathring{U})$.
By virtue of \eqref{B.def}, we discover
\begin{align}
 \mathcal{E}_{\alpha}(t)
 \geq \; &	 \epsilon \int_{\Omega}  \varGamma(\mathring{v})^{-1}\mathcal{M}_{0}(\mathring{v})\mathrm{D}_*^{\alpha}H\cdot \mathrm{D}_*^{\alpha}H
 +  \epsilon \int_{\Omega} S_{-}(\mathring{\nu})\mathrm{D}_*^{\alpha}u\cdot \mathrm{D}_*^{\alpha}u  \nonumber \\[1mm]
&+ c(K)\big\|\mathrm{D}_*^{\alpha}(p,v,S)(t)\big\|^2_{L^2(\Omega)}
	-C(K)\mathcal{M}(t).
	\label{E.cal:es3}
\end{align}

We compute that the eigenvalues of the positive definite matrix $\mathcal{M}_{0}(\mathring{v})$ are
\begin{alignat*}{3}
	& \varGamma(\mathring{v})^{-1}=\sqrt{1-\epsilon^2|\mathring{v}|^2}\qquad&&\textrm{with multiplicity }2,\\[1mm]
	& \big(1+\epsilon^2 \varGamma(\mathring{v})^2|\mathring{v}|^2 \big)\varGamma(\mathring{v})^{-1}\qquad&&\textrm{with multiplicity }1,
\end{alignat*}
and the symmetrizer $S_{-}(\mathring{\nu})$ defined by \eqref{S-3D} has double eigenvalues
$1$ and $1\pm \epsilon|\mathring{\nu}|$.
Applying the spectral theorem for symmetric matrices (see, e.g., \cite[Theorem 8.6.10]{A91MR1129886}) to \eqref{E.cal:es3}, we discover
\begin{align}
 		\mathcal{E}_{\beta+\mathbf{e}_3}(t)+
	\mathcal{E}_{\beta+\mathbf{e}_4}(t)
	\geq \; &	\epsilon \int_{\Omega} \big(1-\epsilon^2|\mathring{v}|^2\big)\big|\mathrm{D}_*^{\beta}\nabla' H\big|^2
	+ \epsilon \int_{\Omega} \big(1-\epsilon|\mathring{\nu}|\big) \big|\mathrm{D}_*^{\beta}\nabla'u\big|^2
	\nonumber \\[1mm]
&
	+ c(K)\big\|\mathrm{D}_*^{\beta}\nabla'(p,v,S)(t)\big\|^2_{L^2(\Omega)}
	-C(K)\mathcal{M}(t)
	\label{E.cal:es4}
\end{align}
for all $\langle\beta\rangle\leq m-1$.
Plugging \eqref{E.cal:es4}  into  \eqref{es4b}
and using the continuity of the basic state $(\mathring{U},\mathring{u},\mathring{\varphi})$,
we can choose $\varepsilon>0$ small enough to absorb the last term in \eqref{es4b} and conclude
\begin{align}
	\big\|\mathrm{D}_*^{\beta}\nabla'(U,u)(t)\big\|_{L^2(\Omega)}^2
	\lesssim_{K} \mathcal{M}(t)
	\label{es5}
\end{align}
for all $\beta=(\beta_0,0,\beta_2, \beta_3,0)\in\mathbb{N}^5$ with $|\beta|\leq m-1$,
provided the stability conditions \eqref{non-coll} and \eqref{stability3D} hold.

	\subsection{Estimate of $\mathcal{R}_{\alpha}$ with time derivatives}\label{sec:Qes2}
This subsection is devoted to showing the estimate of the term $\mathcal{R}_{\alpha}(t)$ given in \eqref{R.cal:def} for $\alpha=(\alpha_0,0,\alpha_2, \alpha_3,0)$ with $\alpha_0\geq 1$ and $|\alpha|\leq m$.
For this purpose,
we combine the identities \eqref{psi:id1.3D} with the first condition in \eqref{ELP4c} to get
\begin{align}
	\mathrm{D}_{\rm tan }\psi
	:=(\p_t, \p_2,\p_3)\psi
	= \mathring{\rm c}_0 (W_2,W_5,w_1)
	+\psi \mathring{\rm c}_1
	\qquad\textrm{on }\Sigma_{T}.
	\label{psi:id2.3D}
\end{align}
Setting $\mathbf{e}_1:=(1,0,0,0,0)$ and $\beta=\alpha-\mathbf{e}_1$,
we use \eqref{cancellation} and \eqref{psi:id2.3D} to infer
\begin{align}
	\mathcal{R}_{\beta+\mathbf{e}_1}(t)=\; & 2 \int_{\Sigma} \mathring{\zeta}
	  \mathrm{D}_*^{\beta} \big(\mathring{\rm c}_0 (W_2,W_5,w_1)
	  +\psi \mathring{\rm c}_1
	  \big)\mathrm{D}_*^{\beta} \big(\p_3\mu_2-\p_2\mu_3\big)\nonumber\\[1mm]
=\; &   \int_{\Sigma} \mathring{\rm c}_1\mathrm{D}_*^{\beta} \widetilde{\mathcal{W}}_{\rm nc}\mathrm{D}_*^{\beta} \nabla'\mu
+\int_{\Sigma} \Big(\mathring{\rm c}_1\big[\mathrm{D}_*^{\beta},\,\mathring{\rm c}_0  \big]\widetilde{\mathcal{W}}_{\rm nc}
+  \mathring{\rm c}_1\mathrm{D}_*^{\beta} \big( \psi \mathring{\rm c}_1\big)\Big)
\mathrm{D}_*^{\beta} \nabla'\mu.
\label{Qt:es1}
\end{align}
Similar to \eqref{R4:es},
the last term in \eqref{Qt:es1} can be bounded by $C(K)\mathcal{M}(t)$.
Passing the first term on the right-hand side of \eqref{Qt:es1} to the volume integral and using integration by parts, we infer
\begin{align}
	\int_{\Sigma} \mathring{\rm c}_1\mathrm{D}_*^{\beta} \widetilde{\mathcal{W}}_{\rm nc}\mathrm{D}_*^{\beta} \nabla'\mu
	=\int_{\Omega} T_{17}+\int_{\Omega}R_7,
	\label{Qt:es2}
\end{align}
where
\begin{align}
	T_{17}:=\;&\mathring{\rm c}_1\mathrm{D}_*^{\beta} \p_1 \widetilde{\mathcal{W}}_{\rm nc}\mathrm{D}_*^{\beta} \nabla'\mu
	+ \mathring{\rm c}_1\mathrm{D}_*^{\beta} \nabla' \widetilde{\mathcal{W}}_{\rm nc}\mathrm{D}_*^{\beta} \p_1 \mu,\nonumber\\[1mm]
	R_7:=\;& \mathring{\rm c}_2\mathrm{D}_*^{\beta}   \widetilde{\mathcal{W}}_{\rm nc}\mathrm{D}_*^{\beta} \nabla'\mu
	+ \mathring{\rm c}_2\mathrm{D}_*^{\beta} \widetilde{\mathcal{W}}_{\rm nc}\mathrm{D}_*^{\beta} \p_1 \mu.
	\nonumber
\end{align}
Thanks to \eqref{W.cal:id} and \eqref{Qt:es2}, we have
\begin{align}
&\int_{\Sigma} \mathring{\rm c}_1\mathrm{D}_*^{\beta} \widetilde{\mathcal{W}}_{\rm nc}\mathrm{D}_*^{\beta} \nabla'\mu
	\lesssim_K
	\mathcal{M}(t)+
	 \int_{\Omega} \big|\mathrm{D}_*^{\beta} \mathrm{D}_*\big( W,  w \big)\big|
	 \big|\mathrm{D}_*^{\beta} \nabla'\big(  W,  w \big)\big|
	 \nonumber\\
&\qquad 	 \lesssim_K
	 	 \mathcal{M}(t)+
	 \varepsilon \int_{\Omega} \big|\mathrm{D}_*^{\beta} \mathrm{D}_*\big( W,  w \big)\big|^2
	 +C(\varepsilon)
	 \int_{\Omega} \big|\mathrm{D}_*^{\beta} \nabla'\big( W,  w \big)\big|^2
	 \label{Qt:es3}
\end{align}
for any $\varepsilon>0$.
Plug \eqref{Qt:es1}--\eqref{Qt:es3} into \eqref{es3} with $\alpha=\beta+\mathbf{e}_1$ and use
the estimate \eqref{es5} for space tangential derivatives to derive
\begin{align}
\mathcal{E}_{\beta+\bf{e}_1}(t)
	\lesssim_K
	C(\varepsilon)\mathcal{M}(t)+
	\varepsilon \int_{\Omega} \big|\mathrm{D}_*^{\beta} \mathrm{D}_*\big( W,  w \big)\big|^2
\nonumber
\end{align}
for all $\beta=(\beta_0,0,\beta_2, \beta_3,0)\in\mathbb{N}^5$ with $|\beta|\leq m-1$.
Combining the last estimate with \eqref{es2} and \eqref{es5}, we take $\varepsilon>0$ small enough to obtain
\begin{align}
		\sum_{\langle \alpha\rangle \leq m}	\|(\mathrm{D}_*^{\alpha}W,\, \mathrm{D}_*^{\alpha}w)(t)\|_{L^2(\Omega)}^2
	\lesssim_K  \mathcal{M}(t). \label{es6}
\end{align}

\subsection{Estimate of front $\psi$} \label{sec.psi}
In view of \eqref{psi:id2.3D}, we get 
\begin{align}
	& \|\psi\|_{H^{m+1/2}(\Sigma_{t})}^2 	\lesssim 	\|\psi\|_{L^{2}(\Sigma_{t})}^2 	+	\|\mathring{\rm c}_1\widetilde{\mathcal{W}}_{\rm nc}\|_{H^{m-1/2}(\Sigma_{t})}^2  +\|\psi \|_{H^{m-1/2}(\Sigma_{t})}^2.
	\label{es:psi1.3D}
\end{align}
To estimate the last term in  \eqref{es:psi1.3D}, we apply $\mathrm{D}_{\rm tan}^{\beta}:=\p_t^{\beta_0}\p_2^{\beta_2}\p_3^{\beta_3}$ 
to the first condition in \eqref{ELP4c} and multiply the resulting equation with $\mathrm{D}_{\rm tan}^{\beta}\psi$ to obtain
\begin{align}
	\|\mathrm{D}_{\rm tan}^{\beta}\psi(t)\|_{L^{2}(\Sigma)}^{2}
	\lesssim_K	\big\|\big(\mathrm{D}_{\rm tan}^{\beta}\psi , \, \mathrm{D}_{\rm tan}^{\beta}W_2,\,
	[\mathrm{D}_{\rm tan}^{\beta},\mathring{v}'\cdot\nabla']\psi, \, \mathrm{D}_{\rm tan}^{\beta}(\mathring{b}_1\psi) \big) \big\|_{L^{2}(\Sigma_t)}^{2}  \nonumber
\end{align}
which leads to
\begin{align}
	\|\mathrm{D}_{\rm tan}^{\beta}\psi(t)\|_{L^{2}(\Sigma)}^{2}
	\lesssim_K
		& {\|}(W_2,\psi){\|}_{H^{|\beta|}(\Sigma_{t})}^2
		+\mathring{\rm C}_{m+4}  \|\psi\|_{L^{\infty}(\Sigma_{t})}^2
		\ \  &&  		\textrm{if }|\beta|\leq m
		.
	\label{es:psi2a}
\end{align}
Then it follows from Gr\"{o}nwall's inequality that
\begin{align}
	{\|}\psi{\|}_{H^{s}(\Sigma_{t})}^2 	 \lesssim_K	 t\mathrm{e}^{C(K)t} \left({\|}W_2{\|}_{H^{s}(\Sigma_{t})}^2
	+\mathring{\rm C}_{m+4}  \|\psi\|_{L^{\infty}(\Sigma_{t})}^2\right)
	\label{es:psi2.3D}
\end{align}
for $s=0,1,\ldots,m$. Interpolating between \eqref{es:psi2.3D} with $s=m-1$ and $s=m$ gives
\begin{align}
	{\|}\psi{\|}_{H^{m-1/2}(\Sigma_{t})}^2 	 \lesssim_K 	t\mathrm{e}^{C(K)t}\left({\|}W_2{\|}_{H^{m-1/2}(\Sigma_{t})}^2
	+ \mathring{\rm C}_{m+4}  \|\psi\|_{L^{\infty}(\Sigma_{t})}^2\right) ,
\nonumber 
\end{align}
which along with \eqref{es:psi1.3D} yields
\begin{align}
	\|\psi\|_{H^{m+1/2}(\Sigma_{t})}^2
	&\lesssim_K  t\mathrm{e}^{C(K)t}
	\left(\big\| \widetilde{\mathcal{W}}_{\rm nc} \big\|_{H^{m-1/2}(\Sigma_{t})}^2+
	\mathring{\rm C}_{m+4}  \big\|\big(\widetilde{\mathcal{W}}_{\rm nc},\, \psi\big)\big\|_{L^{\infty}(\Sigma_{t})}^2\right),
	\label{es:psi4.3D}
\end{align}
where $\widetilde{\mathcal{W}}_{\rm nc}$ is defined by \eqref{W.cal:nc}.

	\subsection{Proof of Theorem \ref{thm:main1}} \label{sec.Thm2.2}
	Applying induction argument to
\eqref{h.inv3}--\eqref{e.inv3} and \eqref{d1w.3D} and using \eqref{es6}, we can deduce
\begin{align}\label{es:Ww}
	\mathcal{I}_m(t):=\sum_{\langle \alpha\rangle \leq m}	\| \mathrm{D}_*^{\alpha}W (t)\|_{L^2(\Omega)}^2
	+\sum_{|\alpha| \leq m}	\|  \mathrm{D}^{\alpha}w(t)\|_{ L^2(\Omega)}^2
	\lesssim_K  \mathcal{M}(t),
\end{align}
for any integer $m\geq 1$, where $\mathcal{M}(t)$ is given in \eqref{R3:es} (cf.~\eqref{M1cal:def} and \eqref{M2cal:def}).
Applying Gr\"{o}nwall's inequality to \eqref{es:Ww}, we obtain
\begin{align} \nonumber
	 \mathcal{I}_m(t) \lesssim_K \mathcal{N}_m(t),
\end{align}
which implies
\begin{align}
\label{es:pr1}
	\|W\|_{H_*^m(\Omega_{T})}^2	+	\|w\|_{H^m(\Omega_{T})}^2
	\lesssim_K
	T\mathrm{e}^{C(K)T}\mathcal{N}_m(T),
\end{align}
where 
\begin{align}\nonumber
	\mathcal{N}_m(t)
	:=	\;&
	{\|} \bm{f}^+ {\|}_{H_*^m(\Omega_{t})}^2+\|\psi\|_{H^{m+1/2}(\Sigma_{t})}^2\\[1mm]
	&\nonumber	+\mathring{\rm C}_{m+4}  \Big(\|(\bm{f}^+,W,w)\|_{W_*^{2,\infty}(\Omega_{t})}^2+\|\psi\|_{L^{\infty}(\Sigma_{t})}^2\Big)
.
\end{align}	
Plug \eqref{es:psi4.3D}
and \eqref{Wnc:es} into \eqref{es:pr1}, take $T>0$ small enough, and
use the embedding $H_*^6(\Omega_{T})\hookrightarrow W_*^{2,\infty}(\Omega_{T})$ (cf.~\cite[Lemma 3.3]{TW21MR4201624}) to obtain
\begin{align}
	\nonumber
&	\|W\|_{H_*^m(\Omega_{T})}^2	+	\|w\|_{H^m(\Omega_{T})}^2	+ \|\psi \|_{H^{m+1/2}(\Sigma_{T})}^2
+\big\| \widetilde{\mathcal{W}}_{\rm nc} \big\|_{H^{m-1/2}(\Sigma_{t})}^2
\\[1mm]
\nonumber	
&\quad \lesssim_K
	T{\|} \bm{f}^+ {\|}_{H_*^m(\Omega_{T})}^2
	+T\mathring{\rm C}_{m+4} \Big(\|(\bm{f}^+,W,w)\|_{H_*^{6}(\Omega_{T})}^2+\|\psi\|_{H^2(\Sigma_{T})}^2\Big)
.
\end{align}
By virtue of the last estimate with $m=6$ and the constraint \eqref{bas1}, we can find a small constant $T_0>0$, depending on $K$, such that if $0<T\leq T_0$, then
\begin{align}
	\nonumber
 \|W\|_{H_*^6(\Omega_{T})}^2	+	\|w\|_{H^6(\Omega_{T})}^2	+ \|\psi \|_{H^{13/2}(\Sigma_{T})}^2
  \lesssim_K
  {\|} \bm{f}^+ {\|}_{H_*^6(\Omega_{T})}^2.
\end{align}
Then we infer
\begin{align}\nonumber
 &\|W\|_{H_*^m(\Omega_{T})}^2	+	\|w\|_{H^m(\Omega_{T})}^2	+ \|\psi \|_{H^{m+1/2}(\Sigma_{T})}^2
 +\big\| \widetilde{\mathcal{W}}_{\rm nc} \big\|_{H^{m-1/2}(\Sigma_{t})}^2
 \\
&\hspace*{9em}\lesssim_K
{\|} \bm{f}^+ {\|}_{H_*^m(\Omega_{T})}^2+\mathring{\rm C}_{m+4} \| \bm{f}^+ \|_{H_*^{6}(\Omega_{T})}^2
\qquad\textrm{for }m\geq 6.
\label{es:pr}
\end{align}

It suffices to show the estimate \eqref{tame.es} for the problem \eqref{ELP1}.
To this end, we derive from \eqref{ELP2a},
\eqref{decom2.3D}, and \eqref{W:def} that
\begin{align*}
	\big\|\tilde{f}^+\big\|^2_{H^m_*(\Omega_{t})}
	\lesssim_K \big\| {f}^+\big\|^2_{H^m_*(\Omega_{t})}
	+\big\| \mathring{\rm c}_2 \big(U^{\natural},\, \mathrm{D}_*U^{\natural},\, \p_1q^{\natural},\, \p_1v^{\natural}   \big)\big\|^2_{H^m_*(\Omega_{t})}.
\end{align*}
Then we can utilize the above estimates and Lemma \ref{lem:homo} to obtain 
\eqref{tame.es}.
This completes the proof of Theorem \ref{thm:main1}.

 \vspace*{2mm}

\section{Nonlinear existence in 2D}	 \label{sec:2D}
In this section, we 
first establish the well-posedness and higher-order tame estimates to the linearized problem \eqref{ELP1} in 2D. 
Then we prove the solvability of the nonlinear problem \eqref{NP1} in 2D through a suitable Nash--Moser iteration; see \cite{AG07MR2304160, S16MR3524197} for a general presentation of this method.

\subsection{2D linear well-posedness}
To solve the problem \eqref{ELP4},
we write the boundary conditions \eqref{ELP4c} in a different form
by eliminating the first derivatives of $\psi$ from the third condition in \eqref{ELP4c} (cf.~\eqref{B.bm2D:def}).
More precisely, using the first condition in \eqref{ELP4c},
we find that the problem \eqref{ELP4} is equivalent to that consisting of
\eqref{ELP4a}, \eqref{ELP4b}, \eqref{ELP4d}, and
\begin{alignat}{3}
\label{BC.2Da}	
&W_2=(\p_t+\mathring{v}_2\p_2+\mathring{b}_1)\psi
\qquad &&\textrm{on }\Sigma_T,\\
\label{BC.2Db}	
&W_1-\mathring{{h}}_2 w_2=-\mathring{b}_2 \psi
&&\textrm{on }\Sigma_T,\\
\label{BC.2Dc}	
&\epsilon \mathring{h}_2 W_2+w_3=\mathring{b}_5 \psi&&\textrm{on }\Sigma_T,
\end{alignat}
where $\mathring{b}_5:=\epsilon \mathring{h}_2 \mathring{b}_1-\mathring{b}_4.$

{We shall prove the existence of solutions to the problem \eqref{ELP4a}, \eqref{ELP4b}, \eqref{ELP4d}, \eqref{BC.2Da}--\eqref{BC.2Dc} by employing a fixed point argument developed in \cite{ST13MR3148595}.}
Indeed, we let $\psi\in H^{3/2}(\Sigma_{T})$ be given and vanish in the past.
Consider the problem \eqref{ELP4a}, \eqref{ELP4b}, \eqref{ELP4d}, \eqref{BC.2Db}--\eqref{BC.2Dc}: 
\begin{subequations}
	\label{FP1}
	\begin{alignat}{3}
		\label{FP1a}
		&\big({\bm{A}}_0^+\p_t  +{ \bm{A}}_1^+\p_1 +{ \bm{A}}_2^+\p_2    +{ \bm{A}}_3^+\big)W = \bm{f}^+
		&\qquad  &\textrm{in }\Omega_T,\\[2mm]
		\label{FP1b}
		& \big({ \bm{A}}_0^-\p_t  +{ \bm{A}}_1^-\p_1 +{ \bm{A}}_2^-\p_2    +{ \bm{A}}_3^-\big) w =0
		&&\textrm{in }\Omega_T,\\[0.5mm]
		\label{FP1c}
&W_1-\mathring{{h}}_2 w_2=-\mathring{b}_2 \psi
&&\textrm{on }\Sigma_T,\\
		\label{FP1c2}
&\epsilon \mathring{h}_2 W_2+w_3=\mathring{b}_5 \psi&&\textrm{on }\Sigma_T,\\[0.5mm]
		\label{FP1d}
		&
		(W,w) =0
		&& \textrm{if } t<0.
	\end{alignat}
\end{subequations}
Once the solutions $(W,w)$ for \eqref{FP1} have been constructed, we take $\phi$ as the solution vanishing in the past of the transport equation 
\begin{align}
	(\p_t+\mathring{v}_2\p_2+\mathring{b}_1)\phi=W_2
	\qquad  \textrm{on }\Sigma_T.
\end{align}
Then we prove that ${\psi}\mapsto \phi$ defines a contraction mapping from $H^{3/2}(\Sigma_{T})$ to itself, which relies on the existence and suitable estimates of solutions for the problems \eqref{FP1} and \eqref{BC.2Da}.

\subsubsection{Existence and estimate for problem \eqref{FP1}}

It follows from \eqref{decom2.2D} that the hyperbolic problem \eqref{FP1} is characteristic of  constant multiplicity
and has correct number of boundary conditions \cite{R85MR0797053,OSY95MR1346224,S95MMASMR1346663,S96aMR1401431,BGS07MR2284507}.
If we consider the boundary conditions \eqref{FP1c} and \eqref{FP1c2} in homogeneous form, that is, if we set ${\psi}=0$, then  the quadratic form
$$Q_0:= \bm{A}_1^+  W\cdot  W
+
\bm{A}_1^-  w\cdot  w= 2\epsilon W_1 W_2
+2w_2 w_3 $$
vanishes on the boundary $\Sigma_T$.
As a result, the boundary conditions \eqref{FP1c} and \eqref{FP1c2} are {\it maximally nonnegative} if $\psi=0$.
The problem \eqref{FP1} can be reduced to that with homogeneous boundary conditions by subtracting from $(W,w)$ a regular function satisfying \eqref{FP1c} and \eqref{FP1c2}.
Therefore, we can apply the results of \cite{S95MMASMR1346663,S96aMR1401431,OSY95MR1346224} to obtain the existence of solutions to the problem \eqref{FP1}.

To show the contraction property of the mapping ${\psi}\mapsto \phi$ described above,
we deduce the $H_*^1\times H^1$ estimate for solutions $(W,w)$ of the problem \eqref{FP1}.

For the 2D problem \eqref{FP1},
we plug \eqref{decom2.2D}--\eqref{decom3} into \eqref{FP1a}--\eqref{FP1b} to get
\begin{alignat}{3}
	\label{d1W}
	&	\begin{bmatrix}
		\epsilon \p_1 W_2\\   \epsilon \p_1 W_1  \\0
	\end{bmatrix}
	=\bm{f}^+-{\bm{A}}_3^+ W-\sum_{i=0,2}{\bm{A}}_i^+\p_i W-{\bm{B}}_0^+ \p_1W
	\quad &&\textrm{in }\Omega_{T}	,\\
	\label{d1w}	
	&	\begin{bmatrix}
		0\\ \p_1 w_3\\  \p_1 w_2
	\end{bmatrix}
	=-{\bm{A}}_3^- w-\sum_{i=0,2}{\bm{A}}_i^-\p_i w-{\bm{B}}_0^- \p_1w
	\quad &&\textrm{in }\Omega_{T}.
\end{alignat}
Using \eqref{decom2.2D}, \eqref{d1W}, and \eqref{d1w},
we can obtain \eqref{es1} and \eqref{R1:es} with $d=2$ and
\begin{align}
	{Q}_{\alpha}=	2\epsilon\mathrm{D}_*^{\alpha}W_1\mathrm{D}_*^{\alpha}W_2
	+2\mathrm{D}_*^{\alpha}w_2\mathrm{D}_*^{\alpha}w_3.
	\label{Q.alpaha}
\end{align}
Let $\alpha=(\alpha_0, 0 ,\alpha_2,\alpha_3)$ with $\langle \alpha\rangle \leq  1$.
Then
$\mathrm{D}_*^{\alpha}:=\p_t^{\alpha_0}\p_2^{\alpha_2} $ with $\alpha_0+\alpha_2\leq 1$
and
\begin{align}
	{Q}_{\alpha}
	=\; & 2\epsilon \mathrm{D}_*^{\alpha} W_2 \mathrm{D}_*^{\alpha}\big(\mathring{h}_2w_2-\mathring{b}_2 {\psi}\big)
	+2\mathrm{D}_*^{\alpha}w_2 \mathrm{D}_*^{\alpha}\big(-\epsilon \mathring{h}_2 W_2+\mathring{b}_5 {\psi}\big)\nonumber\\[1mm]
	=\; &
	\mathring{\rm c}_1
	\big(w_2\mathrm{D}_*^{\alpha} W_2-W_2\mathrm{D}_*^{\alpha} w_2\big)
	-2\epsilon \mathrm{D}_*^{\alpha} W_2 \mathrm{D}_*^{\alpha}\big(\mathring{b}_2 {\psi}\big)
	+2\mathrm{D}_*^{\alpha}w_2 \mathrm{D}_*^{\alpha}\big(\mathring{b}_5 {\psi}\big), \nonumber
\end{align}
due to the boundary conditions \eqref{FP1c} and \eqref{FP1c2}.
Consequently,
\begin{align} \nonumber
	\int_{\Sigma_t} {Q}_{\alpha}
	\,  & \lesssim_K  \|  {\psi} \|_{H^{3/2}(\Sigma_t)}^2 + \|(W_2, \, w_2) \|_{H^{1/2}(\Sigma_t)}^2  \\[0.5mm]
	& \lesssim_K  \|  {\psi} \|_{H^{3/2}(\Sigma_t)}^2
	+\| (W_2,w_2 )\|_{H^{1}(\Omega_{t})}^2
	\qquad \textrm{for }\langle \alpha\rangle \leq  1,
	\nonumber
\end{align}
which together with \eqref{es1} and \eqref{R1:es} yields
\begin{align}
	\|\mathrm{D}_*^{\alpha}(W,w)(t)\|_{L^2(\Omega)}^2
 \lesssim_K 	\|  {\psi} \|_{H^{3/2}(\Sigma_t)}^2+ {\|}(\bm{f}^+,W){\|}_{H_*^1(\Omega_{t})}^2
	+ \| (W_2, w)  \|_{H^{1}(\Omega_{t})}^2
	\label{FP:es1}	
\end{align}
for ${\langle\alpha\rangle\leq 1}$.
Introduce
\begin{align} \nonumber
	\mathcal{I}(t):= 	\sum_{\langle\alpha\rangle\leq 1}\big\|\big(\mathrm{D}_*^{\alpha}W,\, \p_1 W_1 , \,\p_1 W_2 , \, \p_1 W_4,\, w,\,\mathrm{D} w\big)(t)\big\|_{L^2(\Omega)}^2.
\end{align}
It follows from \eqref{d1W}--\eqref{FP:es1} and \eqref{H.inv3}--\eqref{h.inv3} that
\begin{align}
	\mathcal{I}(t) \lesssim_K 	\|  {\psi} \|_{H^{3/2}(\Sigma_t)}^2+ {\|}(\bm{f}^+,W){\|}_{H_*^1(\Omega_{t})}^2
	+ \| (W_2, w)  \|_{H^{1}(\Omega_{t})}^2.
	\nonumber 
\end{align}
Applying Gr\"{o}nwall's inequality to the last estimate implies
\begin{align}
	{\|} W{\|}_{H_*^1(\Omega_{T})}^2
	+\| (W_1, W_2, W_4,  w)   \|_{H^{1}(\Omega_{T})}^2
	\lesssim_K 	T   \left( \|  {\psi} \|_{H^{3/2}(\Sigma_T)}^2+ {\|} \bm{f}^+ {\|}_{H_*^1(\Omega_{T})}^2 \right).
	\label{FP:es4}	
\end{align}

\subsubsection{Existence for problem \eqref{ELP4}}
According to the inequality \eqref{FP:es4}, $W_2$ belongs to $H^{1}(\Omega_T)$,
and hence $W_2|_{\Sigma}\in H^{1/2}(\Sigma_T)$.
This implies the existence of solutions $\phi$ in $H^{1/2}(\Sigma_T)$ of the equation \eqref{BC.2Da}.

To show  that  $\phi$ belongs to $H^{3/2}(\Sigma_T)$, we deduce some boundary constraints for $\phi$.
As in the proof of Lemma \ref{lem:homo}, we can deduce that
solutions $(W,w, \phi)$ of the problems \eqref{FP1} and \eqref{BC.2Da} satisfy
\begin{alignat}{3}
&\big(\p_t+\mathring{v}_2\p_2+\p_2\mathring{v}_2\big)\big(H\cdot\mathring{N}-\mathring{H}_2\p_2\phi+\p_1\mathring{H}\cdot\mathring{N}\phi \big)=0
 \qquad &&\textrm{on } \Sigma_T,	\\[1mm]
	\label{cons2v:id1.2D}
&	\epsilon \p_t\big(h\cdot \mathring{N}\big)
	+\p_2\big(e +\epsilon \p_t\mathring{\psi} h_2\big)
	=0
	\qquad&&\textrm{on } \Sigma_{T}.
\end{alignat}
Then we have 
\begin{align} \label{psi:id1}
	W_4=H\cdot\mathring{N}=\mathring{H}_2\p_2\phi-\p_1\mathring{H}\cdot\mathring{N}\phi \qquad \textrm{on } \Sigma_{T}.
\end{align}
By virtue of 
\eqref{mu:def1}, \eqref{w:def1}, and \eqref{FP1c}--\eqref{FP1c2},
we infer
\begin{align}
	e +\epsilon \p_t\mathring{\psi} h_2
	=\epsilon \mathring{v}_2 w_1 +w_3
	= \epsilon \mathring{v}_2 w_1-\epsilon \mathring{h}_2 W_2 +\mathring{b}_5 \psi.
	\nonumber
\end{align}
Plugging the last identity into \eqref{cons2v:id1.2D}, we can utilize the equation \eqref{BC.2Da} and the second condition in \eqref{bas4} to discover
\begin{align} \label{G.cal:eq}
	\epsilon\p_t \mathcal{G}+\epsilon\p_2\big(\mathring{v}_2 \mathcal{G} \big)
	+\p_2\big(\mathring{b}_5 (\psi-\phi)\big)=0 \qquad \textrm{on } \Sigma_{T},
\end{align}
where
\begin{align} \label{G.cal:def}
	\mathcal{G}:=w_1-\mathring{h}_2\p_2 \phi -\p_2\mathring{h}_2\phi .
\end{align}

Thanks to \eqref{bas3} and \eqref{bas10}, we derive
\begin{align} \label{noncol2}
	\big|\mathring{H}_2\big|+\big|\mathring{h}_2\big|\geq \kappa_0 >0 \qquad \textrm{on }\Sigma_T
\end{align}
for some constant $\kappa_0>0$. Then it follows from \eqref{BC.2Da}, \eqref{psi:id1}, and \eqref{G.cal:def} that
\begin{align}
	(\p_t\phi,\, \p_2\phi)
	=\mathring{\rm c}_0 (W_2,\, W_4,\, w_1-\mathcal{G})
	+\mathring{\rm c}_1 \phi
	\qquad \textrm{on }\Sigma_T.
	\label{psi:id2}
\end{align}
Plug \eqref{psi:id2} into \eqref{G.cal:eq} to obtain
\begin{align*}
	\epsilon\p_t \mathcal{G}+\epsilon\p_2\big(\mathring{v}_2 \mathcal{G} \big)
	+\mathring{\rm c}_1 \mathcal{G}
	=\mathring{\rm c}_2\phi+
	\mathring{\rm c}_1(W_2,\, W_4,\, w_1)
	+\mathring{\rm c}_1\p_2\psi+\mathring{\rm c}_2\psi
	\quad \textrm{on }\Sigma_T.
\end{align*}
Applying a standard energy method to the last equation yields the estimates of $\|\mathcal{G} \big\|_{L^{2}(\Sigma_T)}$ and $\|\mathcal{G} \big\|_{H^{1}(\Sigma_T)}$.
Then we employ interpolation to get
\begin{align}
	\big\|\mathcal{G} \big\|_{H^{1/2}(\Sigma_T)}^2
	\lesssim_K	
	T \mathrm{e}^{C(K)T} \big\|\big(W_2, W_4, w_1,  \phi,  \psi,   \p_2 \psi\big)\big\|_{H^{1/2}(\Sigma_T)}^2,
	\nonumber
\end{align}
which together with \eqref{psi:id2} and \eqref{FP:es4} leads to
\begin{align}
	\| \phi\|_{H^{3/2}(\Sigma_T)}^2
	\lesssim_K	\; &
	\big\|\big(W_2, W_4,  w_1 \big)\big\|_{H^{1/2}(\Sigma_T)}^2
	+T \big\|\psi \big\|_{H^{3/2}(\Sigma_T)}^2	
	\nonumber\\[0.5mm]
	\lesssim_K	\; &
	T\|  {\psi} \|_{H^{3/2}(\Sigma_T)}^2+  T {\|} \bm{f}^+ {\|}_{H_*^1(\Omega_T)}^2
	\label{FP:es5}
\end{align}
for $T>0$ small enough.
This defines a mapping ${\psi}\mapsto \phi$ in $H^{3/2}(\Sigma_T)$.

To show that the mapping ${\psi}\mapsto \phi$ is a contraction, we take
$\psi_1, \psi_2\in H^{3/2}(\Sigma_T)$ and let $(W_1, w_1, \phi_1), (W_2, w_2, \phi_2)\in H_*^{1}(\Omega_T)\times H^{1}(\Omega_T)\times H^{3/2}(\Sigma_T)$ be the corresponding solution of \eqref{FP1} and \eqref{BC.2Da}.
By virtue of the linearity of the problems \eqref{FP1} and \eqref{BC.2Da}, similar to \eqref{FP:es5}, we can deduce
\begin{align}
	\| \phi_1-\phi_2\|_{H^{3/2}(\Sigma_T)}^2
	\lesssim_K	
	T \|  \psi_1-\psi_2 \|_{H^{3/2}(\Sigma_T)}^2.
	\nonumber
\end{align}
Consequently, there is a small constant $T_0>0$, such that if $0<T\leq T_0$, then ${\psi}\mapsto \phi$ has a unique fixed point by the contraction mapping principle.
Then the fixed point ${\psi}=\phi$, together with the corresponding solution $(W,w)$ of \eqref{FP1}, provides the solution $(W , w , \phi ) \in H_*^{1}(\Omega_T)\times H^{1}(\Omega_T)\times H^{3/2}(\Sigma_T)$
for the problem \eqref{ELP4a}--\eqref{ELP4b}, \eqref{ELP4d}, \eqref{BC.2Da}--\eqref{BC.2Dc} and also for the equivalent problem \eqref{ELP4}.

\subsubsection{Proof Theorem \ref{thm:main2}}

Let $(W,w,\psi)$ be the solution of the problem \eqref{ELP4}.
We shall estimate $(W,w,\psi)$ in $H_*^m(\Omega_T)\times H^m(\Omega_T)\times H^{m+1/2}(\Sigma_T)$ from \eqref{es1} for any integer $m\geq 6$.

Using \eqref{decom2.2D}--\eqref{decom3} and \eqref{d1W}--\eqref{d1w}, similar to \eqref{es2} and \eqref{Wnc:es}, we have
\begin{align}
	\sum_{\langle \alpha\rangle \leq m,\
		\alpha_1+\alpha_3>0}	\|(\mathrm{D}_*^{\alpha}W,\, \mathrm{D}_*^{\alpha}w)(t)\|_{L^2(\Omega)}^2
		+	\big\| \widetilde{\mathcal{W}}_{\rm nc} \big\|_{H^{m-1/2}(\Sigma_t)}^2
	\lesssim_K  \mathcal{M}(t),
	\label{es:nor2D}
\end{align}
where $\mathcal{M}(t)$ and $\widetilde{\mathcal{W}}_{\rm nc}$ are defined by \eqref{R3:es} and \eqref{W.cal:nc}.
 	 If $\alpha_1=\alpha_3=0$, we use \eqref{BC.2Db} and \eqref{BC.2Dc} to infer
	\begin{align}\nonumber
		Q_{\alpha} =\; &2 \epsilon \mathrm{D}_*^{\alpha} W_2 \mathrm{D}_*^{\alpha}(\mathring{h}_2w_2-\mathring{b}_2 \psi )
		+2\mathrm{D}_*^{\alpha}w_2 \mathrm{D}_*^{\alpha}(-\epsilon \mathring{h}_2 W_2+\mathring{b}_5 \psi) \\[0.5mm]
		=\; &  \mathrm{D}_*^{\alpha} (W_2, w_2)[\mathrm{D}_*^{\alpha}, \, \mathring{\rm c}_0](W_2, w_2) +\mathrm{D}_*^{\alpha} (W_2, w_2)\mathrm{D}_*^{\alpha} ( \psi\mathring{\rm c}_1).
		\label{Q.alpha:id}
	\end{align}
Then similar to the estimate \eqref{R3:es}, we can deduce
	\begin{align} \nonumber
		\int_{\Sigma_t}{Q}_{\alpha}
		 \lesssim  \|\mathrm{D}_*^{\alpha}(W_2,w_2)\|_{H^{-1/2}(\Sigma_t)}^2 + \|([\mathrm{D}_*^{\alpha},\, \mathring{\rm c}_0](W_2,w_2),\, \mathrm{D}_*^{\alpha} (\psi \mathring{\rm c}_1 )) \|_{H^{1/2}(\Sigma_t)}^2   \lesssim_K  \mathcal{M}(t). \nonumber 
	\end{align}
Plugging the last estimate and \eqref{R1:es} into \eqref{es1}, using
\eqref{h.inv3} and \eqref{d1w}, and applying the induction argument  lead to
\begin{align}\label{es:H1}
	\sum_{\langle \alpha\rangle \leq m}	\| \mathrm{D}_*^{\alpha}W(t)\|_{L^2(\Omega)}^2
	+	\sum_{|\alpha| \leq m}	\|  \mathrm{D}^{\alpha}w(t)\|_{ L^2(\Omega)}^2
	\lesssim_K  \mathcal{M}(t).
\end{align}

In view of the constraint \eqref{bas10} (or equivalently, \eqref{noncol2}), we  utilize \eqref{H.cons3}--\eqref{h.cons3} and \eqref{ELP4c} to get
\begin{align}
	(\p_t\psi,\, \p_2\psi)
	=\mathring{\rm c}_0 (W_2,\, W_4,\, w_1)
	+\mathring{\rm c}_1 \psi
	\qquad \textrm{on }\Sigma_T.
\end{align}
Then we can obtain as in \S\S \ref{sec.psi}--\ref{sec.Thm2.2} that the estimates \eqref{es:psi4.3D} and \eqref{es:pr} also hold for the 2D case, which implies the tame estimate \eqref{tame.es}.
Applying the arguments in \cite{CP82MR0678605, S95MMASMR1346663} and using the estimate \eqref{es:pr},
one can establish the existence and uniqueness of solutions $(W,w,\psi )$ in $H_*^m(\Omega_{T})\times H^m(\Omega_{T})\times H^{m+1/2}(\Sigma_{T})$ to the problem \eqref{ELP4} for any integer $m\geq 6$ and sufficiently small $T>0$.
Then we can use Lemma \ref{lem:homo}
and \eqref{es:pr} to obtain the estimate \eqref{tame.es} and complete the proof of Theorem \ref{thm:main2}.

\subsection{Approximate solutions}	
To apply Theorem \ref{thm:main2}, which is valid for functions vanishing in the past,
we reduce the nonlinear problem \eqref{NP1} to the case with zero initial data.
For this purpose, we introduce the so-called approximate solutions to absorb the initial data into the interior equations.
The compatibility conditions on the initial data introduced in Definition \ref{def:1} are necessary in the following construction of the approximate solutions.

\begin{lemma}\label{lem:app}
	Assume that all the conditions in Lemma \ref{lem:CA1} are satisfied.
	Assume further that the initial data $(U_0, u_0, \varphi_0)$ are compatible up to order $m$ and satisfy the constraints \eqref{H.h.inv2}--\eqref{H.h.cons2} and \eqref{stability2D}.
	Then
	there exist functions $(U^a,u^a)$ and $\varphi^a$
	defined respectively on $Q:=\mathbb{R}\times \Omega$
	and $\omega:=\mathbb{R}\times \Sigma$, which satisfy the constraints
	\eqref{bas2}--\eqref{bas6},
	and
	\begin{alignat}{3}	
		&\p_t^{j}\mathbb{L}_+(U^{a},\Phi^{a+})\big|_{t=0}=0
		\quad&&\textrm{for } j=0,1,\ldots,m-1,		\label{app1a}\\
		&\p_t^{j}\mathbb{L}_-(u^{a}, \nu^{a},\Phi^{a-})\big|_{t=0}=0
		\quad&&\textrm{for } j=0,1,\ldots,m-1,		\label{app1b}\\		
		& \mathbb{B}(U^a,u^a,\varphi^a)=0\quad &&\textrm{on }\omega ,
		\label{app1d}	\\
		& (U^{a},u^{a}, \varphi^{a})=(U_0, u_0, \varphi_0)\qquad &&\textrm{if } t<0,
		\label{app1e}
	\end{alignat}
	where $\Phi^{a\pm}(t,x):=\pm x_1 +\chi(x_1)\varphi^{a}(t,x_2)$
	and $\nu^a(t,x):=\chi(x_1)v^a(t,0, x_2)$ for  $x\in\Omega$.
	Moreover,
	\begin{align}
		&	\label{app2b}
		\big|{H}^{a}\big|+\big| {h}^{a}\big|\geq \frac{3\delta_0}{4}>0
		\quad \textrm{on }\omega,\qquad
		\|\varphi^a\|_{L^{\infty}(\omega)}\leq 2,\\[0.5mm]			
		&	\|(U^{a}-\widebar{U}, u^{a}-\bar{u})\|_{H^{m+1}(Q)} +\|\varphi^a\|_{H^{m+5/2}(\omega)}  \leq C (M_0),
		\label{app2a}
	\end{align}
	for some positive constant $C(M_0)$ depending on $M_0$ $($cf.~\eqref{M0:def}$)$.
\end{lemma}
\begin{proof}
	The proof is divided into five steps.
	
	\vspace*{2mm}
	\noindent {\it Step 1}.\quad
	Since all the conditions in Lemma \ref{lem:CA1} hold, we can
	take $\varphi^a\in H^{m+5/2}(\omega)$
	and $(v_2^{a}-\bar{v}_2, {S}^{a}-\widebar{S})\in H^{m+2}(Q)$
	to satisfy	
	\begin{alignat*}{3}
		\p_t^{j}(\varphi^a, \, v_2^{a},\,  {S}^{a})\big|_{t=0}
		=(\varphi_{(j)},\, v_{2(j)},\,  S_{(j)})
		\quad \textrm{for } j=0,1,\ldots,m.	
	\end{alignat*}
	The trace theorem implies
	$ w^a:=\p_t\varphi^a +  \p_2\varphi^a v_2^{a} |_{\omega} \in  H^{m+3/2}(\omega). $
	In view of \eqref{trace.id1},
	we can choose $v_1^{a}$ with trace
	$v_1^{a}|_{\omega} =w^a$, such that
	$ v_1^{a}-\bar{v}_1\in H^{m+2}(Q)$ and
	$$
	\p_t^{j}v_1^{a}\big|_{t=0}=v_{1(j)}\qquad \textrm{in } \Omega \quad
	\textrm{for } j=0,1, \ldots,m.$$
	Then the first condition in \eqref{app1d} follows  directly.
	
	\vspace*{2mm}
	\noindent {\it Step 2}.\quad	
	Set $\Psi^{a}:=\chi( x_1)\varphi^{a} $
	and $\Phi^{a\pm }(t,x):= \pm x_1 +\Psi^{a }(t,x) $.
	Then $\Psi^a \in  H^{m+5/2}(Q)$.
	Take ${H}^{a}$ as the unique solution of \eqref{Hbb:def} (with $( {v}, {\Phi}^+)$ replaced by $(v^a,\Phi^{a+})$)
	supplemented with the initial data ${H}^{a}|_{t=0}={H}_{0}.$
	Then ${H}^{a}-\widebar{H}\in H^{{m+1}}(Q)$.
	
	\vspace*{2mm}
	\noindent {\it Step 3}.\quad	
	Define
	$\nu^a(t,x):=\chi(x_1)v^a(t,0, x_2)$.
	Then $v^{a}|_{\omega}-\bar{v}\in H^{m+3/2}(\omega)$ implies
	$\nu^{a}-\bar{v}\in H^{m+3/2}(Q)$.
	{Applying a similar argument as in Appendix \ref{app:D} and using the compatibility conditions \eqref{comp2},
		we can take $u^a=(h_1^a, h_2^a, e^a)$ to be the unique solution of \eqref{NP1b}
		(with $( {\nu}, {\Phi}^-)$ replaced by $(\nu^a,\Phi^{a-})$)
		 supplemented with the boundary and initial conditions:}
	\begin{align*}
		e^a+\epsilon \p_t \varphi^a  h_2^a=0\quad \textrm{on }\omega,\qquad {u}^{a}={u}_{0}\quad \textrm{if }t=0.
	\end{align*}
	Then $u^a-\bar{u}\in H^{m+1}(Q)$
	and $\p_t^{j}u^{a}|_{t=0}=u_{(j)}$ for $j=0,1, \ldots,m$.

	\vspace*{2mm}
	\noindent {\it Step 4}.\quad		
	By virtue of the compatibility conditions \eqref{comp1},
	we can apply the lifting result in \cite[Theorem 2.3]{LM72MR0350178}
	to find ${q}^{a}$ satisfying ${q}^{a}-\bar{q}\in H^{m+2}(Q)$,
	the second condition in \eqref{app1d}, and
	$\p^{j} {q}^{a}  |_{t=0}= {q}_{(j)}$ for $j=0,1, \ldots,m$.

	\vspace*{2mm}
	\noindent {\it Step 5}.\quad
Note that $H^a$ and $u^a$ are the solutions of the equations \eqref{bas4} (with $(\mathring{v},\mathring{\nu},\mathring{\Phi}^{\pm})$ replaced by $(v^a,\nu^a,\Phi^{a\pm})$).
As in Appendix \ref{app:C}, we can show that $(H^a, h^a, \Phi^{a\pm})$ satisfies
	\eqref{H.h.inv2}--\eqref{H.h.cons2} for all $t\in\mathbb{R}$, since they hold initially.
	Using Lemma \ref{lem:CA1} and the continuity of the lifting operators, we find that
	$(U^a,u^a,\varphi^a)$ constructed above satisfies \eqref{bas3}--\eqref{bas6}, \eqref{app1a}--\eqref{app1e}, and \eqref{app2a}.
	By a cut-off argument in the above procedure, we can choose  $U^a$, $u^a$,  and $\varphi^a$  to satisfy \eqref{bas2} and \eqref{app2b}. The proof is complete.
\end{proof}

The vector function $(U^a,u^a,\varphi^a)$ constructed in Lemma \ref{lem:app} is called the {\it approximate solution} to the nonlinear problem \eqref{NP1}. Let us define
\begin{align}\nonumber 
	f^{a}:=
	\bigg\{\begin{aligned}
		& -\mathbb{L}_+(U^{a},\Phi^{a+}) \ \  &\textrm{if }t>0,\\
		& \; 0 \ \  &\textrm{if }t<0.
	\end{aligned}
\end{align}
Then it follows from \eqref{app2a} and \eqref{app1a} that
$	f^{a}\in H^{m}(\Omega_T)$ and
\begin{align}\label{f^a:est}
	\|f^{a}\|_{ H^{m}(\Omega_T)}\leq\delta_0\left(T\right),
\end{align}
where $\delta_0(T)\to 0$ as $T\to 0$.
Since $(U^a,u^a,\varphi^a)$ satisfies \eqref{app1a}--\eqref{app1e} and \eqref{bas4}, we find that
$(\widehat{U},\hat{u},\hat\varphi)$ is a solution of  the problem \eqref{NP1} on the time interval $[0,T]$,
provided that $(U,u,\psi):=(\widehat{U},\hat{u},\hat\varphi)-(U^{a},u^{a},\varphi^a)$ solves
\begin{align} \label{P.new}
	\left\{
	\begin{aligned}
		&\mathcal{L}_+(U,\Psi):=\mathbb{L}_+(U^a+U,\Phi^{a+}+\Psi)-\mathbb{L}_+(U^a,\Phi^{a+})=f^a &\quad &\textrm{in }\Omega_{T},\\[0.3mm]
		&\mathcal{L}_-(u,\nu,\Psi):=\mathbb{L}_-(u^a+u,\nu^a+\nu, \Phi^{a-}+\Psi)=0 &&\textrm{in }\Omega_{T},\\[0.3mm]
		&\mathcal{B}(U,u,\psi):=\mathbb{B}(U^a+U,u^a+u,\varphi^a+\psi)=0&&\textrm{on }\Sigma_T,\\[0.3mm]
		&(U,u,\psi)=0\ &&\textrm{if }t< 0,
	\end{aligned}\right.
\end{align}
for $\Psi(t,x):=\chi(x_1)\psi(t,x_2)$
and $\nu(t,x):=\chi(x_1)v(t,0, x_2)$.
From Lemma \ref{lem:app}, we infer that $(U,u,\psi)\equiv 0$ satisfies \eqref{P.new} for $t<0$.
Therefore, the original problem on the time interval $[0,T]$ has been reduced to the problem on $(-\infty, T]$ with solutions vanishing in the past.

\subsection{Nash--Moser iteration}
To describe the iteration scheme,
we state the following properties on smoothing operators from
\cite{A89MR976971,CS08MR2423311,T09ARMAMR2481071}.
Let $\mathscr{F}_*^k(\Omega_{T})$ (resp.,~$\mathscr{F}^s(\Omega_{T})$, $\mathscr{F}^s(\Sigma_T)$) denote the class of $H_*^{k}(\Omega_{T})$ (resp.,~$H^{s}(\Omega_{T})$, $H^{s}(\Sigma_T)$) vanishing in the past
{with $k\in\mathbb{N}$ and $s\geq 0$ (not necessarily an integer)}.
\begin{proposition} \label{pro:smooth}
	Let $T>0$.
	Then there exists a family of smoothing operators $\{\mathcal{S}_{\theta}\}_{\theta\geq 1}$  from $  \mathscr{F}_*^0(\Omega_{T}) $ to $ \bigcap_{k\geq 0}\mathscr{F}_*^k(\Omega_{T})$, such that
	\begin{subequations}\label{smooth.p1}
		\begin{alignat}{2}
			&  \|\mathcal{S}_{\theta} u\|_{H_*^{k}(\Omega_{T})}\lesssim_m \theta^{(k-j)_+}\|u\|_{H_*^{j}(\Omega_{T})}
			&& \textrm{for  \;} k,j\geq 0,    \label{smooth.p1a}  \\[1.5mm]
			&  \|\mathcal{S}_{\theta} u-u\|_{H_*^{k}(\Omega_{T})}\lesssim_m \theta^{k-j}\|u\|_{H_*^{j}(\Omega_{T})}
			&& \textrm{for  \;}0\leq k\leq j \leq m,   \label{smooth.p1b} \\
			&  \left\|\frac{\d}{\d \theta}\mathcal{S}_{\theta} u\right\|_{H_*^{k}(\Omega_{T})}
			\lesssim_m \theta^{k-j-1}\|u\|_{H_*^{j}(\Omega_{T})}
			&\qquad&\textrm{for \;}  k,j\geq 0,    \label{smooth.p1c}
		\end{alignat}
	\end{subequations}
	where $(k-j)_+:=\max\{0,\, k-j \}$.
	Moreover, there are two families of smoothing operators (still denoted by $\mathcal{S}_{\theta}$) acting respectively on $\mathscr{F}^6(\Omega_{T})$ and $\mathscr{F}^6(\Sigma_T)$ that satisfy  the properties  \eqref{smooth.p1} with corresponding norms $\|\cdot\|_{H^{s}(\Omega_{T})}$ and   $\|\cdot\|_{H^{s}(\Sigma_T)}$ with $s\geq 0$ (not necessarily an integer).
\end{proposition}

Next we follow \cite{CS08MR2423311,T09ARMAMR2481071,TW21MR4201624, 
	TW22aMR4444136,TW22cMR4506578} to describe the iteration scheme for the reformulated problem \eqref{P.new}.

\vspace{2mm}
\noindent{\bf Assumption (A-1)}. {\it  Set $ (U_0,u_0, \psi_0)=0$. Let $(U_k, u_k,\psi_k)$ be given and vanish in the past for $k=1,2,\ldots, n$. Set $\Psi_k(t,x):=\chi( x_1)\psi_k(t,x_2)$
and $\nu_k(t,x):=\chi( x_1)v_k(t,0,x_2)$.}

\vspace*{1mm}

Consider
\begin{align}\label{NM0}
	U_{{n}+1}=U_{{n}}+\delta U_{{n}},
	\quad u_{n+1}=u_n+\delta u_{n},
	\quad \psi_{{n}+1}=\psi_{{n}}+\delta \psi_{{n}},
\end{align}
where 
$\delta U_{{n}},$ $\delta u_{{n}},$ and $\delta \psi_{{n}}$ are specified by the effective linear problem
\begin{align}
	\left\{\begin{aligned}
		&\mathbb{L}'_{e+}(U^a+U_{n+1/2},\Phi^{a+}+\Psi_{n+1/2})\delta \dot{U}_{{n}}=f_{{n}}^+
		&\ &\textrm{in } \Omega_T, \\
		&L_-(\nu^a+\nu_{n+1/2}, \Phi^{a-}+\Psi_{n+1/2})\delta \dot{u}_{{n}}=f_{{n}}^-
		&&\textrm{in } \Omega_T, \\
		& \mathbb{B}'_{n+1/2}(\delta \dot{U}_{{n}}, \delta\dot{u}_{{n}},\delta\psi_n)=g_n
		&&\textrm{on } \Sigma_T ,\\
		&(\delta \dot{U}_{{n}},\delta \dot{u}_{{n}},\delta\psi_{{n}}) =0 &&
		\textrm{if } t<0,
	\end{aligned} \right.
	\label{effective.NM}
\end{align}
with
\begin{align}
	\label{B.bb':2}
	&\mathbb{B}'_{n+1/2}:=\mathbb{B}'_{e}(U^a+U_{n+1/2},u^a+u_{n+1/2}, \varphi^a+\psi_{n+1/2}).
\end{align}
The operators $\mathbb{L}'_{e+}$, $\mathbb{B}'_{e}$, and $L_-$ are given in \eqref{ELP1a}, \eqref{B'e:def}, and \eqref{L-:def}, respectively.
We extend $\delta \nu_n$ (resp., $\nu_{n+1/2}$) to $\Omega_{T}$ from the trace $\delta v_n|_{x_1=0}$ (resp., $v_{n+1/2}|_{x_1=0}$) as in \eqref{nu:def1}.
Here $(U_{n+1/2},u_{n+1/2}, \psi_{n+1/2})$ is a smooth modified state such that $(U^a+U_{n+1/2},u^a+u_{n+1/2},\varphi^a+\psi_{n+1/2})$ satisfies the constraints \eqref{bas1}--\eqref{bas6},
and $\delta \dot{U}_{{n}},$ $\delta \dot{u}_{{n}}$ are the good unknowns 
\begin{align}
	\label{good.NM}
	\left\{
	\begin{aligned}
		&\delta \dot{U}_{{n}}:=\delta U_{{n}}-\frac{\delta\Psi_{{n}}}{\p_1 (\Phi^{a+}+\Psi_{n+1/2})}\p_1 (U^a+U_{n+1/2}), \\[0.3mm]
		&
		\delta \dot{u}_{{n}}:=\delta u_{{n}}-\frac{\delta\Psi_{{n}}}{\p_1 (\Phi^{a-}+\Psi_{n+1/2})}\p_1 (u^a+u_{n+1/2}),
	\end{aligned}
	\right.
\end{align}
for $\delta\Psi_{{n}}:=\chi(x_1)\delta \psi_{{n}}$ and $\Psi_{n+1/2}:=\chi(x_1)\psi_{n+1/2}$.
The construction of the modified state will be provided in Proposition \ref{pro:modified}.
The source terms $f_n^{\pm}$ and $g_n$ will be determined through the accumulated error terms later on.

\vspace{2mm}
\noindent{\bf Assumption (A-2)}. {\it Set $f_0^{+}:=\mathcal{S}_{\theta_0}f^a$ and $(e_0^{\pm},\tilde{e}_0 ,f_0^-,g_0):=0$ for $\theta_0\geq 1$ sufficiently large.
	Let $(e_k^{\pm},\tilde{e}_k ,f_k^{\pm},g_k)$ be given and vanish in the past for $k=1,2,\ldots,{n}-1$.}

\vspace{1mm}

Under Assumptions (A-1) and (A-2), the accumulated error terms at Step ${n}$ are defined by
\begin{align}  \label{E.E.tilde}
	E_{{n}}^{\pm}:=\sum_{k=0}^{{n}-1}e_{k}^{\pm},\quad \widetilde{E}_{{n}} :=\sum_{k=0}^{{n}-1}\tilde{e}_{k} ,
\end{align}
and the source terms $f_n^{\pm}$ and $g_n$ are chosen from
\begin{align} \label{source}
	\sum_{k=0}^{{n}} f_k^+ +\mathcal{S}_{\theta_{{n}}}E_{{n}}^+=\mathcal{S}_{\theta_{{n}}}f^a,\quad
	\sum_{k=0}^{{n}}f_k^{-}+\mathcal{S}_{\theta_{{n}}}{E}_{{n}}^{-}=0,\quad
	\sum_{k=0}^{{n}}g_k +\mathcal{S}_{\theta_{{n}}}\widetilde{E}_{{n}} =0,
\end{align}
where  $\mathcal{S}_{\theta_{{n}}}$
are the smoothing operators given in Proposition \ref{pro:smooth} with
$\theta_{{n}}:=\sqrt{\theta^2_0+{n}}.$
Once $f_n^{\pm}$ and $g_n $ are specified,
we can get
$(\delta \dot{U}_{{n}},\delta\dot{u}_n, \delta \psi_{{n}})$ as the solution of
the problem \eqref{effective.NM} by using Theorem \ref{thm:main2}.
Then we obtain $(\delta U_{{n}}, \delta u_{{n}})$ and $(U_{n+1}, u_{n+1},\psi_{{n}+1})$ from \eqref{good.NM} and \eqref{NM0}.

To complete the description of the iteration scheme,
we define the error terms at Step $n$ through the decompositions
\begin{align}
	\nonumber&\mathcal{L}_+(U_{{n}+1}, \Psi_{{n}+1})-\mathcal{L}_+(U_{{n}}, \Psi_{{n}})\\
	\nonumber	&{}= \mathbb{L}_+'(U^a+U_{{n}}, \Phi^{a+}+ \Psi_{{n}})(\delta U_{{n}},\delta \Psi_{{n}})+e_{{n}+}'\\
	\nonumber&{} = \mathbb{L}_+'(U^a+\mathcal{S}_{\theta_{{n}}}U_{{n}}, \Phi^{a+}+\mathcal{S}_{\theta_{{n}}} \Psi_{{n}})(\delta U_{{n}},\delta \Psi_{{n}})+e_{{n}+}'+e_{{n}+}''\\
	\nonumber&{} = \mathbb{L}_+'(U^a+U_{n+1/2}, \Phi^{a+}+ \Psi_{n+1/2})(\delta U_{{n}},\delta \Psi_{{n}})+e_{{n}+}'+e_{{n}+}''+e_{{n}+}'''\\
	\label{deco1}
	&{} = \mathbb{L}_{e+}'(U^a+U_{n+1/2}, \Phi^{a+}+ \Psi_{n+1/2})\delta \dot{U}_{{n}}+e_{{n}+}'+e_{{n}+}''+e_{{n}+}'''+e_{{n}+}^{*},
	\\[1mm]
	\nonumber&
	\mathcal{L}_-(u_{{n}+1}, \nu_{{n}+1}, \Psi_{{n}+1})-\mathcal{L}_-(u_{{n}},\nu_{{n}},  \Psi_{{n}})\\
	\nonumber	&{}	
	= \mathbb{L}_-'(u^a+u_{{n}}, \nu^a+\nu_{{n}},  \Phi^{a-}+ \Psi_{{n}})(\delta u_{{n}},\delta \nu_{{n}}, \delta \Psi_{{n}})+e_{{n}-}'\\
	\nonumber&{} = \mathbb{L}_-'(u^a+\mathcal{S}_{\theta_{{n}}}u_{{n}},\nu^a+\mathcal{S}_{\theta_{{n}}}\nu_{{n}}, \Phi^{a-}+\mathcal{S}_{\theta_{{n}}} \Psi_{{n}})(\delta u_{{n}},\delta \nu_{{n}},\delta \Psi_{{n}})+e_{{n}-}'+e_{{n}-}''\\
	\nonumber&{} = \mathbb{L}_-'(u^a+u_{n+1/2},\nu^a+\nu_{n+1/2}, \Phi^{a-}+ \Psi_{n+1/2})(\delta u_{{n}},\delta \nu_{{n}},\delta \Psi_{{n}})+e_{{n}-}'+e_{{n}-}''+e_{{n}-}'''\\
	\label{deco2}&{} = {L}_{-}(\nu^a+\nu_{n+1/2}, \Phi^{a-}+ \Psi_{n+1/2})\delta \dot{u}_{{n}}+e_{{n}-}'+e_{{n}-}''+e_{{n}-}'''+e_{{n}-}^{*},
\end{align}
and
\begin{align}
	\nonumber&\mathcal{B}(U_{{n}+1},u_{n+1},\psi_{{n}+1})-\mathcal{B}(U_{{n}},u_n, \psi_{{n}})  \\
	\nonumber&{}
	= \mathbb{B}'(U^a+U_{{n}},u^a+u_n,\varphi^a+\psi_{{n}})(\delta U_{{n}},\delta u_n, \delta\psi_{{n}})+\tilde{e}_{{n} }'\\
	\nonumber&{} = \mathbb{B}'(U^a+\mathcal{S}_{\theta_{{n}}}U_{{n}},u^a+\mathcal{S}_{\theta_{{n}}}u_{{n}},\varphi^a
	+\mathcal{S}_{\theta_{{n}}}\psi_{{n}})(\delta U_{{n}},\delta u_n, \delta\psi_{{n}})+\tilde{e}_{{n} }'+\tilde{e}_{{n} }''\\
	\label{deco3}&{} =\mathbb{B}'_{n+1/2}(\delta \dot{U}_{{n}} ,\delta\dot{u}_n,\delta\psi_{{n}})+\tilde{e}_{{n} }'+\tilde{e}_{{n} }''+\tilde{e}_{{n} }''',
\end{align}
where the operators $\mathcal{L}_{\pm}$, $\mathcal{B}$,
$\mathbb{L}_{\pm}'$,
$\mathbb{B}'$, and $\mathbb{B}'_{n+1/2}$ are given in
\eqref{P.new}, \eqref{Alinhac1}, \eqref{Alinhac2},
\eqref{B'bb2D:def}, and \eqref{B.bb':2}, respectively.
Setting the error terms as
\begin{align} \label{e.e.tilde}
	e_{{n}}^{\pm}:=e_{{n}\pm}'+e_{{n}\pm}''+e_{{n}\pm}'''+e_{n\pm}^*,\quad
	\tilde{e}_{{n}} :=\tilde{e}_{{n} }'+\tilde{e}_{{n} }''+\tilde{e}_{{n} }'''
\end{align}
gives the iteration scheme, which is of Nash--Moser type \cite{A89MR976971,H76MR0602181,CS08MR2423311}.

Summing \eqref{deco1}--\eqref{deco3} from ${n}=0$ to $N$, respectively,
we can obtain from \eqref{effective.NM}, \eqref{E.E.tilde}, \eqref{source}, and \eqref{e.e.tilde} that
\begin{align}
	\label{LB:conver}
	\left\{
	\begin{aligned}
		&	\mathcal{L}_+(U_{N+1},\Psi_{N+1})= \mathcal{S}_{\theta_{N}}f^a
		+({I}-\mathcal{S}_{\theta_{N}})E_{N}^+ +e_{N}^+,
		\\
		&\mathcal{L}_-(u_{N+1}, \nu_{N+1}, \Psi_{N+1})=({I}-\mathcal{S}_{\theta_{N}})E_{N}^- +e_{N}^-,
		\\
		&\mathcal{B}(U_{N+1},u_{N+1}, \psi_{N+1})=({I}-\mathcal{S}_{\theta_{N}})\widetilde{E}_{N}
		+\tilde{e}_{{N}}.
	\end{aligned}
	\right.
\end{align}
Since $\mathcal{S}_{\theta_{N}}\to {I}$ as $N\to \infty$,
we can formally obtain the existence of solutions of the problem \eqref{P.new}
from the expectation $(e_{{N}}^{+},e_{{N}}^{-}, \tilde{e}_{N})\to 0$.

\subsection{Inductive hypothesis and error estimates}
Let $m\geq {13}$ be an integer and set $\widetilde{\alpha}:=m-{5}$.
The initial data $(U_0,u_0, \varphi_0)$ are assumed to satisfy all the conditions in Lemma \ref{lem:app},
which leads to the estimates \eqref{app2b}--\eqref{f^a:est}.
Suppose that Assumptions {\bf (A-1)} and {\bf (A-2)} are fulfilled.
For some integer $ {\alpha }\in [7,\widetilde{\alpha}-1]$
and constant $\varepsilon>0$,  which will be chosen later on,
our inductive hypothesis reads
\begin{align*}\tag{\textrm{$\mathbf{H}_{{N}-1}$}}
	\left\{\begin{aligned}
		\textrm{(a)} \ &
		\|(\delta U_{n}, \delta u_{n}, \delta \psi_{n},\delta \Psi_{n})\|_{s}
		\leq \varepsilon \theta_{n}^{s-{\alpha }-1}\varDelta_{n}
		\\[0.5mm]
		&\quad \quad \textrm{for } n=0,1,\ldots,{N}-1\textrm{ and }s=6,7,\ldots,\widetilde{\alpha} ;\\[1mm]
		\textrm{(b)} \  & \|\mathcal{L}_+( U_{n},  \Psi_{n})-f^a\|_{H_*^s(\Omega_{T})}\leq 2 \varepsilon \theta_{n}^{s-{\alpha }-1}, \\[0.5mm]
		&
		\|\mathcal{L}_-( u_{n},  v_n, \Psi_{n})\|_{H^{s+1}(\Omega_{T})}  \leq 2 \varepsilon \theta_{n}^{s-{\alpha }-1}\\[0.5mm]
		& \quad \quad  \textrm{for } n=0,1,\ldots,{N}-1\textrm{ and } s= 6,7,\ldots,\widetilde{\alpha}-2;\\[1mm]
		\textrm{(c)} \  &\|\mathcal{B}( U_{n}, u_{n}, \psi_{n})\|_{H^{s+3/2}(\Sigma_{T})} \leq  \varepsilon \theta_{n}^{s-{\alpha }-1}
		\\[0.5mm]
		& \quad \quad  \textrm{for } n=0,1,\ldots,{N}-1\textrm{ and } s=7,8,\ldots,{\alpha },
	\end{aligned}\right.
\end{align*}
where  $\varDelta_{n}:=\theta_{n+1}-\theta_{n}$ with
$\theta_{{n}}:=\sqrt{\theta^2_0+{n}}$ and
\begin{align} \label{norm:ring1}
	\|(U, u, \psi,\Psi)\|_{s}:=
	\|U\|_{H_*^{s}(\Omega_{T})}+ \|u\|_{H^{s}(\Omega_{T})}+	 \|\psi\|_{H^{s+1/2}(\Sigma_T)}
	+\| \Psi \|_{H^{s+1/2}(\Omega_{T})}.
\end{align}
We shall prove that
hypothesis ($\mathbf{H}_{{N}-1}$) implies ($\mathbf{H}_{{N}}$)
and that hypothesis ($\mathbf{H}_0$) holds,
for $T,\varepsilon>0$ sufficiently small and $\theta_0\geq 1$ large enough.

Supposing that hypothesis ($\mathbf{H}_{{N}-1}$) is satisfied,
we can obtain the following lemma as in \cite{CS08MR2423311,T09ARMAMR2481071,TW21MR4201624,TW22aMR4444136,TW22cMR4506578}.
\begin{lemma}
	\label{lem:triangle}
	If $\theta_0\geq 1$ is large enough, then
	\begin{align}
		& \| ( U_{n}, u_{n}, \psi_{n},\Psi_{n})\|_{s}
		\leq \bigg\{\begin{aligned}
			&\varepsilon \theta_{n}^{(s-{\alpha })_+}   &&\textrm{if }s\neq {\alpha },\\
			&\varepsilon \log \theta_{n}   &&\textrm{if }s= {\alpha },
		\end{aligned}
		\nonumber  
		\\[1mm]
		& \big\|\big(({I}-\mathcal{S}_{\theta_{n}})U_{n},\, ({I}-\mathcal{S}_{\theta_{n}})u_{n} ,\, ({I}-\mathcal{S}_{\theta_{n}})\psi_{n}, ({I}-\mathcal{S}_{\theta_{n}})\Psi_{n} \big)\big\|_{s}
		\lesssim \varepsilon \theta_{n}^{s-{\alpha }},
		\nonumber 
	\end{align}
	for ${n}=0,1,\ldots,{N}-1$ and  $s=6,7,\ldots,\widetilde{\alpha}$,
	where $\|\cdot \|_{s}$ is defined by \eqref{norm:ring1}. Moreover,
	\begin{align}
		&
		\big\|\big(\mathcal{S}_{\theta_{n}}U_{n} ,\, \mathcal{S}_{\theta_{n}}u_{n} ,\,\mathcal{S}_{\theta_{n}}\psi_{n}, \mathcal{S}_{\theta_{n}}\Psi_{n} \big)\big\|_{s}
		\lesssim  \left\{\begin{aligned}
			&\varepsilon \theta_{n}^{(s-{\alpha })_+}  &&\textrm{if }s\neq {\alpha },\\
			&\varepsilon \log \theta_{n}  &&\textrm{if }s= {\alpha },
		\end{aligned}\right.  \nonumber  
	\end{align}
	for ${n}=0,1,\ldots,{N}-1$ and $s=6,7,\ldots,\widetilde{\alpha}+6$.
\end{lemma}

To prove hypothesis ($\mathbf{H}_{{N}}$) from ($\mathbf{H}_{{N}-1}$),
we need to estimate the error terms $E_{k}^{\pm}$ and $\widetilde{E}_{k}$ defined by \eqref{E.E.tilde} and \eqref{e.e.tilde}.
For this purpose, we introduce the second derivatives $\mathbb{L}_{\pm}''$ and $\mathbb{B}''$ of  $\mathbb{L}_{\pm}$ and $\mathbb{B}$ as
\begin{align*}
	\mathbb{L}_+''\big(\mathring{U},\mathring{\Phi}\big)
	\big(\big(U,\Psi\big),\, \big(\widetilde{U},\widetilde{\Psi}\big)\big)
	:=\;& \frac{\d}{\d \theta}
	\mathbb{L}_+'\big(\mathring{U}+\theta \widetilde{U},
	\mathring{\Phi}+\theta \widetilde{\Psi}\big)
	\big(U,\Psi\big)\Big|_{\theta=0},\\[0.5mm]
	\mathbb{L}_-''\big(\mathring{u},\mathring{\nu},\mathring{\Phi}\big)
	\big(\big(u,\nu, \Psi\big),\, \big(\tilde{u},\tilde{\nu},\widetilde{\Psi}\big)\big)
	:= \;&\frac{\d}{\d \theta}
	\mathbb{L}_-'\big(\mathring{u}+\theta \tilde{u},\mathring{\nu}+\theta \tilde{\nu},
	\mathring{\Phi}+\theta \widetilde{\Psi}\big)
	\big(u,\nu,\Psi\big)\Big|_{\theta=0},	
	\\[0.5mm]
	\mathbb{B}''
	\big(\big(U, u, \psi\big),\, \big(\widetilde{U},\tilde{u},\tilde{\psi}\big)\big)
	:= \;& \frac{\d}{\d \theta}
	\mathbb{B}'\big(\mathring{U}+\theta \widetilde{U},
	\mathring{u}+\theta \tilde{u},
	\mathring{\varphi}+\theta \tilde{\psi}\big)\big(U,u,\psi\big)\Big|_{\theta=0},
\end{align*}
where $\mathbb{L}_{\pm}'$ and $\mathbb{B}'$ are given in \eqref{Alinhac1}, \eqref{Alinhac2}, and \eqref{B'bb2D:def}.
In particular, we have
\begin{align}
	\label{B'':def}
	\mathbb{B}''
	\big(\big(U, u, \psi\big),\, \big(\widetilde{U},\tilde{u},\tilde{\psi}\big)\big)
	=\begin{bmatrix}
		\tilde{v}_2\p_2 \psi+ v_2\p_2 \tilde{\psi}\\[1mm]
		-\tilde{h}\cdot h+\tilde{e}e\\[1mm]
		\epsilon\p_t \tilde{\psi}h_2+\epsilon\tilde{h}_2\p_t \psi
	\end{bmatrix}.	
\end{align}
Then we can deduce the following result by
applying the Moser-type calculus and embedding inequalities. 

\begin{proposition}  \label{pro:tame2}
	Let $T>0$ and $s\in\mathbb{N}$ with $s\geq 6$.
	Assume that $(U,u,\Psi)$ belongs to $H_*^{s+2}(\Omega_{T})\times H^{s+2}(\Omega_{T})\times H^{s+2}(\Omega_T)$ and satisfies 	
	$$\|(U,\Psi)\|_{H_*^{6}(\Omega_{T})} +\|(u,\Psi)\|_{H^{4}(\Omega_{T})} \leq M$$ for some ${M}>0$.
	If $(U_i, u_i,\Psi_i)\in H_*^{s+2}(\Omega_{T})\times H^{s+2}(\Omega_{T})\times H^{s+2}(\Omega_T)$ and $ \psi_i\in  H^{s+2}(\Sigma_T)$ for $i=1,2$, then
	\begin{align*}
		&\big\|\mathbb{L}_+''\big(\widebar{U}+U ,x_1+{\Psi} \big) \big((U_1,\Psi_1),(U_2,\Psi_2) \big)\big\|_{H_*^s(\Omega_{T})}
		\\[1mm]
		& 	\hspace*{1em}	 \lesssim_{{M}} 		\sum_{i \neq j} \|(U_i,\Psi_i)\|_{H_*^{6}(\Omega_{T})}   \|(U_j,\Psi_j)\|_{H_*^{s+2}(\Omega_{T})}
		\\
		&\hspace*{3em}	+ 		\|(U_1,\Psi_1)\|_{H_*^{6}(\Omega_{T})} \|(U_2,\Psi_2)\|_{H_*^{6}(\Omega_{T})}		\| (U ,\Psi )\|_{H_*^{s+2}(\Omega_{T})}
		,
\\[2mm]
		&
		\big\|\mathbb{L}_-''\big(\bar{u}+u,\bar{v}+\nu, x_1+\Psi\big) \big((u_1,\nu_1, \Psi_1),(u_2,\nu_2,\Psi_2) \big)\big\|_{H^{s+1}(\Omega_{T})}
		\\[1mm]
		&\hspace*{1em} \lesssim_{{M}}
		\sum_{i\neq j} \|(u_i,\nu_i,\Psi_i)\|_{H^{4}(\Omega_{T})}   \left( \|(u_j,\Psi_j)\|_{H^{s+2}(\Omega_{T})}+\|\nu_j \|_{H^{s+1}(\Omega_{T})}		\right)
		\\
		&\hspace*{3em}	+	
				\|(u_1,\nu_1,\Psi_1)\|_{H^{4}(\Omega_{T})} \|(u_2,\nu_2,\Psi_2)\|_{H^{4}(\Omega_{T})} 			 \| (u,v, {\Psi})\|_{H^{s+2}(\Omega_{T})}
		,
\\[2mm]
		&		\left\| \mathbb{B} '' \big((U_1,u_1,\psi_1),\, (U_2,u_2,\psi_2) \big)  \right\|_{H^{s+1}(\Sigma_{T})}	\\[1mm]
		&	
		\hspace*{1em} \lesssim_{{M}}
		\sum_{i\neq j} \|(U_i,u_i,\mathrm{D}_{\rm tan}\psi_i)\|_{H^{2}(\Sigma_T)}		
		\|(U_j,u_j,\mathrm{D}_{\rm tan}\psi_j)\|_{H^{s+1}(\Sigma_T)},
	\end{align*}
	where $\widebar{U}\in\mathbb{R}^6$ and $\bar{u}\in\mathbb{R}^3$ are the constant vectors given in \eqref{initial:H2}, and $\nu$, $\nu_1$, $\nu_2$ are extended to the region $\Omega_{T}$ as in \eqref{nu:def1}.
\end{proposition}

Let us first apply Proposition \ref{pro:tame2} to estimate the quadratic error terms $e'_{n\pm }$ and $\tilde{e}_{n}'$ defined in \eqref{deco1}--\eqref{deco3}; see \cite[Lemma 4.6]{TW21MR4201624} or \cite[Lemma 4.5]{TW22cMR4506578} for the proof.
\begin{lemma}\label{lem:quad}
	If $\theta_0\geq 1$ is sufficiently large and $\varepsilon>0$ is small enough, then
	\begin{align}\nonumber 
		\|e_{n+}'\|_{H_*^s(\Omega_{T})}+\|e_{n-}'\|_{H^{s+1}(\Omega_{T})}+ \| \tilde{e}_{n}' \|_{H^{s+1}(\Sigma_{T})}
		\lesssim \varepsilon^2 \theta_{n}^{\varsigma_1(s)-1}\varDelta_{n}
	\end{align}
	for $n=0,1,\ldots,{N}-1$  and $s=6,7,\ldots,\widetilde{\alpha}-2 $,
	where $$\varsigma_1(s):=\max\{(s+2-{\alpha })_++10-2{\alpha },\,s+6-2{\alpha } \}.$$
\end{lemma}

Applying Proposition \ref{pro:tame2}, we can obtain the following estimate for the first substitution error terms $e_{n\pm}''$ and $\tilde{e}_{n}''$ defined in \eqref{deco1}--\eqref{deco3}. The proof is similar to that for \cite[Lemma 4.7]{TW21MR4201624} or \cite[Lemma 4.6]{TW22cMR4506578}, and  we omit it for brevity.
\begin{lemma} \label{lem:1st}
	If $\theta_0\geq 1$ is sufficiently large and $\varepsilon>0$ is small enough, then
	\begin{align}  \nonumber 
		\|e_{n+}''\|_{H_*^s(\Omega_{T})}+\|e_{n-}''\|_{H^{s+1}(\Omega_{T})}+ \|\tilde{e}_{n}''\|_{H^{s+1}(\Sigma_{T})}
		\lesssim \varepsilon^2 \theta_n^{\varsigma_2(s)-1}\varDelta_{n}
	\end{align}
	for $n=0,1,\ldots,{N}-1$  and $s=6,7,\ldots,\widetilde{\alpha}-2 $,	
	where $$\varsigma_2(s):=\max\{(s+2-{\alpha })_++12-2{\alpha },\, s+8-2{\alpha } \}.$$
\end{lemma}

For applying Theorem \ref{thm:main2} to the problem \eqref{effective.NM},
we shall construct a smooth vector-valued function $(U_{n+1/2},v_{n+1/2}, \psi_{n+1/2})$ vanishing in the past, such that $(U^a+U_{n+1/2},u^a+u_{n+1/2},\varphi^a+\psi_{n+1/2})$ satisfies the constraints \eqref{bas1}--\eqref{bas6}.
\begin{proposition}
	\label{pro:modified}
	Let ${\alpha }\geq 8$.
	If $\theta_0\geq 1$ is sufficiently  large and $\varepsilon,T>0$ are small enough, then
	there exist functions $U_{n+1/2}$, $u_{n+1/2}$, and $\psi_{n+1/2}$,
	such that $(U^a+U_{n+1/2},u^a+u_{n+1/2}, \varphi^a+\psi_{n+1/2})$ satisfies \eqref{bas1}--\eqref{bas6}
	and \eqref{bas10}
	with $(U^a, u^a, \varphi^a)$ being the approximate solution constructed in Lemma \textrm{\ref{lem:app}}.
	Moreover,
	\begin{alignat}{3}
		\label{MS.id1}	& \psi_{n+1/2}=\mathcal{S}_{\theta_n}\psi_{{n}},\quad \Psi_{n+1/2}=\chi(x_1)\psi_{n+1/2},&& \quad v_{2, n+1/2}=\mathcal{S}_{\theta_n} v_{2, n},\\[1mm]
		\label{MS.est1}	&\|\mathcal{S}_{\theta_n}\Psi_n-\Psi_{n+1/2}\|_{H^s(\Omega_{T})}\lesssim \varepsilon \theta_n^{s-{\alpha }}
		&& \textrm{for } s=6,7,\ldots,\widetilde{\alpha}+6,  \\
		\label{MS.est2}	&  \|\mathcal{S}_{\theta_n}U_n-U_{n+1/2}\|_{H_*^s(\Omega_{T})}
		\lesssim \varepsilon \theta_n^{s+2-{\alpha }}
		&&\textrm{for } s=6,7,\ldots,\widetilde{\alpha}+4,  \\
		\label{MS.est3}	 & \|\mathcal{S}_{\theta_n}u_n-u_{n+1/2}\|_{H^{s}(\Omega_{T})}
		\lesssim \varepsilon \theta_n^{s+1-{\alpha }}
		&&   \textrm{for } s=6,7,\ldots,\widetilde{\alpha}+5.
	\end{alignat}
\end{proposition}
\begin{proof}
	The proof is divided into six steps.
	
	\vspace*{2mm}
	\noindent {\it Step 1}.\quad	
	Since the modified state will be chosen to vanish in the past,
	it follows from Lemma \ref{lem:app} that
	$(U^a+U_{n+1/2},u^a+u_{n+1/2},\varphi^a+\psi_{n+1/2})$ will satisfy \eqref{bas1}--\eqref{bas2} and \eqref{bas10} for $T>0$ sufficiently small.
	Hence it suffices to focus on the constraints \eqref{bas3}--\eqref{bas6}.
	Let us define $q_{{n}+1/2}:=\mathcal{S}_{\theta_n} q_n$ and $S_{{n}+1/2}:=\mathcal{S}_{\theta_n} S_n$.
	
	\vspace*{2mm}
	\noindent {\it Step 2}.\quad	
	Set $\psi_{{n}+1/2}:=\mathcal{S}_{\theta_n}\psi_{{n}}$
	and $\Psi_{{n}+1/2}:=\chi(x_1)\psi_{n+1/2}$.
	If $6\leq s\leq \widetilde{\alpha}$, then it follows from Proposition \ref{pro:smooth} and Lemma \ref{lem:triangle} that
	\begin{align}
		\nonumber \| \mathcal{S}_{\theta_n}\Psi_n-\Psi_{n+1/2}\|_{H^s(\Omega_{T})}
		&  \lesssim \|(\mathcal{S}_{\theta_{{n}}}-{I})(\chi( x_1) \psi_{{n}}) +\chi( x_1)({I}-\mathcal{S}_{\theta_{{n}}})\psi_{{n}}\|_{H^s(\Omega_{T})}\\
		&  \lesssim \theta_{{n}}^{s-\widetilde{\alpha}} \left(
		\|\chi( x_1) \psi_{{n}}\|_{ H^{\widetilde{\alpha}}(\Omega_{T})}
		+\|\psi_n\|_{H^{\widetilde{\alpha}}(\Sigma_T)}\right)
		\lesssim \varepsilon \theta_{{n}}^{s-{\alpha }}.
		\nonumber
	\end{align}
	If $\widetilde{\alpha}<s\leq \widetilde{\alpha}+6$, then Lemma \ref{lem:triangle} implies
	\begin{align}
		\| \mathcal{S}_{\theta_n}\Psi_n-\Psi_{n+1/2}\|_{H^s(\Omega_{T})}\nonumber
		\lesssim
		\| \mathcal{S}_{\theta_n}\Psi_n\|_{H^s(\Omega_{T})}
		+\|\mathcal{S}_{\theta_n}\psi_{{n}} \|_{H^{s}(\Sigma_T)}
		\lesssim \varepsilon \theta_{{n}}^{s-{\alpha }},
		\nonumber
	\end{align}
	which completes the proof of \eqref{MS.est1}.

	\vspace*{2mm}
	\noindent {\it Step 3}.\quad	
	Set  $v_{2,n+1/2}:=\mathcal{S}_{\theta_n} v_{2, n}$.
	Let us define
	\begin{gather*}
		v_{1,n+1/2}:=
		\mathcal{S}_{\theta_{{n}}}v_{1,n}+
		\mathfrak{R}_T\left(\hat{w}_{n}- \left.\left(\mathcal{S}_{\theta_{{n}}}v_{1,n} \right)\right|_{x_1=0} \right),
	\end{gather*}
	where $\mathfrak{R}_T$ is the lifting operator $H^{s-1}(\Sigma_{T})\to H_*^{s}(\Omega_{T})$ (cf.~\cite[Theorem 1]{OSY94MR1289186}) and
	\begin{gather*}
		\hat{w}_{n}:=\p_t \psi_{n+1/2} +
		(v_2^{a}+v_{2,n+1/2}) \big|_{x_1=0}\p_2 \psi_{n+1/2} +v_{2,n+1/2}\big|_{x_1=0}
		\p_2\varphi^a  .
	\end{gather*}
	Then it follows from \eqref{app1d} that the third condition  in \eqref{bas3} holds for $(v^a+v_{n+1/2},\varphi^a+\psi_{n+1/2})$.
	As proved in \cite[Proposition 4.8]{TW21MR4201624}, we have
	\begin{align}
		\|v_{n+1/2}-\mathcal{S}_{\theta_{{n}}}v_{n}\|_{H_*^{s}(\Omega_{T})}
		\lesssim  \varepsilon \theta_n^{s-{\alpha }}
		\quad \textrm{for }\ s =6,7,\ldots, \widetilde{\alpha}+{6}.
		\label{MS.e2}
	\end{align}
	
	\vspace*{1mm}
	\noindent {\it Step 4}.\quad
	Noting that $\Psi_{n+1/2} $ and ${v}_{n+1/2}$ have been specified,
	we take ${H}_{n+1/2} $ as the unique solution of the linear transport equation
	\begin{align} \label{MS.eq1}
		\mathbb{H}(H^{a}+{H}_{n+1/2},v^{a}+{v}_{n+1/2},\Phi^{a+}+\Psi_{n+1/2})=0,
	\end{align}
	supplemented with zero initial data ${H}_{n+1/2} |_{t=0}=0,$
	where the operator $\mathbb{H}$ is defined by \eqref{Hbb:def}.
	Since $(v^a+v_{n+1/2},\varphi^a+\psi_{n+1/2})$ satisfies the third condition  in \eqref{bas3},
	no boundary condition is needed for the solvability of $H_{n+1/2}$.
	
	Since ${H}_{n+1/2}$ and $\psi_{{n}+1/2}$ vanish at $t=0$,
	it follows from Lemma \ref{lem:app}
	that $(H^a+H_{n+1/2},\varphi^a+\psi_{n+1/2})$ satisfies the second constraint in \eqref{bas3} and the first equations in \eqref{bas4}--\eqref{bas5}.
	
	As shown in \cite[Proposition 4.8]{TW21MR4201624}, we can use the Moser-type calculus inequalities for anisotropic Sobolev spaces  and hypothesis $({\bf H}_{N-1})$ to derive
	\begin{align}
		\big\| {H}_{n+1/2} -\mathcal{S}_{\theta_n}H_{n} \big\|_{H_*^{s}(\Omega_{T})}
		\lesssim \varepsilon \theta_n^{s-{\alpha }+2}
		\quad \textrm{for }\ s =6,7,\ldots, \widetilde{\alpha}+{4}.	
		\nonumber
	\end{align}

	\vspace*{2mm}
	\noindent {\it Step 5}.\quad	
	we take ${u}_{n+1/2} $ as the unique solution of Maxwell's equations
	\begin{align}    \label{MS:u.eq}
		\mathbb{L}_-(u^{a}+{u}_{n+1/2},\nu^{a}+{\nu}_{n+1/2},\Phi^{a-}+\Psi_{n+1/2})=0
		\quad\textrm{in }\Omega_{T},
	\end{align}
	supplemented with  zero initial data ${u}_{n+1/2} |_{t=0}=0$ and
	the boundary condition
	\begin{align}
		{e_{n+1/2}+\epsilon \p_t(\varphi^a+\psi_{n+1/2})h_{2, n+1/2}+\epsilon\p_t \psi_{n+1/2}h_2^a
			=0\quad\textrm{on } \Sigma_T, }
		\label{MS:u.bdy}		
	\end{align}
	where $\nu_{n+1/2}(t,x):=\chi(x_1)v_{n+1/2}(t,0,x_2)$.
	The solvability of the problem \eqref{MS:u.eq}--\eqref{MS:u.bdy} can be obtained as in Appendix \ref{app:D}.
	Since ${u}_{n+1/2}$ and $\psi_{{n}+1/2}$ vanish at the initial time, as in Appendix \ref{app:C},
	we use Lemma \ref{lem:app} to deduce
	that
	the first constraint in \eqref{bas3}, the second equations in \eqref{bas4}--\eqref{bas5}, and the condition \eqref{bas6}
	are satisfied for $(h^a+h_{n+1/2},\varphi^a+\psi_{n+1/2})$.

	\vspace*{2mm}
	\noindent {\it Step 6}.\quad	
	It remains to prove \eqref{MS.est3}.
	For this purpose,  we use \eqref{MS:u.eq} to decompose
	\begin{align}
		\mathbb{L}_-({u}_{n+1/2}-\mathcal{S}_{\theta_{n}} u_n,\nu^{a}+{\nu}_{n+1/2},\Phi^{a-}+\Psi_{n+1/2})
		=  \mathcal{T}_1+\mathcal{T}_2+\mathcal{T}_3+\mathcal{T}_4, \label{MS.e3}	
	\end{align}
	where $\mathcal{T}_1:= -\mathcal{S}_{\theta_{n}}\mathcal{L}_-( u_{n-1},{\nu}_{n-1},\Psi_{n-1})$,
	\begin{align*}
		\mathcal{T}_2:= \; &
		\mathcal{S}_{\theta_{n}}\mathcal{L}_-( u_{n-1},{\nu}_{n-1},\Psi_{n-1})
		-\mathcal{S}_{\theta_{n}}\mathcal{L}_-(u_n, {\nu}_{n},\Psi_{n}),\\
		\mathcal{T}_3:=\; &\mathcal{S}_{\theta_{n}}\mathcal{L}_-(u_n, {\nu}_{n},\Psi_{n})
		- \mathcal{L}_-(\mathcal{S}_{\theta_{n}} u_n, \mathcal{S}_{\theta_{n}} {\nu}_{n},\Psi_{n+1/2})		,\\
		\mathcal{T}_4:=\; &\mathcal{L}_-(\mathcal{S}_{\theta_{n}} u_n, \mathcal{S}_{\theta_{n}} {\nu}_{n},\Psi_{n+1/2})- \mathcal{L}_-(\mathcal{S}_{\theta_{n}} u_n,{\nu}_{n+1/2},\Psi_{n+1/2}).
	\end{align*}
	From Propositions \ref{pro:smooth} and hypothesis $(\mathbf{H}_{{N}-1})$, we have
	\begin{align}
		\|\mathcal{T}_1\|_{H^{s}(\Omega_{T})}
		\lesssim  \theta_{{n}}^{s-6}
		\|\mathcal{L}_-( u_{n-1},{\nu}_{n-1},\Psi_{n-1})\|_{H^{6}(\Omega_{T})}
		\lesssim \varepsilon \theta_n^{s-{\alpha }}\quad
		\textrm{for  }s\geq 6.
		\label{MS.p5}
	\end{align}	
	Since $\nu_{n+1/2}:=\chi(x_1)v_{n+1/2}|_{\Sigma}$ and $\nu_{n}(t,x):=\chi(x_1)v_{n}|_{\Sigma}$,
	we obtain
	\begin{align} \nonumber
		\big\|\nu_{n+1/2}-\mathcal{S}_{\theta_{{n}}}\nu_{n}\big\|_{H^{s}(\Omega_{T})}
		\lesssim 	\;&  \big\|v_{n+1/2}-\mathcal{S}_{\theta_{{n}}}v_{n} \big\|_{H^s(\Sigma_{T})}\\
		&	+\big\|\chi(x_1)(\mathcal{S}_{\theta_{{n}}}v_n)|_{\Sigma}
		-\mathcal{S}_{\theta_{{n}}}\big(\chi(x_1)v_n|_{\Sigma}\big) \big\|_{H^s(\Omega_{T})}. \label{MS:es9}
	\end{align}
	Using the trace theorem and \eqref{MS.e2} leads to
	\begin{align*}
		\|v_{n+1/2}-\mathcal{S}_{\theta_{{n}}}v_{n} \|_{H^s(\Sigma_{T})}\lesssim
		\|v_{n+1/2}-\mathcal{S}_{\theta_{{n}}}v_{n} \|_{H_*^{s+1}(\Omega_{T})}
		\lesssim  \varepsilon \theta_n^{s+1-{\alpha }}
	\end{align*}
	for $s =6,7,\ldots, \widetilde{\alpha}+{5}$.
	The last term in \eqref{MS:es9} can be estimated by
	virtue of similar calculations as in Step 2 and the trace theorem. Then
	we have
	\begin{align}
		\|\nu_{n+1/2}-\mathcal{S}_{\theta_{{n}}}\nu_{n}\|_{H^{s}(\Omega_{T})}
		\lesssim \varepsilon \theta_{{n}}^{s+1-{\alpha }}
		\quad \textrm{for }\ s =6,7,\ldots, \widetilde{\alpha}+{5}.		
		\label{MS:es8}
	\end{align}
	Using the Moser-type calculus inequalities for standard Sobolev spaces, 
	we can deduce from definition \eqref{NM0}, Lemma \ref{lem:app}, Propositions \ref{pro:smooth}, hypothesis $(\mathbf{H}_{{N}-1})$, and estimates \eqref{MS.est1}, \eqref{MS:es8} that
	\begin{align}
		\|(\mathcal{T}_2,\, \mathcal{T}_3, \, \mathcal{T}_4)\|_{H^{s}(\Omega_{T})}
		\lesssim \varepsilon \theta_n^{s+1-{\alpha }}
		\quad \textrm{for }\ s =6,7,\ldots, \widetilde{\alpha}+{5}.	
		\nonumber
	\end{align}	
	Combine the last estimate with \eqref{MS.e3} and \eqref{MS.p5} to deduce \eqref{MS.est3}, which completes the proof of this lemma.
\end{proof}

Thanks to Proposition \ref{pro:modified},
we can obtain the following estimate for the second substitution error terms $e_{n\pm}'''$ and $\tilde{e}_{n}'''$ given in \eqref{deco1}--\eqref{deco3}.
\begin{lemma} \label{lem:2nd}
	Let ${\alpha }\geq 8$.
	If $\theta_0\geq 1$ is sufficiently  large and $\varepsilon,T>0$ are small enough, then for $n=0,1, \ldots,N-1$  and $s=6,7,\ldots,\widetilde{\alpha}-2 $,
	\begin{align} \nonumber 
		&\|e_{n+}'''\|_{H_*^s(\Omega_{T})}+\|e_{n-}'''\|_{H^{s+1}(\Omega_{T})}\lesssim \varepsilon^2 \theta_n^{\varsigma_3(s)-1}\varDelta_n,\\
		& \|\tilde{e}_{n}''' \|_{ H^{s+1/2}(\Sigma_{T})}\lesssim \varepsilon^2 \theta_n^{s+7-2\alpha}\varDelta_n,
		\nonumber 
	\end{align}
	where	 $\varsigma_3(s):=\max\{(s+2-{\alpha })_++14-2{\alpha },\, s+10-2{\alpha } \}.$
\end{lemma}
\begin{proof}
	By virtue of \eqref{B'':def} and \eqref{MS.id1}, we have
	\begin{align}
		\nonumber \tilde{e}_{n}'''=
		\;&
		\mathbb{B}''
		\big((\delta U_n ,\delta u_n,\delta \psi_n),\, (\mathcal{S}_{\theta_n} U_n-U_{n+1/2},\mathcal{S}_{\theta_n} u_n-u_{n+1/2},0)\big)  \\[1mm]
		=\;&  	\begin{bmatrix}
			(\mathcal{S}_{\theta_n} v_{2,n}-v_{2,n+1/2})\p_2 \delta\psi_n\\[1mm]
			-\delta h_n\cdot (\mathcal{S}_{\theta_n} h_{n}-h_{n+1/2})
			+\delta e_n  (\mathcal{S}_{\theta_n} e_{n}-e_{n+1/2})\\[1mm]
			\epsilon (\mathcal{S}_{\theta_n} h_{2,n}-h_{2,n+1/2})\p_t \delta\psi_n
		\end{bmatrix} .\nonumber
	\end{align}
	Applying the Moser-type calculus inequalities, 
	we use hypothesis $(\mathbf{H}_{{N}-1})$, \eqref{MS.est3}, and \eqref{MS:es8}
	to get
	\begin{align*}
		\|\tilde{e}_{n}''' \|_{ H^{s+1/2}(\Sigma_{T})}\lesssim \;&
		\|\delta\psi_n\|_{ H^{6}(\Sigma_{T})}\|(\mathcal{S}_{\theta_n} v_{n}-v_{n+1/2},\, \mathcal{S}_{\theta_n} h_{n}-h_{n+1/2} )\|_{ H^{s+1/2}(\Sigma_{T})}\\
		& +		\|\delta\psi_n\|_{ H^{s+3/2}(\Sigma_{T})}\|(\mathcal{S}_{\theta_n} v_{n}-v_{n+1/2},\, \mathcal{S}_{\theta_n} h_{n}-h_{n+1/2} )\|_{L^{\infty}(\Sigma_{T})}\\
		& +\|\delta u_n\|_{ H^{s+1}(\Omega_{T})}
		\|\mathcal{S}_{\theta_n} u_{n}-u_{n+1/2}\|_{H^{6}(\Omega_{T})}\\
		&+\|\delta u_n\|_{ H^{6}(\Omega_{T})}
		\|\mathcal{S}_{\theta_n} u_{n}-u_{n+1/2}\|_{H^{s+1}(\Omega_{T})}
		\lesssim \varepsilon^2 \theta_n^{s+7-2\alpha}\varDelta_n.
	\end{align*}
	Employing similar arguments yields the estimate for $e_{n\pm }'''$,  which completes the proof of the lemma.
\end{proof}

Using \eqref{Alinhac1} and \eqref{Alinhac2}, we rewrite the last error terms $e_{{n}\pm}^{*}$ in \eqref{deco1} and \eqref{deco2} as
\begin{align*}
	&e_{{n}+}^{*}=\frac{\delta \Psi_{{n}}
	}{\p_1(\Phi^{a+}+\Psi_{n+1/2} )}\p_1\mathbb{L}_{+}(U^{a}+V_{n+1/2},\Phi^{a+}+\Psi_{n+1/2}) ,\\[1mm]
	&e_{{n}-}^{*}=\frac{\delta \Psi_{{n}}
	}{\p_1(\Phi^{a-}+\Psi_{n+1/2} )} \p_1\mathbb{L}_{-}(h^{a}+h_{n+1/2},\nu^{a}+\nu_{n+1/2},\Phi^{a-}+\Psi_{n+1/2}).
\end{align*}
As in \cite[Lemma 4.10]{TW21MR4201624},
we can get the next lemma by use of the embedding and Moser-type calculus inequalities,
hypothesis ($\mathbf{H}_{N-1}$), and Proposition \ref{pro:modified}.
\begin{lemma}\label{lem:last}
	Let $\widetilde{\alpha}\geq {\alpha }+2$ and ${\alpha }\geq 8$.
	If $\theta_0\geq 1$ is sufficiently  large and $\varepsilon,T>0$ are small enough, then	
	\begin{align} \nonumber 
		\|e_{n+}^*\|_{H_*^s(\Omega_{T})}+\|e_{n-}^*\|_{H^{s+1}(\Omega_{T})}
		\lesssim \varepsilon^2 \theta_n^{\varsigma_4 (s)-1}\varDelta_n
	\end{align}
	for $n=0,1, \ldots,N-1$  and $s=6,7,\ldots,\widetilde{\alpha}-2 $,
	where $$ \varsigma_4(s):=\max\{(s-{\alpha })_++18-2{\alpha },\, s+12-2{\alpha }\}.$$
\end{lemma}

We utilize Lemmas \ref{lem:quad}--\ref{lem:last} to obtain the following estimates for the accumulated error terms $E_{N}^{\pm}$ and $\widetilde{E}_N $ defined by \eqref{E.E.tilde} and \eqref{e.e.tilde}.
\begin{lemma} \label{lem:sum2}
	Let $\widetilde{\alpha}={\alpha }+4$ and ${\alpha }\geq 13$.
	If $\theta_0\geq 1$ is sufficiently  large and $\varepsilon,T>0$ are small enough, then	
	\begin{align} \label{es.sum2}
		\|E_{N}^+\|_{H_*^{{\alpha }+2}(\Omega_{T})}+\|E_{N}^-\|_{H^{{\alpha }+3}(\Omega_{T})}
		\lesssim \varepsilon^2 \theta_{{N}},\quad
		\| \widetilde{E}_{{N}} \|_{H^{{\alpha }+5/2}(\Sigma_{T})}\lesssim \varepsilon^2.
	\end{align}
\end{lemma}
\begin{proof}
	Lemmas \ref{lem:quad}--\ref{lem:last} imply that the error terms $e_{n}^{\pm}$, $\tilde{e}_n$ defined by \eqref{e.e.tilde} satisfy
	\begin{align} \label{es.sum1a}
		\left\{
		\begin{aligned}
			&	\|e_{n}^+ \|_{H_*^s(\Omega_{T})}+\|e_{n}^- \|_{H^{s+1}(\Omega_{T})}
			\lesssim \varepsilon^2 \theta_n^{\varsigma_4(s)-1}\varDelta_n,\\[1mm]
			&	\|\tilde{e}_n\|_{H^{s+1/2}(\Sigma_T)}
			\lesssim \varepsilon^2 \theta_n^{\varsigma_2(s)-1}\varDelta_n,
		\end{aligned}
		\right.
	\end{align}	
	for $n=0,1,\ldots,{N}-1 $ and $s=6,7,\ldots,\widetilde{\alpha}-2$.
	Note that $\varsigma_4({\alpha }+2)\leq 1$ for $\alpha\geq 13$.
	Let ${\alpha }+2= \widetilde{\alpha}-2$ and ${\alpha }\geq 13$. Then
	\begin{align*}
		&\|E_{{N}}^+\|_{H_*^{\alpha+2}(\Omega_{T})}+\|E_{N}^-\|_{H^{{\alpha }+3}(\Omega_{T})}
		\lesssim \sum_{n=0}^{{N}-1}\left( \|e_{n}^+ \|_{H_*^{\alpha+2}(\Omega_{T})}
		+\|e_{n}^-\|_{H^{{\alpha }+3}(\Omega_{T})}\right)
		\lesssim \varepsilon^2\theta_{{N}},
	\end{align*}
	due to  $\varDelta_n=\theta_{n+1}-\theta_n$ and $\theta_{{n}}:=\sqrt{\theta^2_0+{n}}$.
	Since $\varsigma_2({\alpha }+2)=10-{\alpha }\leq -3$, we get
	\begin{align*}
		\|\widetilde{E}_{{N}}\|_{H^{{\alpha }+5/2}(\Sigma_T)}
		\lesssim \sum_{n=0}^{{N}-1} \|\tilde{e}_{n} \|_{H^{{\alpha }+5/2}(\Sigma_T)}
		\lesssim \varepsilon^2,
	\end{align*}
	which finishes the proof of this lemma.
\end{proof}

\subsection{Proof of Theorem \ref{thm:main3}}
To get hypothesis $(\mathbf{H}_{{N}})$ from $(\mathbf{H}_{{N}-1})$,
we deduce the estimates for the source terms $f_{{N}}^{\pm}$ and $g_{{N}} $ given in \eqref{source} as follows.
\begin{lemma} \label{lem:source}
	Let $\widetilde{\alpha}={\alpha }+4$ and ${\alpha }\geq 13$.
	If $\theta_0\geq 1$ is large enough and $\varepsilon,T>0$ are sufficiently small, then
	\begin{align}
		&
		\|f_{{N}}^+\|_{H_*^s(\Omega_{T})}+\|f_{{N}}^-\|_{H^{s+1}(\Omega_{T})}
		\lesssim \varDelta_{{N}}\left(\theta_{{N}}^{s-{\alpha }-2} (\|f^a\|_{H_*^{\alpha+1}(\Omega_{T})}
		+\varepsilon^2 )  +\varepsilon^2\theta_{{N}}^{\varsigma_4(s)-1}\right),
		\nonumber \\[0.5mm] 
		&
		\| g_{{N}} \|_{H^{s+3/2}(\Sigma_{T}) }
		\lesssim  \varepsilon^2 \varDelta_{{N}}\left(\theta_{{N}}^{s-{\alpha }-2}
		+\theta_{{N}}^{\varsigma_2(s+1)-1}\right)\nonumber
		\qquad\qquad \textrm{for }s=6,7,\ldots,\widetilde{\alpha}+1.
	\end{align}
\end{lemma}
\begin{proof}
	It follows from Proposition \ref{pro:smooth} and estimates \eqref{es.sum1a} that
	\begin{align*}
		&	\big\|\mathcal{S}_{\theta_{{N}}} e_{N-1}^+ \big\|_{H_*^s(\Omega_{T})}
		\lesssim 	\big\|  e_{N-1}^+ \big\|_{H_*^s(\Omega_{T})}
		\lesssim \varepsilon^2 \theta_N^{\varsigma_4(s)-1}\varDelta_N
		\quad&& \textrm{if } 6\leq s\leq \widetilde{\alpha}-2,\\
		&	\big\|\mathcal{S}_{\theta_{{N}}} e_{N-1}^+ \big\|_{H_*^s(\Omega_{T})}
		\lesssim \theta_{{N}}^{s-\widetilde{\alpha}+2}	
		\big\|  e_{N-1}^+ \big\|_{H_*^{\widetilde{\alpha}-2}(\Omega_{T})}
		\lesssim \varepsilon^2 \theta_N^{\varsigma_4(s)-1}\varDelta_N
		\quad &&\textrm{if }  s\geq \widetilde{\alpha}-2.
	\end{align*}
	Similarly, for all $s\geq 6$, we have
	\begin{align*}
		\big\|\mathcal{S}_{\theta_{{N}}} e_{N-1}^- \big\|_{H_*^{s+1}(\Omega_{T})}
		\lesssim \varepsilon^2 \theta_N^{\varsigma_4(s)-1}\varDelta_N,
		\quad 	
		\big\|\mathcal{S}_{\theta_{{N}}} \tilde{e}_{N-1}  \big\|_{H_*^{s+3/2}(\Sigma_{T})}
		\lesssim \varepsilon^2 \theta_N^{\varsigma_2(s+1)-1}\varDelta_N.
	\end{align*}
	By virtue of  identities \eqref{source}, Proposition \ref{pro:smooth}, and Lemma \ref{lem:sum2}, we infer
	\begin{align*}
		\big\|f_N^+\big\|_{H_*^{s}(\Omega_{T})}
		&\lesssim
		\big\|(\mathcal{S}_{\theta_{{N}}}-\mathcal{S}_{\theta_{{N}-1}})f^a
		-(\mathcal{S}_{\theta_{{N}}}-\mathcal{S}_{\theta_{{N}-1}})E_{{N}-1}^+
		-\mathcal{S}_{\theta_{{N}}}e_{{N}-1}^+\big\|_{H_*^{s}(\Omega_{T})} \\
		&
		\lesssim \varDelta_{{N}}
		\theta_{{N}}^{s-{\alpha }-2}\big(\big\|f^a\big\|_{H_*^{\alpha +1}(\Omega_{T})}
		+\theta_{{N}}^{-1}\big\|E_{{N}-1}^+\big\|_{H_*^{\alpha +2}(\Omega_{T})} \big)
		+\varepsilon^2 \theta_N^{\varsigma_4(s)-1}\varDelta_N,\\[1mm]
		\big\|f_{{N}}^-\big\|_{H^{s+1}(\Omega_{T})}
		&\lesssim \big\|(\mathcal{S}_{\theta_{{N}}}-\mathcal{S}_{\theta_{{N}-1}})E_{{N}-1}^-\big\|_{H^{s+1}(\Omega_{T})}
		+\big\|\mathcal{S}_{\theta_{{N}}}e_{{N}-1}^-\big\|_{H^{s+1}(\Omega_{T})}\\
		&\lesssim
		\varDelta_{{N}}
		\theta_{{N}}^{s-{\alpha }-3}\big\|E_{{N}-1}^-\big\|_{H^{\alpha+3}(\Omega_{T})}
		+\varepsilon^2 \theta_N^{\varsigma_4(s)-1}\varDelta_N,\\[1mm]
		\big\|g_{{N}} \big\|_{H^{s+3/2}(\Sigma_{T})}
		&\lesssim \big\|(\mathcal{S}_{\theta_{{N}}}-\mathcal{S}_{\theta_{{N}-1}})\widetilde{E}_{{N}-1}\big\|_{H^{s+3/2}(\Sigma_{T})}
		+\big\|\mathcal{S}_{\theta_{{N}}}\tilde{e}_{{N}-1}\big\|_{H^{s+3/2}(\Sigma_{T})}\\
		&\lesssim
		\varDelta_{{N}}
		\theta_{{N}}^{s-{\alpha }-2}\big\|\widetilde{E}_{{N}-1}\big\|_{H^{\alpha+5/2}(\Sigma_{T})}
		+\varepsilon^2 \theta_N^{\varsigma_2(s+1)-1}\varDelta_N,
	\end{align*}
	which together with Lemma \ref{lem:sum2} completes the proof of this lemma.
\end{proof}

Applying the tame estimate \eqref{tame.es} to the problem \eqref{effective.NM},
we can use Proposition \ref{pro:modified} and Lemma \ref{lem:source} to derive the estimate (a) in hypothesis $(\mathbf{H}_{{N}})$.
For brevity we omit the proof, which is similar to that for \cite[Lemma 16]{CS08MR2423311} and
\cite[Lemma 15]{T09ARMAMR2481071}.
\begin{lemma}\  \label{lem:Hl1}
	Let $\widetilde{\alpha}={\alpha }+4$ and ${\alpha }\geq 13$.
	If $\varepsilon,T,\|f^a\|_{H_*^{\alpha+1}(\Omega_{T})}/\varepsilon$ are sufficiently small and $\theta_0\geq 1$ is large enough, then
	\begin{align} \nonumber  
		\|\delta U_N\|_{H_*^s(\Omega_{T})}+\|\delta u_N\|_{H^s(\Omega_{T})} +\| \delta\psi_N\|_{H^{s+1/2}(\Sigma_T)}
		\leq \varepsilon \theta_{{N}}^{s-{\alpha }-1}\varDelta_{{N}}
	\end{align}
	for $s=6,7,\ldots,\widetilde{\alpha} $.
\end{lemma}

By virtue of identities \eqref{LB:conver}, estimates \eqref{es.sum2}--\eqref{es.sum1a}, and Proposition \ref{pro:smooth},
we can obtain the other estimates in hypothesis $(\mathbf{H}_{{N}})$ as stated in the next lemma, whose proof is similar to that of  \cite[Lemmas 17--18]{CS08MR2423311} and
\cite[Lemma 16]{T09ARMAMR2481071}.

\begin{lemma}\ \label{lem:Hl2}
	Let $\widetilde{\alpha}={\alpha }+4$ and ${\alpha }\geq 13$.
	If $\varepsilon,T, \|f^a\|_{H_*^{\alpha+1}(\Omega_{T})}/{\varepsilon}$ are sufficiently small and $\theta_0\geq 1$ is large enough, then
	\begin{alignat}{3}\label{Hl.b}
		\max\left\{\|\mathcal{L}_+( U_N,  \Psi_N)-f^a\|_{H_*^s(\Omega_{T})},\,
		\|\mathcal{L}_-( u_N,  \nu_N, \Psi_N)\|_{H^{s+1}(\Omega_{T})}\right\}
		\leq 2 \varepsilon \theta_{{N}}^{s-{\alpha }-1}
	\end{alignat}
	for  $s=6,7,\ldots,\widetilde{\alpha}-2$, and
	\begin{align}
		\label{Hl.c}
		\|\mathcal{B}( U_N, u_N, \psi_N) \|_{H^{s+3/2}(\Sigma_{T})}
		\leq  \varepsilon \theta_{{N}}^{s-{\alpha }-1}
		\quad  \textrm{for } s=7,8,\ldots,{\alpha }.
	\end{align}
\end{lemma}

From Lemmas \ref{lem:Hl1} and \ref{lem:Hl2}, we have derived hypothesis $(\mathbf{H}_{{N}})$,
provided that ${\alpha }=\widetilde{\alpha}-4\geq 13$ holds,
$\theta_0 \geq 1$ is large enough,
and $\varepsilon,$ $T$, $\|f^a\|_{H_*^{\alpha+1}(\Omega_{T})}/\varepsilon$ are sufficiently small.
Fixing the constants ${\alpha }\geq 13$, $\widetilde{\alpha}={\alpha }+4$,
$\varepsilon>0$, and $\theta_0\geq1$,
we can establish $(\mathbf{H}_{0})$ as in \cite[Lemma 17]{T09ARMAMR2481071}.

\begin{lemma}\ \label{lem:H0}
	If time $T>0$ is small enough, then hypothesis $(\mathbf{H}_0)$ is satisfied.
\end{lemma}

We are ready to conclude the proof of Theorem \ref{thm:main3}, that is, the existence of 2D relativistic plasma--vacuum interfaces.

\vspace{2mm}
\noindent  {\bf Proof of Theorem {\rm\ref{thm:main3}}.}\quad
Suppose that the initial data $(U_0,u_0, \varphi_0)$ satisfy all the assumptions in Theorem {\rm\ref{thm:main3}}.
Set $\widetilde{\alpha}=m+{5}$ and ${\alpha }=\widetilde{\alpha}-4\geq 13$.
Then the initial data $(U_0,u_0,\varphi_0)$ are compatible up to order $
\widetilde{\alpha}+{5}$.
Taking $\varepsilon, T>0$ small enough and $\theta_0\geq 1$ suitably large, we can
verify all the requirements of Lemmas \ref{lem:Hl1}--\ref{lem:H0} from  \eqref{app2a} and \eqref{f^a:est}.
Therefore, we can find some time $T>0$, such that hypothesis $(\mathbf{H}_{{N}})$ is satisfied for all ${N}\in\mathbb{N}$, leading to
\begin{align*}
	\sum_{N=0}^{\infty}\left(
	\|\delta U_N\|_{H_*^s(\Omega_{T})}+\|\delta u_N\|_{H^{s}(\Omega_{T})}
	+\| \delta\psi_N \|_{H^{s+1/2}(\Sigma_T)}
	\right)
	\lesssim \sum_{N=0}^{\infty}\theta_N^{s-{\alpha }-2} <\infty
\end{align*}
for $s=6,7,\ldots, {\alpha }-1.$
So the sequence $(U_{N},u_N,\psi_N)$ converges to some limit $(U, u, \psi)$ in $H_*^{{\alpha }-1}(\Omega_{T})\times H^{{\alpha }-1}(\Omega_{T}) \times H^{{\alpha }-1/2}({\Sigma_T})$.
Passing to the limit in \eqref{Hl.b} and \eqref{Hl.c} for $s=\alpha-1=m $ implies \eqref{P.new}, and
hence $(\widehat{U}, \hat{u}, \varphi)=(U^a+U, u^a+u, \varphi^a+\psi)$
is a solution of the nonlinear problem \eqref{NP1} on $[0,T]$.
The uniqueness of solutions to the problem \eqref{NP1} can be achieved by a standard argument (see, {\it e.g.}, \cite[\S 13]{ST14MR3151094} or \cite[\S 3.5]{TW22aMR4444136}).
The proof is complete.
\qed

	
	\appendix
	\titleformat{\section}{\large\bfseries\centering}{Appendix \thesection}{1em}{}

\vspace*{2mm}

\section[RMHD and symmetrization]{RMHD and symmetrization} \label{app:A}
The equations of ideal relativistic magnetohydrodynamics (RMHD) for perfectly conducting plasmas in Minkowski spacetime $\mathbb{R}^{1+d}$ read as (see \cite[\S\S 30--34]{Lichnerowicz67}, \cite[Chapter 22]{GKP19}, and \cite[\S 2.4]{A90zbMATH05040442}):
\begin{align}\label{RMHD1}
	\nabla_{\alpha} (\rho u^{\alpha})=0,\quad
	\nabla_{\alpha}  T^{\alpha\beta}=0,\quad
	\nabla_{\alpha}(u^{\alpha}b^{\beta}-u^{\beta}b^{\alpha})=0,
\end{align}
where $\nabla_{\alpha}$ is the covariant derivative with respect to the Minkowski metric
\begin{align*}
	(g_{\alpha \beta}):=
	\left\{
	\begin{aligned}
		&\mathrm{diag}\,(-1,1,1)\quad &&\textrm{if } d=2,\\
		&\mathrm{diag}\,(-1,1,1,1)\quad &&\textrm{if }d=3.
	\end{aligned}
	\right.
\end{align*}
We denote by $u^{\alpha}$ the components of the $(d+1)$-velocity,
by $b^{\alpha}$ the components of the magnetic field $(d+1)$-vector with respect to the plasma velocity,
and by $T^{\alpha\beta}$ the total energy--momentum tensor,
so that
\begin{align}\label{tensorT}
	g_{\alpha \beta} u^{\alpha} u^{\beta}=-1,\ \
	g_{\alpha \beta} u^{\alpha} b^{\beta}=0,\ \
	T^{\alpha\beta}=(\rho \mathfrak{f}+|b|^2)u^{\alpha}u^{\beta}+\epsilon^2 q g^{\alpha\beta} -b^{\alpha}b^{\beta},
\end{align}
where $\mathfrak{f}=1+\epsilon^{2}(\mathfrak{e}+ p/\rho)$ is the index of the fluid,
$\epsilon^{-1}$ is the speed of light,
and $q=p+\frac{1}{2} \epsilon^{-2}|b|^2$ is the total pressure,
with $|b|^2:=g_{\alpha \beta} b^{\alpha} b^{\beta}$.
The thermodynamic variables $\rho$ (particle number density),
$\mathfrak{e}$ (internal energy), and $p$ (pressure) should satisfy
the Gibbs relation \eqref{Gibbs}.

As in \cite{TW21MR4201624,FT13MR3044369},
we can derive the RMHD equations \eqref{RMHD0}--\eqref{H.inv1} from \eqref{RMHD1}--\eqref{tensorT} by introducing the coordinate velocity $v:=(v_1,\ldots,v_d)$,
the magnetic field $H:=(H_1,\ldots,H_d)$, and
the spacetime coordinates $(t,x)=(\epsilon x^0, x^1,\ldots,x^d)$
with
\begin{align*}
	v_i:=\epsilon^{-1}\varGamma^{-1} u^{i},\quad
	\varGamma:=u^0>0,\quad
	H_i:=\epsilon^{-1}(u^0 b^i-u^i b^0).
\end{align*}

For smooth solutions, the Gibbs relation \eqref{Gibbs} implies $u^{\alpha} \nabla_{\alpha} S=0$ (cf.~\cite[\S 34]{Lichnerowicz67}), and hence $(\p_t +v\cdot \nabla) S=0$.
{The symmetrization of the RMHD system \eqref{RMHD0} was obtained in \cite{RS81MR0628566,A90zbMATH05040442,AP87MR0877994} via Godunov's procedure, in \cite{FT13MR3044369} using the Lorentz transform, and in \cite{TW21MR4201624} by direct verification.}
More precisely,
under the physical assumption \eqref{cs:def},
for smooth solutions $V:=(p,\varGamma v,H,S)$ with $\rho_*<\rho<\rho^*$,
the RMHD equations \eqref{RMHD0} can be reduced to the following {\it symmetric hyperbolic system} (cf.~\cite[(5.10)]{TW21MR4201624}):
\begin{align}\label{RMHD.vec}
	B_0(V)\p_t V+
	B_j(V)\p_j V=0,
\end{align}
where
\begin{align} \label{B.def}
	B_0(V):=
	\begin{bmatrix}
		\;\dfrac{\varGamma}{\rho a^2}\;&\; \epsilon^2 v^{\mathsf{T}}\;& \;0\;&\; 0\;\\[2.5mm]
		\epsilon^2 v &\mathcal{A}_0&O_d& 0\\[1.5mm]
		0& O_d&\mathcal{M}_0&0\\[1mm]
		\w0\;&\;\w0\;&\;\w0\;&\;\w1
	\end{bmatrix},\quad
	B_j(V)=\begin{bmatrix}
		\;\dfrac{\varGamma v_j}{\rho a^2}\;
		& \;\bm{e}_j^{\mathsf{T}}\;& \;0\;& \;0\;\\[2.5mm]
		\bm{e}_j& \mathcal{A}_j & \mathcal{N}_j^{\mathsf{T}} & 0\\[1.5mm]
		0& \mathcal{N}_j & v_j\mathcal{M}_0 &0\\[1mm]
		\w0\;&\;\w0\;&\;\w0\;&\;\w{v_j}
	\end{bmatrix},
\end{align}
with
\begin{align} \label{M0.cal:def}
	\mathcal{M}_0:= \;&  \varGamma^{-1} \big(I_d+\epsilon^2\varGamma^2 v\otimes v\big) ,\quad a:=\sqrt{p_{\rho}(\rho,S)},	\\[0.5mm]
\nonumber	\mathcal{A}_0:=\;&
	\big(\rho \mathfrak{f}\varGamma+{\epsilon^2 \varGamma^{-1}|H|^2} \big)\big(I_d-\epsilon^2v\otimes v \big)
	- \epsilon^2 \varGamma^{-1} \big(|b|^2  v\otimes v+ H\otimes H\big)\\[0.5mm]
	&
	+ \epsilon^4 \varGamma^{-1}   (v\cdot H )\big(H\otimes v+v\otimes H\big),
	\nonumber
\end{align}	
and for $j=1,\ldots,d$,
\begin{align*}	
	\mathcal{N}_j:= \;&	\big\{\varGamma^{-2} {H}+\epsilon^2(v\cdot H )v \big\}\otimes	\big(\bm{e}_j- \epsilon^2 v_j v \big)- {\varGamma^{-2}} {H_j} I_d,\\[1mm]	
	\mathcal{A}_j:= \;&
	v_j\left\{
	\big(\rho \mathfrak{f}\varGamma+{\epsilon^2\varGamma^{-1} |H|^2}\big)\big(I_d-\epsilon^2v\otimes v \big)
	+\epsilon^2{\varGamma^{-1}} \big(|b|^2 v\otimes v  -H\otimes H\big)\right\}\\[0.5mm]
	&
	+{\epsilon^2 {\varGamma^{-1}} H_j}
	\big\{\varGamma^{-2} {\big(H\otimes v+v\otimes H\big)}
	-2(v\cdot H )\big(I_d-\epsilon^2 v\otimes v\big) \big\}\\[0.5mm]
	&+{\varGamma^{-1}} \left\{\epsilon^2 (v\cdot H )H-|b|^2v\right\}\otimes  {\bm{e}_j}
	+{\varGamma^{-1}} {\bm{e}_j} \otimes \left\{\epsilon^2 (v\cdot H )H-|b|^2v\right\}.
\end{align*}
Here we denote $\bm{e}_j:=(\delta_{j1},\ldots,\delta_{jd})$, where $\delta_{ij}$ is the Kronecker delta.
{The positivity of $B_0(V)$ is guaranteed by the physical condition \eqref{cs:def}.}
Changing the primary unknowns to $U:=(q,v,H,S) $, we obtain the equations \eqref{RMHD:vec} with coefficient matrices $A_j^+(U)$ given by
\setlength{\arraycolsep}{4.5pt}
\begin{align}
	\label{Ai.def.r}
	A_{j}^+(U):= {\bm{J}}^{-1}B_0(V)^{-1} B_{j}(V) {\bm{J}}\qquad
	\textrm{for }j=1,\ldots,d,
\end{align}
where
\begin{align}\label{J.bm:def}
	{\bm{J}}:=\frac{\p V}{\p U}
	=\begin{bmatrix}
		1 &  \bm{a}^{\mathsf{T}} & -\bm{b}^{\mathsf{T}} &0\\
		0 &   \varGamma^2\mathcal{M}_0 & O_d &0\\
		0 & O_d & I_d &0\\
		\w{0} & \w{0} & \w{0}& \w{1}
	\end{bmatrix},
	\quad
	\begin{bmatrix}
		\bm{a}\\[1.5mm] \bm{b}
	\end{bmatrix}:=
	\begin{bmatrix}
		\epsilon^2 |H|^2 v -\epsilon^2(v\cdot H) H\\[1.5mm]
		\varGamma^{-2}{H}+\epsilon^2 (v\cdot H)v
	\end{bmatrix}.
\end{align}

Since system \eqref{RMHD.vec} is symmetric hyperbolic, we can choose the symmetrizer $S^+(U)$  for system \eqref{RMHD:vec} as
\begin{align}
	\label{S+U}
	S^+(U):={\bm{J}}^{\mathsf{T}} B_0(V) {\bm{J}},
\end{align}
where $B_0(V)$ and ${\bm{J}}$ are defined in \eqref{B.def} and \eqref{J.bm:def}, respectively.
From \eqref{Ai.def.r}--\eqref{S+U}, we find that $S^+(U)A_{j}^+(U)$ is identical to $A_{j}(U)$ defined in \cite[(5.12)]{TW21MR4201624} for $j=1,\ldots,d$. In view of the identity \cite[(5.13)]{TW21MR4201624},
we can use the first conditions in \eqref{BC2a} and \eqref{BC2b} to derive
\begin{align} 
S^+(U)A_{\rm b}^+
	=\begin{bmatrix}
		0 & \varGamma N^{\mathsf{T}} & 0   \\
		\varGamma N
		&  O_d  &0 \\
		0 & 0 & O_{d+1}
	\end{bmatrix} \qquad
	\textrm{ on } \Sigma(t),	
	\label{iden:App1}
\end{align}
where $A_{\rm b}^+$ is defined by \eqref{Ab:def}.

\vspace*{2mm}
	\section[Boundary conditions on the interface
	]{Boundary conditions on the interface
	} \label{app:B}
	

For the sake of completeness, we derive here the boundary conditions \eqref{BC1} for the relativistic plasma--vacuum interface problem. We focus on the spatial dimension $d=3$, since the 2D case can be treated similarly.

Introduce the plasma electric field
	\begin{align} \label{E:def}
		E:=-\epsilon v\times H.
	\end{align}
	In view of \eqref{Gamma:def} and \eqref{E:def}, we infer
	\begin{align}
		\nonumber
		&2q=2p+|H|^2-|E|^2,\qquad
		E\times H=\epsilon |H|^2v-\epsilon (v\cdot H)H,\\
		\nonumber &
		E\otimes E+\epsilon^2\big(|v|^2H\otimes H + |H|^2 v\otimes v  -(v\cdot H)(H\otimes v+v\otimes H)\big)=|E|^2I_3,
	\end{align}
	from which we can rewrite \eqref{RMHD0b}--\eqref{RMHD0d} as
	\begin{align}\label{RMHD2b}
		&\p_t\big(\rho \mathfrak{f} \varGamma^2-\epsilon^2p\big)+\tfrac{1}{2}\epsilon^2\p_t\big(|H|^2+|E|^2\big)
		+\nabla \cdot \big( \rho \mathfrak{f} \varGamma^2 v+\epsilon E\times H \big)=0,\\[0.5mm]
		\nonumber		&\p_t\big( \rho \mathfrak{f} \varGamma^2 v+\epsilon E\times H \big)
		+\nabla p+\tfrac{1}{2}\nabla\big(|H|^2+ |E|^2\big)
		\\
		&\hspace*{7.7em}\label{RMHD2c}
		+\nabla\cdot \big( \rho \mathfrak{f} \varGamma^2 v\otimes v\big)-
		\nabla\cdot \big( H\otimes H+E\otimes E \big)=0,\\[0.5mm]
		&\epsilon\p_t H+\nabla\times E=0.\label{RMHD2d}
	\end{align}
	By virtue of the Maxwell's equations \eqref{Maxwell1}--\eqref{e.inv1}, we obtain
	\begin{align} \label{Max2}
		&	\tfrac{1}{2}\epsilon^2\p_t\big(|h|^2+|e|^2\big)+\nabla\cdot \big(\epsilon e\times h\big)=0,\\[1mm]
		\label{Max3}
		& 	 \p_t(\epsilon  e\times h)+\tfrac{1}{2}\nabla\cdot\big(|h|^2+|e|^2\big)-\nabla\cdot (h\otimes h+e\otimes e)=0.
	\end{align}	

Let $\Sigma(t)\subset\mathbb{R}^3$ be a smooth surface with a well-defined unit normal $\bm{n}(t,x)$ and move with the normal speed $\mathcal{V}(t,x)$ at point $x\in \Sigma(t)$ and time $t\geq 0$.
Denote by $\Omega^+(t)$ and $\Omega^-(t)$ the spatial regions occupied by the plasma and vacuum at time $t$, respectively, so that $\Sigma(t)=\overline{\Omega^+(t)}\cap\overline{\Omega^-(t)}$.
Assume that the unit normal $\bm{n}$ points into $\Omega^+(t)$ and that the interface $\Sigma(t)$ moves with the plasma flow, so that $\mathcal{V}=v\cdot \bm{n}$ on $\Sigma(t)$.
Then from the conservation laws \eqref{H.inv1}, \eqref{Maxwell1}--\eqref{h.inv1}, and \eqref{RMHD2b}--\eqref{Max3}, we obtain the following boundary conditions on $\Sigma(t)$:
	\begin{align} \nonumber 
		\left\{
		\begin{aligned}
			& -\epsilon^2 p \mathcal{V} +\tfrac{1}{2}\epsilon^2 \big(|H|^2+ |E|^2-|h|^2-|e|^2  \big)  \mathcal{V}
			=		\epsilon \bm{n}\cdot \big( E\times H-e\times h\big),\\[0.5mm]
			&\epsilon\mathcal{V} \big(E\times H -e\times h\big)+\big(H\otimes H+E\otimes E-h\otimes h-e\otimes e \big)\bm{n}
			\\
			&\hspace*{12em}
			=p\bm{n} +\tfrac{1}{2}\big(|H|^2+|E|^2-|h|^2-|e|^2\big)\bm{n},\\[0.5mm]
			&\epsilon\mathcal{V} \big(H-h\big)=\bm{n}\times (E-e),\qquad\quad
			\bm{n}\cdot (H-h)=0,
		\end{aligned}
		\right.
	\end{align}
	where $E$ is defined by \eqref{E:def}.
Assuming that both the plasma and vacuum magnetic fields are tangential to the interface $\Sigma(t)$, we obtain the desired boundary conditions \eqref{BC1} for 3D relativistic plasma--vacuum interfaces.
	
 
 \vspace*{2mm}
\section[Proof of Proposition \ref{Pro:invol}]{Proof of Proposition \ref{Pro:invol}} \label{app:C}	
The vacuum equations \eqref{NP1b} can be rewritten as
\begin{align} \label{h:equ.non}
	\epsilon \p_t^{\Phi^-} h+\nabla^{\Phi^-}\times e+\epsilon  (\nabla^{\Phi^-}\cdot h) \nu =0,
\end{align}	
and
\begin{alignat}{3}
\label{e:equ.non.3D}	
		&\epsilon \p_t^{\Phi^-} e-\nabla^{\Phi^-}\times h+\epsilon  (\nabla^{\Phi^-}\cdot e) \nu =0\qquad&&\textrm{if }d=3,\\
\label{e:equ.non.2D}			
		&\epsilon \p_t^{\Phi^-} e-\nabla^{\Phi^-}\times h =0\qquad&&\textrm{if }d=2. 	
\end{alignat}

Applying the divergence operators $\nabla^{\Phi^+}\cdot$
and $\nabla^{\Phi^-}\cdot$ to \eqref{Hbb:def} and \eqref{h:equ.non}--\eqref{e:equ.non.3D}, respectively, we get
\begin{alignat*}{3}
	&\big(\p_t^{\Phi^+}+v\cdot\nabla^{\Phi^+}\big)\big(\nabla^{\Phi^+}\cdot H\big)+\big(\nabla^{\Phi^+}\cdot v\big)\big(\nabla^{\Phi^+}\cdot H\big)=0\quad&&\textrm{in \ }\Omega,\\
	&\big(\p_t^{\Phi^-}+\nu \cdot\nabla^{\Phi^-}\big)\big(\nabla^{\Phi^-}\cdot h\big)+\big(\nabla^{\Phi^-}\cdot \nu \big)\big(\nabla^{\Phi^-}\cdot h\big)=0\quad&&\textrm{in \ }\Omega,\\
	&\big(\p_t^{\Phi^-}+\nu \cdot\nabla^{\Phi^-}\big)\big(\nabla^{\Phi^-}\cdot e\big)+\big(\nabla^{\Phi^-}\cdot \nu \big)\big(\nabla^{\Phi^-}\cdot e\big)=0\quad&&\textrm{in \ }\Omega,
	\quad\textrm{if }d=3.
\end{alignat*}
Since $\p_t\varphi=v \cdot N=\nu \cdot N$ on $\Sigma$ (cf.~\eqref{NP1c} and \eqref{nu:def0}),
we can deduce \eqref{H.h.inv2} and \eqref{e.inv2} provided that they hold at the initial time.
	
Take the scalar product of \eqref{Hbb:def} with $N:=(1, -\nabla'\varphi)$ and use the first condition in
\eqref{NP1c} to infer
\begin{align*}
	(\p_t+v'\cdot\nabla')(H\cdot N)+ (\nabla' \cdot v')(H\cdot N)=0\quad \textrm{on \ } \Sigma.
\end{align*}
Consequently, the first equation in \eqref{H.h.cons2} is satisfied for $t>0$ as long as it holds initially.

Taking the scalar product of \eqref{h:equ.non} with $N$, we utilize
the conditions \eqref{NP1c}
to discover
\begin{align*}
	\p_t(h\cdot N)=0\quad\textrm{on \ }\Sigma.
\end{align*}
The last identity implies the second equation in \eqref{H.h.cons2} provided it holds at $t=0$.

\vspace*{2mm}
	\section[Solvability of problem \eqref{u.natural:eq}]{Solvability of problem \eqref{u.natural:eq}} \label{app:D}

Let us demonstrate the existence and regularity of solutions $u^{\natural}$ to problem \eqref{u.natural:eq} in the following five steps.

\vspace*{2mm}
\noindent {\it Step 1}.\quad We first derive an important boundary constraint for \eqref{u.natural:eq}.
Dropping the subscript $\natural$ for notational simplicity,
we rewrite the problem \eqref{u.natural:eq} as
\begin{subequations} \label{B1}
	\begin{alignat}{3}
		\label{B1a}
		&\epsilon \p_t^{\mathring{\Phi}^-}h+\nabla^{\mathring{\Phi}^-}\times e+\epsilon   (\nabla^{\mathring{\Phi}^-}\cdot h) \mathring{\nu}= (f_1^-,\, f_2^-,\, f_3^-)
		&\qquad &  \textrm{in } \Omega_T,\\[0.5mm]
		\label{B1b}
		& \epsilon \p_t^{\mathring{\Phi}^-}e-\nabla^{\mathring{\Phi}^-}\times h+\epsilon  ( \nabla^{\mathring{\Phi}^-}\cdot e )\mathring{\nu}=(f_4^-,\, f_5^-,\, f_6^-)
		&& \textrm{in }\Omega_T,\\[0.5mm]
		\label{B1c}
	& e_2+\p_2\mathring{\varphi}e_1-\epsilon \p_t\mathring{\varphi} h_3=g_3
&\quad  &\textrm{on } \Sigma_T,\\[0.5mm]
		\label{B1d}
& e_3+\p_3\mathring{\varphi}e_1+\epsilon \p_t\mathring{\varphi} h_2=g_4
&\quad  &\textrm{on } \Sigma_T,\\[0.5mm]
		\label{B1e}
		&
		(h, e) =0   &&\textrm{if } t<0.
	\end{alignat}
\end{subequations}
Take the scalar product of \eqref{B1a} with $\mathring{N}=(1,-\p_2\mathring{\Phi}^-,-\p_3\mathring{\Phi}^-)$
and use the third constraint in \eqref{bas3} to obtain
\begin{align}\nonumber
\epsilon \p_t\big(h\cdot \mathring{N}\big)
&	+\p_2\big(e_3 +\p_3\mathring{\varphi}e_1+\epsilon \p_t\mathring{\varphi} h_2\big)
\\[0.5mm]
&	-\p_3\big(e_2+\p_2\mathring{\varphi}e_1-\epsilon \p_t\mathring{\varphi} h_3 \big)
=f_1^--\p_2\mathring{\varphi}f_2^- -\p_3\mathring{\varphi}f_3^-
	\quad\textrm{on } \Sigma_{T}.
		\label{B:id1}
\end{align}
Combining the last identity with \eqref{B1c}--\eqref{B1e} implies the boundary constraint
\begin{align}
	\label{B:con1}
	h\cdot \mathring{N} =g_5
	\qquad\textrm{on } \Sigma_{T},
\end{align}
where
\begin{align} \label{g4:def}
	g_5(t,x'):=\epsilon^{-1}\int_{-\infty}^{t}\big(f_1^- -\p_2\mathring{\varphi}f_2^-
	-\p_3\mathring{\varphi}f_3^-
	 -\p_2 g_4+\p_3 g_3\big)(s,0, x')\mathrm{d}s.
\end{align}

\vspace*{2mm}
\noindent {\it Step 2}.\quad
Let us subtract from the solution of \eqref{B1} a regular function satisfying \eqref{B1c}--\eqref{B1e} and \eqref{B:con1}.
For this purpose, we set $u^{\sharp}:=(h_1^{\sharp}, 0, 0, 0, e_2^{\sharp},e_3^{\sharp})$ with
\begin{align} \label{u.sharp:def}
	h_1^{\sharp}:=\mathfrak{R}_{T} g_5,\quad
	e_2^{\sharp}:=\mathfrak{R}_{T} g_3,\quad
	e_3^{\sharp}:=\mathfrak{R}_{T} g_4,
\end{align}
where $\mathfrak{R}_T$ is a continuous lifting operator from $H^{{m+3/2}}(\Sigma_T)$ to $H^{{m+2}}(\Omega_{T})$.
Then $u^{\sharp}$ satisfies \eqref{B1c}--\eqref{B1e}, \eqref{B:con1}, and
\begin{align} \nonumber
	\|u^{\sharp}\|_{H^{m+2}(\Omega_{T})}^2\, &\lesssim_K
	\|(f_1^-,\, \p_2\mathring{\varphi}f_2^- ,\,\p_3\mathring{\varphi}f_3^- ,\, \p_2 g_4,\,\p_3 g_3,\, g_3,\, g_4)\|_{H^{m+3/2}(\Sigma_T)}^2\\
	& \lesssim_K
	\| (g_3,g_4)\|_{H^{m+5/2}(\Sigma_T)}^2
	+ \|f^-\|_{H^{m+2}(\Omega_{T})}^2
	+\mathring{\rm C}_{m+4}\|f^-\|_{H^{6}(\Omega_{T})}^2,
	\label{u.sharp:es}
\end{align}
where $\mathring{\rm C}_{m+4}$ is defined by \eqref{norm:ring}.
Introduce
\begin{align} \label{u.tilde:def}
	\tilde{u}=(\tilde{h},\tilde{e}):=u-u^{\sharp},
\end{align}
so that
\begin{align} \label{u.tilde:eq}
	\left\{
	\begin{aligned}
		&L_-(\mathring{\nu},\mathring{\Phi}^-) \tilde{u}=\tilde{f}^-
		:= f^- -L_-(\mathring{\nu},\mathring{\Phi}^-)  u^{\sharp}
		&\quad  &\textrm{in } \Omega_T,	\\
	& \tilde{e}_2+\p_2\mathring{\varphi}\tilde{e}_1-\epsilon \p_t\mathring{\varphi} \tilde{h}_3=0
&\quad  &\textrm{on } \Sigma_T,\\[0.5mm]
& \tilde{e}_3+\p_3\mathring{\varphi}\tilde{e}_1+\epsilon \p_t\mathring{\varphi} \tilde{h}_2=0
&\quad  &\textrm{on } \Sigma_T,\\[0.5mm]
		&\tilde{u}=0
		&& \textrm{if } t<0,
	\end{aligned}
	\right.
\end{align}
and
\begin{align} \label{B:con2}
	\tilde{h}\cdot \mathring{N}  =0
	\quad  \textrm{on } \Sigma_T.
\end{align}
Similar to \eqref{mu:def1.3D}--\eqref{mu:def2}, we introduce
\begin{align}	\label{mu.tilde:def}
	\tilde{\mu}:=\mathring{J}_2^{-1}\tilde{u}.
\end{align}
Then we can reduce \eqref{u.tilde:eq} to the following equivalent problem:
\begin{align}\label{B3a}
	\left\{
	\begin{aligned}
		&	{\mathsf{A}}_0^-\p_t   \mu+  {\mathsf{A}}_j^-\p_j   \mu  +{\mathsf{A}}_{4}^- \mu =\mathring{J}_2^{\mathsf{T}}S_{-}(\mathring{\nu})\tilde{f}^-
		\quad \textrm{in }\Omega_T,
		\\[0.5mm]
		&		\tilde\mu_5=\tilde\mu_6=0 \quad
		\textrm{on }\Sigma_T,\qquad\qquad
		\tilde\mu  =0\quad
		\textrm{if } t<0,
	\end{aligned}
	\right.
\end{align}
where the coefficient matrices $\mathsf{A}_i^{-}$ and the symmetrizer $S_-$
are defined by \eqref{A.sf:def} and  \eqref{S-3D}, respectively.
The second identity in \eqref{decom1.3D} implies that
the problem \eqref{B3a} is characteristic and has correct number of boundary conditions.
However, the boundary conditions in \eqref{B3a} are {\it not maximally nonnegative} \cite{R85MR0797053}.

\vspace*{2mm}
\noindent {\it Step 3}.\quad
To overcome the above difficulty, We make use of the boundary constraint \eqref{B:con2}.
More precisely, we set
\begin{align}
	\widetilde{w}:=
	\big( \tilde\mu_1  ,\, \tilde\mu_2-\epsilon \mathring{\nu}_3 \tilde{\mu_4},\,  \tilde\mu_3+\epsilon \mathring{\nu}_2 \tilde{\mu_4},  \,
	\tilde{\mu_4}  ,\,    \tilde\mu_5+\epsilon \mathring{\nu}_3 \tilde\mu_1,  \,\tilde\mu_6-\epsilon \mathring{\nu}_2 \tilde\mu_1
	\big)
	=\mathring{J}_3^{-1} \mathring{J}_2^{-1}\tilde{u},
	\label{w.tilde:def}	
\end{align}
where $\mathring{J}_3$ is defined by \eqref{J3.ring:def}.
Using \eqref{B:con2} leads to the reformulated problem
\begin{align}\label{B3c}
	\left\{
	\begin{aligned}
		& { \bm{A}}_0^-\p_t \widetilde{w}   +{ \bm{A}}_j^-\p_j\widetilde{w}     +{ \bm{A}}_4^-  \widetilde{w} =
		\mathring{J}_3^{\mathsf{T}}\mathring{J}_2^{\mathsf{T}}S_{-}(\mathring{\nu})\tilde{f}^-
		\quad \textrm{in }\Omega_T, \\[1mm]
		&\widetilde{w}_5=\widetilde{w}_6=0 \quad
		\textrm{on }\Sigma_T,\qquad\qquad
		\widetilde{w}  =0\quad
		\textrm{if } t<0,
	\end{aligned}
	\right.
\end{align}
where $	\bm{A}_i^-$, for $i=0,\ldots,4$, are defined by \eqref{A.bm:def}.
It follows from the second identity in \eqref{decom2.3D} that the boundary conditions in \eqref{B3c} are {\it maximally nonnegative}.
Therefore, we can apply the results of \cite{S95MMASMR1346663,S96aMR1401431,OSY95MR1346224} to obtain the existence and regularity of solutions $\widetilde{w} $ to the problem \eqref{B3c} in anisotropic Sobolev spaces.

\vspace*{2mm}
\noindent {\it Step 4}.\quad
To compensate the missing normal regularity of the characteristic variables $\widetilde{w}_1=\tilde{h}\cdot\mathring{N}$
and $\widetilde{w}_4=\tilde{e}\cdot\mathring{N}$,
we apply the
divergence operator $\nabla^{\mathring{\Phi}^-}\cdot$ to the interior equations in \eqref{u.tilde:eq} and obtain
\begin{alignat}{3}
\nonumber	
&	\epsilon \big(\p_t^{\mathring{\Phi}^-}+\mathring{\nu}\cdot \nabla^{\mathring{\Phi}^-}\big) \nabla^{\mathring{\Phi}^-}\cdot \tilde{h}+\epsilon \big(\nabla^{\mathring{\Phi}^-}\cdot \mathring{\nu}\big) \nabla^{\mathring{\Phi}^-}\cdot \tilde{h}=
	\p_1^{\mathring{\Phi}^-}\tilde{f}_1^-  +  	\p_2^{\mathring{\Phi}^-}\tilde{f}_2^-
	+  	\p_3^{\mathring{\Phi}^-}\tilde{f}_3^-,\\
&	\epsilon \big(\p_t^{\mathring{\Phi}^-}+\mathring{\nu}\cdot \nabla^{\mathring{\Phi}^-}\big) \nabla^{\mathring{\Phi}^-}\cdot \tilde{e}+\epsilon \big(\nabla^{\mathring{\Phi}^-}\cdot \mathring{\nu}\big) \nabla^{\mathring{\Phi}^-}\cdot \tilde{e}=
\p_1^{\mathring{\Phi}^-}\tilde{f}_4^-  +  	\p_2^{\mathring{\Phi}^-}\tilde{f}_5^-
+  	\p_3^{\mathring{\Phi}^-}\tilde{f}_6^-.
\nonumber
\end{alignat}
Since $\p_t\mathring{\varphi}=\mathring{\nu}\cdot \mathring{N}$ on $\Sigma$, we can use the standard energy method 
to infer
\begin{align} \nonumber
\|\mathrm{D}^{\alpha}(\nabla^{\mathring{\Phi}^-}\cdot \tilde{h},\, \nabla^{\mathring{\Phi}^-}\cdot \tilde{e})(t) \|_{L^2(\Omega)}^2
	\lesssim_K
	\|(\tilde{f}^-, \widetilde{w})\|_{H^{m+1}(\Omega_{T})}^2
	+ \mathring{\rm C}_{m+4}
	\|(\tilde{f}^-, \widetilde{w})\|_{W^{2,\infty}(\Omega_{T})}^2
\end{align}
for all $|\alpha|\leq m$.
Combining the last inequality with the anisotropic energy estimates for solutions $\widetilde{w}$ of \eqref{B3c} (see, for instance, \cite[Theorem A]{S95MMASMR1346663}) and using the estimate \eqref{u.sharp:es}, we can derive 
\begin{align} \nonumber
	\| \widetilde{w} \|_{H^{m+1}(\Omega_{T})}^2
	\lesssim_K
	\|(f^-, g)\|_{H^{m+2}(\Omega_{T})\times H^{m+5/2}(\Sigma_T)}^2
	+ \mathring{\rm C}_{m+4}
	\|(f^-, g)\|_{H^{7}(\Omega_{T})\times H^{8}(\Sigma_T)}^2.
\end{align}
The same estimate holds also for the solution $u$ of the problem \eqref{B1} due to \eqref{u.sharp:es}, \eqref{u.tilde:def}, and \eqref{w.tilde:def}.

\vspace*{2mm}
\noindent {\it Step 5}.\quad
Finally, it remains to show the equivalence of the problems \eqref{B3a} and \eqref{B3c}.
It follows from the identity \eqref{B:id1} and
the definitions \eqref{g4:def}, \eqref{u.sharp:def}, \eqref{u.tilde:def},  \eqref{mu.tilde:def} that both solutions of \eqref{B3a} and \eqref{B3c} should satisfy
\begin{align} \label{B:id2}
	\epsilon \p_t\tilde{\mu}_1+\p_2\tilde{\mu}_6-\p_3\tilde{\mu}_5
	=0
	\quad
	\textrm{on }\Sigma_{T}.
\end{align}
In view of \eqref{w.tilde:def},
solutions of \eqref{B3c} should satisfy $\tilde\mu_5+\epsilon \mathring{v}_3 \tilde\mu_1=\tilde\mu_6-\epsilon \mathring{v}_2\tilde\mu_1=0$ on $\Sigma$, which along with \eqref{B:id2} and Gr\"{o}nwall's inequality yields $\tilde{\mu}_1|_{\Sigma}=0$
and  $\tilde{\mu}_5|_{\Sigma}=\tilde{\mu}_6|_{\Sigma}=0$.
On the other hand,
plugging the boundary conditions in \eqref{B3a} into \eqref{B:id2} gives $\tilde{\mu}_1|_{\Sigma}=0$, and hence $\widetilde{w}_5|_{\Sigma}=\widetilde{w}_6|_{\Sigma}=0$ follows.


	\vspace*{4mm}
	

	{\footnotesize 
		  }

\end{document}